\documentclass[a4paper,11pt,fleqn]{article}
\usepackage[obeyspaces,hyphens,spaces]{url}
\usepackage[dvipsnames]{xcolor}
\usepackage{enumitem}
\usepackage{amsmath}
\usepackage{amssymb}
\usepackage{amsfonts}
\usepackage{theorem}
\usepackage{mathtools}
\usepackage{pifont} %for dingolist
\usepackage{euscript}
\usepackage{srcltx} %for forward references in xdvi
\usepackage{charter}
\usepackage{etoolbox}
\usepackage{mathrsfs} %pour mathscr
\usepackage{authblk}
\newcommand{\email}[1]{\href{mailto:#1}{\nolinkurl{#1}}}
\renewcommand\familydefault{\rmdefault}
\DeclareMathAlphabet{\mathrm}{OT1}{\familydefault}{m}{n}
\makeatletter
\def\operator@font{\mathgroup\symoperators\rm}
\makeatother
\allowdisplaybreaks %allows align to split
\setlist[enumerate]{itemsep=-2pt,topsep=1pt}
\setlist[itemize]{itemsep=-2pt,topsep=1pt}
\definecolor{labelkey}{rgb}{0,0.08,0.45}
\definecolor{refkey}{rgb}{0,0.6,0.0}
\definecolor{lblue}{HTML}{0455BF}
\definecolor{dgreen}{HTML}{02724A}
\definecolor{myellow}{HTML}{D97904}
\definecolor{dred}{HTML}{D90404}
\usepackage{upref}%upright number when using ref in theorems
\usepackage{hyperref}
\hypersetup{colorlinks=true,linktocpage=true,linkcolor=lblue,%
citecolor=dgreen,urlcolor=dred}
\renewcommand{\leq}{\ensuremath{\leqslant}}
\renewcommand{\geq}{\ensuremath{\geqslant}}

\newcommand{\minimize}[2]{\ensuremath{\underset{\substack{{#1}}}%
{\text{\rm minimize}}\;\;#2}}

\newcommand{\scal}[2]{{\langle{{#1}\mid{#2}}\rangle}}
\newcommand{\sscal}[2]{{\big\langle{{#1}\mid{#2}}\big\rangle}}

\newcommand{\menge}[2]{\big\{{#1}~|~{#2}\big\}} 
 
\DeclareFontFamily{U}{BOONDOX-calo}{\skewchar\font=45 }
\DeclareFontShape{U}{BOONDOX-calo}{m}{n}{
  <-> s*[1.05] BOONDOX-r-calo}{}
\DeclareFontShape{U}{BOONDOX-calo}{b}{n}{
  <-> s*[1.05] BOONDOX-b-calo}{}
\DeclareMathAlphabet{\mathcalboondox}{U}{BOONDOX-calo}{m}{n}
\SetMathAlphabet{\mathcalboondox}{bold}{U}{BOONDOX-calo}{b}{n}
\DeclareMathAlphabet{\mathbcalboondox}{U}{BOONDOX-calo}{b}{n}
\newcommand{\MM}{\ensuremath{\mathcalboondox{m}}}
\newcommand{\LL}{\ensuremath{\mathcalboondox{l}}}
\newcommand{\CC}{\ensuremath{\mathcalboondox{c}}}
\newcommand{\sad}{\ensuremath{\boldsymbol{\EuScript{S}}}}
\newcommand{\GGG}{\ensuremath{\boldsymbol{\mathcal{G}}}}
\newcommand{\HHH}{\ensuremath{\boldsymbol{\mathcal{H}}}}
\newcommand{\XXX}{\ensuremath{\boldsymbol{\mathsf{X}}}}

\newcommand{\ZZ}{\ensuremath{{\mathcal{Z}}}}
\newcommand{\HH}{\ensuremath{{\mathcal{H}}}}

\newcommand{\GG}{\ensuremath{{\mathcal{G}}}}
\newcommand{\KK}{\ensuremath{{\mathcal{K}}}}
\newcommand{\WW}{\ensuremath{{\mathcal{V}}}}
\newcommand{\Sum}{\ensuremath{\displaystyle\sum}}

\newcommand{\emp}{\ensuremath{\varnothing}}

\newcommand{\Id}{\ensuremath{\mathrm{Id}}}
\newcommand{\RR}{\ensuremath{\mathbb{R}}}
\newcommand{\RP}{\ensuremath{\left[0,{+}\infty\right[}}

\newcommand{\RPP}{\ensuremath{\left]0,{+}\infty\right[}}

\newcommand{\RX}{\ensuremath{\left]{-}\infty,{+}\infty\right]}}
\newcommand{\RXX}{\ensuremath{\left[{-}\infty,{+}\infty\right]}}

\newcommand{\NN}{\ensuremath{\mathbb{N}}}

\newcommand{\weakly}{\ensuremath{\rightharpoonup}}
\newcommand{\exi}{\ensuremath{\exists\,}}
\DeclareMathOperator{\card}{card}

\DeclareMathOperator{\zer}{zer}
\newcommand{\pinf}{\ensuremath{{{+}\infty}}}

\DeclareMathOperator{\epi}{epi}
\DeclareMathOperator{\dom}{dom}

\newcommand{\prox}{\ensuremath{\mathrm{prox}}}
\newcommand{\proj}{\ensuremath{\mathrm{proj}}}

\DeclareMathOperator{\gra}{gra}

\DeclareMathOperator{\sri}{sri}
\DeclareMathOperator{\reli}{ri}
\newcommand{\infconv}{\ensuremath{\mbox{\small$\,\square\,$}}}
\newcommand{\zeroun}{\ensuremath{\left]0,1\right[}}

\newcommand{\einfconv}{\ensuremath{\mbox{\footnotesize$\,\boxdot\,$}}}
\def\abstract{\noindent{\bfseries Abstract}. \ignorespaces}

\newtheorem{theorem}{Theorem}[section]
\newtheorem{lemma}[theorem]{Lemma}
\newtheorem{corollary}[theorem]{Corollary}
\newtheorem{proposition}[theorem]{Proposition}
\newtheorem{assumption}[theorem]{Assumption}
\theoremstyle{plain}{\theorembodyfont{\rmfamily}%
}

\theoremstyle{plain}{\theorembodyfont{\rmfamily}%
\newtheorem{example}[theorem]{Example}}
\theoremstyle{plain}{\theorembodyfont{\rmfamily}%
\newtheorem{remark}[theorem]{Remark}}
\theoremstyle{plain}{\theorembodyfont{\rmfamily}%
\newtheorem{algorithm}[theorem]{Algorithm}}
\theoremstyle{plain}{\theorembodyfont{\rmfamily}%
}
\theoremstyle{plain}{\theorembodyfont{\rmfamily}%
\newtheorem{definition}[theorem]{Definition}}
\theoremstyle{plain}{\theorembodyfont{\rmfamily}%
\newtheorem{problem}[theorem]{Problem}}
\theoremstyle{plain}{\theorembodyfont{\rmfamily}%
}

\numberwithin{equation}{section}
\newcommand*\mute{{\mkern 2mu\cdot\mkern 2mu}}

\begin{document}

\title{\sffamily\huge%
\vskip -12mm 
Multivariate Monotone Inclusions in Saddle Form\thanks{%
Contact author: P. L. Combettes.
Email: \email{plc@math.ncsu.edu}.
Phone: +1 919 515 2671.
This work was supported by the National Science Foundation
under grant CCF-1715671.
}
}

\author{Minh N. B\`ui and Patrick L. Combettes\\
\small North Carolina State University,
Department of Mathematics,
Raleigh, NC 27695-8205, USA\\
\small \email{mnbui@ncsu.edu}\: and \:\email{plc@math.ncsu.edu}
}

\date{~}
\maketitle

\vskip -18mm

\begin{abstract}
We propose a novel approach to monotone operator splitting based on
the notion of a saddle operator. Under investigation is a highly
structured multivariate monotone inclusion problem involving a mix
of set-valued, cocoercive, and Lipschitzian monotone operators, as
well as various monotonicity-preserving operations among them. This
model encompasses most formulations found in the literature. A
limitation of existing primal-dual algorithms is that they operate
in a product space that is too small to achieve full splitting of
our problem in the sense that each operator is used individually.
To circumvent this difficulty, we recast the problem as that of
finding a zero of a saddle operator that acts on a bigger space.
This leads to an algorithm of unprecedented flexibility, which
achieves full splitting, exploits the specific attributes of each
operator, is asynchronous, and requires to activate only blocks of
operators at each iteration, as opposed to activating all of them.
The latter feature is of critical importance in large-scale
problems. Weak convergence of the main algorithm is established,
as well as the strong convergence of a variant.
Various applications are discussed, and instantiations of
the proposed framework in the context of variational inequalities
and minimization problems are presented.
\end{abstract}


\begin{MSC}
47H05, 49M27, 47J20, 65K05, 90C25.
\end{MSC}

\section{Introduction}
\label{sec:1}

In 1979, several methods appeared to solve the basic problem
of finding a zero of the sum of two maximally monotone operators
in a real Hilbert space \cite{Lion79,Merc79,Pass79}. Over the past
forty years, increasingly complex inclusion problems and
solution techniques have been considered 
\cite{Bot13a,Siop11,Bri18a,Warp20,Siop13,MaPr18,%
Ecks17,John20,Tsen00} 
to address concrete problems in fields as diverse as 
game theory \cite{Atto08,Bric13,YiPa19},
evolution inclusions \cite{Sico10},
traffic equilibrium \cite{Sico10,Fuku96},
domain decomposition \cite{Atto16},
machine learning \cite{Bach12,Bric19},
image recovery \cite{Bane20,Bot14,Jmiv11,Hint06},
mean field games \cite{Bri18b},
convex programming \cite{Mono18,Liyu15},
statistics \cite{Ejst20,Bien20},
neural networks \cite{Neur20},
signal processing \cite{Smms05},
partial differential equations \cite{Ghou09},
tensor completion \cite{Mizo19},
and optimal transport \cite{Papa14}.
In our view, two challenging issues in the field of 
monotone operator splitting algorithms are the following:
\begin{itemize}
\item
A number of independent monotone inclusion models coexist with
various assumptions on the operators and different types of
operation among these operators. At the same time, as will be seen
in Section~\ref{sec:4}, they are not sufficiently general to cover
important applications. 
\item
Most algorithms do not allow asynchrony and impose that all the
operators be activated at each iteration. They can therefore not
handle efficiently modern large-scale problems. The only methods
that are asynchronous and block-iterative are limited to specific
scenarios \cite{MaPr18,Ecks17,John20} and they do not cover
inclusion models such as that of \cite{Siop13}. 
\end{itemize}
In an attempt to bring together and extend the application scope of
the wide variety of unrelated models that coexist in the
literature, we propose the following multivariate formulation which
involves a mix of set-valued, cocoercive, and Lipschitzian monotone
operators, as well as various monotonicity-preserving operations
among them.

\begin{problem}
\label{prob:1}
Let $(\HH_i)_{i\in I}$ and $(\GG_k)_{k\in K}$
be finite families of real Hilbert spaces with
Hilbert direct sums $\HHH=\bigoplus_{i\in I}\HH_i$ and
$\GGG=\bigoplus_{k\in K}\GG_k$.
Denote by $\boldsymbol{x}=(x_i)_{i\in I}$
a generic element in $\HHH$.
For every $i\in I$ and every $k\in K$,
let $s_i^*\in\HH_i$, let $r_k\in\GG_k$,
and suppose that the following are satisfied:
\begin{enumerate}[label={\rm[\alph*]}]
\item
\label{prob:1a}
$A_i\colon\HH_i\to 2^{\HH_i}$ is maximally monotone,
$C_i\colon\HH_i\to\HH_i$ is cocoercive with constant
$\alpha_i^{\CC}\in\RPP$,
$Q_i\colon\HH_i\to\HH_i$ is monotone and Lipschitzian
with constant $\alpha_i^{\LL}\in\RP$, and 
$R_i\colon\HHH\to\HH_i$.
\item
\label{prob:1b}
$B_k^{\MM}\colon\GG_k\to 2^{\GG_k}$ is maximally monotone,
$B_k^{\CC}\colon\GG_k\to\GG_k$ is cocoercive with constant
$\beta_k^{\CC}\in\RPP$, and $B_k^{\LL}\colon\GG_k\to\GG_k$
is monotone and Lipschitzian with constant $\beta_k^{\LL}\in\RP$.
\item
\label{prob:1c}
$D_k^{\MM}\colon\GG_k\to 2^{\GG_k}$ is maximally monotone,
$D_k^{\CC}\colon\GG_k\to\GG_k$ is cocoercive with constant
$\delta_k^{\CC}\in\RPP$, and $D_k^{\LL}\colon\GG_k\to\GG_k$
is monotone and Lipschitzian with constant $\delta_k^{\LL}\in\RP$.
\item
\label{prob:1d}
$L_{ki}\colon\HH_i\to\GG_k$ is linear and bounded.
\end{enumerate}
In addition, it is assumed that
\begin{enumerate}[resume,label={\rm[\alph*]}]
\item
\label{prob:1e}
$\boldsymbol{R}\colon\HHH\to\HHH\colon
\boldsymbol{x}\mapsto(R_i\boldsymbol{x})_{i\in I}$
is monotone and Lipschitzian with constant $\chi\in\RP$.
\end{enumerate}
The objective is to solve the primal problem
\begin{multline}
\label{e:1p}
\text{find}\;\:\overline{\boldsymbol{x}}\in\HHH
\;\:\text{such that}\;\:(\forall i\in I)\;\;
s_i^*\in A_i\overline{x}_i+C_i\overline{x}_i+Q_i\overline{x}_i
+R_i\overline{\boldsymbol{x}}\\
+\Sum_{k\in K}L_{ki}^*\Bigg(\Big(
\big(B_k^{\MM}+B_k^{\CC}+B_k^{\LL}\big)\infconv
\big(D_k^{\MM}+D_k^{\CC}+D_k^{\LL}\big)\Big)
\Bigg(\Sum_{j\in I}L_{kj}\overline{x}_j-r_k\Bigg)\Bigg)
\end{multline}
and the associated dual problem
\begin{multline}
\label{e:1d}
\text{find}\;\:\overline{\boldsymbol{v}}^*\in\GGG
\;\:\text{such that}\;\:
(\exi\boldsymbol{x}\in\HHH)(\forall i\in I)(\forall k\in K)\\
\begin{cases}
s^*_i-\Sum_{j\in K}L_{ji}^*\overline{v}_j^*\in
A_ix_i+C_ix_i+Q_ix_i+R_i\boldsymbol{x}\\
\overline{v}_k^*\in
\Big(\big(B_k^{\MM}+B_k^{\CC}+B_k^{\LL}\big)\infconv
\big(D_k^{\MM}+D_k^{\CC}+D_k^{\LL}\big)\Big)
\Bigg(\Sum_{j\in I}L_{kj}x_j-r_k\Bigg).
\end{cases}
\end{multline}
\end{problem}

Our highly structured model involves three basic monotonicity
preserving operations, namely addition, composition with linear
operators, and parallel sum. It extends the state-of-the-art
model of \cite{Siop13}, where the simpler form
\begin{equation}
\label{e:12p}
(\forall i\in I)\quad
s_i^*\in A_i\overline{x}_i+Q_i\overline{x}_i+\Sum_{k\in K}L_{ki}^*
\Bigg(\big(B_k^\MM\infconv D_k^\MM\big)
\Bigg(\Sum_{j\in I}L_{kj}\overline{x}_j-r_k\Bigg)\Bigg)
\end{equation}
of the system in \eqref{e:1p} was investigated;
see also \cite{Sico10,MaPr18} for special cases.
In an increasing number of applications, the sets $I$
and $K$ can be sizable. To handle such large-scale problems, it is
critical to implement block-iterative solution algorithms, in which
only subgroups of the operators involved in the problem need to be
activated at each iteration. In addition, it is desirable that the
algorithm be asynchronous in the sense that, at any iteration, it
has the ability to incorporate the result of calculations initiated
at earlier iterations. Such methods have been proposed for special
cases of Problem~\ref{prob:1}: first in \cite{MaPr18} for the
system 
\begin{equation}
\label{e:7p}
\text{find}\;\:\overline{\boldsymbol{x}}\in\HHH
\;\:\text{such that}\;\:(\forall i\in I)\;\;
s_i^*\in A_i\overline{x}_i
+\Sum_{k\in K}L_{ki}^*\Bigg(B_k^{\MM}
\Bigg(\Sum_{j\in I}L_{kj}\overline{x}_j-r_k\Bigg)\Bigg),
\end{equation}
then in \cite{Ecks17} for the inclusion
(we omit the subscript `$1$')
\begin{equation}
\label{e:8p}
\text{find}\;\:\overline{x}\in\HH\;\:\text{such that}\;\:
0\in\Sum_{k\in K}L_k^*\big(B_k^{\MM}
(L_k\overline{x})\big),
\end{equation}
and more recently in \cite{John20} for the inclusion
\begin{equation}
\label{e:9p}
\text{find}\;\:\overline{x}\in\HH\;\:\text{such that}\;\:
0\in A\overline{x}+Q\overline{x}
+\Sum_{k\in K}L_k^*\big((B_k^{\MM}+B_k^{\LL})
(L_k\overline{x})\big).
\end{equation}
It is clear that the formulations \eqref{e:7p} and \eqref{e:9p}
are not interdependent. Furthermore, as we shall see in
Section~\ref{sec:4}, many applications of interest are not covered
by either of them. From both a theoretical and a practical
viewpoint, it is therefore important to unify and extend these
approaches. To achieve this goal, we propose to design an
algorithm for solving the general Problem~\ref{prob:1} which
possesses simultaneously the following features:
\begin{dingautolist}{192}
\setlength\itemsep{1pt}
\item
\label{f:1}
It has the ability to process all the operators individually and
exploit their specific attributes, e.g., set-valuedness,
cocoercivity, Lipschitz continuity, and linearity.
\item
\label{f:2}
It is block-iterative in the sense that it does not need to
activate all the operators at each iteration, but only a
subgroup of them.
\item
\label{f:3}
It is asynchronous.
\item
\label{f:4}
Each set-valued monotone operator is scaled by its own, 
iteration-dependent, parameter.
\item
\label{f:5}
It does not require any knowledge of the norms of the linear
operators involved in the model.
\end{dingautolist}
Let us observe that the method of \cite{MaPr18} has features
\ref{f:1}--\ref{f:5}, but it is restricted to \eqref{e:7p}.
Likewise, the method of \cite{John20} has features
\ref{f:1}--\ref{f:5}, but it is restricted to \eqref{e:9p}.

Solving the intricate Problem~\ref{prob:1} with the requirement
\ref{f:1} does not seem possible with existing tools. The presence
of requirements \ref{f:2}--\ref{f:5} further complicates this task.
In particular, the Kuhn--Tucker approach initiated in \cite{Siop11}
--- and further developed in 
\cite{Siop14,Bot13a,Siop13,MaPr18,John20,John19} ---
relies on finding a zero of an operator acting on the primal-dual
space $\HHH\oplus\GGG$. However, in the context of
Problem~\ref{prob:1}, this primal-dual space is too small to
achieve full splitting in the sense that each operator is used
individually. To circumvent this difficulty, we propose a novel
splitting strategy that consists of recasting the problem
as that of finding a zero of a saddle operator acting on the
bigger space $\HHH\oplus\GGG\oplus\GGG\oplus\GGG$.
This is done in Section~\ref{sec:2}, where we define the saddle
form of Problem~\ref{prob:1}, study its properties, and propose
outer approximation principles to solve it. In Section~\ref{sec:3},
the main asynchronous block-iterative algorithm is presented and
we establish its weak convergence under mild conditions on
the frequency at which the operators are selected.
We also present a strongly convergent variant. The specializations
to variational inequalities and multivariate minimization are
discussed in Section~\ref{sec:4}, along with several applications.
Appendix~\ref{sec:A} contains auxiliary results.

\noindent
{\bfseries Notation.}
The notation used in this paper is standard and follows
\cite{Livre1}, to which one can refer for background and
complements on monotone operators and convex analysis. 
Let $\KK$ be a real Hilbert space.
The symbols $\scal{\cdot}{\cdot}$ and $\|\cdot\|$
denote the scalar product of $\KK$ and the associated norm,
respectively. The expressions $x_n\weakly x$ and $x_n\to x$
denote, respectively, the weak and the strong convergence of a
sequence $(x_n)_{n\in\NN}$ to $x$ in $\KK$, and $2^{\KK}$ denotes
the family of all subsets of $\KK$. Let $A\colon\KK\to 2^{\KK}$.
The graph of $A$ is $\gra A=\menge{(x,x^*)\in\KK\times\KK}{x^*\in
Ax}$, the set of zeros of $A$ is
$\zer A=\menge{x\in\KK}{0\in Ax}$, the inverse of $A$ is
$A^{-1}\colon\KK\to 2^{\KK}\colon x^*\mapsto
\menge{x\in\KK}{x^*\in Ax}$, and the resolvent of $A$ is
$J_A=(\Id+A)^{-1}$, where $\Id$ is the identity operator on $\KK$.
Further, $A$ is monotone if
\begin{equation}
\big(\forall(x,x^*)\in\gra A\big)
\big(\forall(y,y^*)\in\gra A\big)\quad
\scal{x-y}{x^*-y^*}\geq 0,
\end{equation}
and it is maximally monotone if, 
for every $(x,x^*)\in\KK\times\KK$, 
\begin{equation} 
\label{e:maxmon2}
(x,x^*)\in\gra A\quad\Leftrightarrow\quad
\big(\forall (y,y^*)\in\gra A\big)\;\;\scal{x-y}{x^*-y^*}\geq 0.
\end{equation}
If $A$ is maximally monotone, then $J_A$ is a single-valued
operator defined on $\KK$.
The parallel sum of
$B\colon\KK\to 2^{\KK}$ and $D\colon\KK\to 2^{\KK}$ is
$B\infconv D=(B^{-1}+D^{-1})^{-1}$. An operator
$C\colon\KK\to\KK$ is cocoercive with constant $\alpha\in\RPP$ if 
$(\forall x\in\KK)(\forall y\in\KK)$
$\scal{x-y}{Cx-Cy}\geq\alpha\|Cx-Cy\|^2$.
We denote by $\Gamma_0(\KK)$ the class of lower semicontinuous
convex functions $f\colon\KK\to\RX$ such that 
$\dom f=\menge{x\in\KK}{f(x)<\pinf}\neq\emp$.
Let $f\in\Gamma_0(\KK)$. The conjugate of $f$ is
the function $\Gamma_0(\KK)\ni f^*\colon x^*\mapsto
\sup_{x\in\KK}(\scal{x}{x^*}-f(x))$ and the subdifferential of $f$
is the maximally monotone operator
$\partial f\colon\KK\to 2^{\KK}\colon
x\mapsto\menge{x^*\in\KK}{(\forall y\in\KK)\;
\scal{y-x}{x^*}+f(x)\leq f(y)}$. In addition,
$\epi f$ is the epigraph of $f$. For every $x\in\KK$, the unique
minimizer of $f+(1/2)\|\mute -x\|^2$ is denoted by $\prox_fx$. We
have $\prox_f=J_{\partial f}$. Given $h\in\Gamma_0(\KK)$, the
infimal convolution of $f$ and $h$ is $f\infconv
h\colon\KK\to\RXX\colon x\mapsto\inf_{y\in\KK}(f(y)+h(x-y))$; the
infimal convolution $f\infconv h$ is exact if the infimum is
achieved everywhere, in which case we write $f\einfconv h$. 
Now let $(\KK_i)_{i\in I}$ be a finite family of real
Hilbert spaces and, for every $i\in I$, let $f_i\colon\KK_i\to\RX$.
Then 
\begin{equation} 
\bigoplus_{i\in I}f_i\colon
\boldsymbol{\KK}=\bigoplus_{i\in I}\KK_i\to\RX\colon
\boldsymbol{x}\mapsto\sum_{i\in I}f_i(x_i).
\end{equation} 
The partial derivative of a differentiable function
$\Theta\colon\boldsymbol{\KK}\to\RR$ relative to $\KK_i$ is denoted
by $\nabla_{\!i}\,\Theta$. Finally, let $C$ be a nonempty convex
subset of $\KK$. A point $x\in C$ belongs to the strong relative
interior of $C$, in symbols $x\in\sri C$, if
$\bigcup_{\lambda\in\RPP}\lambda(C-x)$ is a closed vector subspace
of $\KK$. If $C$ is closed, the projection operator onto it is
denoted by $\proj_C$ and the normal cone operator of $C$ is the
maximally monotone operator
\begin{equation}
N_C\colon\KK\to 2^\KK\colon x\mapsto
\begin{cases}
\menge{x^*\in\KK}{\sup\scal{C-x}{x^*}\leq 0},&
\text{if}\;\:x\in C;\\
\emp,&\text{otherwise}.
\end{cases}
\end{equation}

\section{The saddle form of Problem~\ref{prob:1}}
\label{sec:2}

A classical Lagrangian setting for convex minimization is the
following. Given real Hilbert spaces $\HH$ and $\GG$,
$f\in\Gamma_0(\HH)$, $g\in\Gamma_0(\GG)$,
and a bounded linear operator $L\colon\HH\to\GG$,
consider the primal problem
\begin{equation}
\label{e:7039p}
\minimize{x\in\HH}{f(x)+g(Lx)}
\end{equation}
together with its Fenchel--Rockafellar dual \cite{Rock67}
\begin{equation}
\label{e:7039d}
\minimize{v^*\in\GG}{f^*\big({-}L^*v^*\big)+g^*(v^*)}.
\end{equation}
The primal-dual pair \eqref{e:7039p}--\eqref{e:7039d} can be
analyzed through the lens of Rockafellar's saddle formalism
\cite{Roc70b,Rock71d} as follows. Set 
$h\colon\HH\oplus\GG\to\RX\colon (x,y)\mapsto f(x)+g(y)$ and 
$U\colon\HH\oplus\GG\to\GG\colon (x,y)\mapsto Lx-y$, and note 
that $U^*\colon\GG\to\HH\oplus\GG\colon v^*\mapsto(L^*v^*,-v^*)$.
Then, upon defining $\KK=\HH\oplus\GG$ and
introducing the variable $z=(x,y)\in\KK$,
\eqref{e:7039p} is equivalent to
\begin{equation}
\label{e:7038p}
\minimize{{z\in\KK,\, Uz=0}}{h(z)}
\end{equation}
and \eqref{e:7039d} to
\begin{equation}
\label{e:7038d}
\minimize{v^*\in\GG}{h^*\big({-}U^*v^*\big)}.
\end{equation}
The Lagrangian associated with \eqref{e:7038p} is
(see \cite[Example~4']{Rock74} or
\cite[Proposition~19.21]{Livre1}) 
\begin{equation}
\label{e:L}
\begin{aligned}
\EuScript{L}\colon\KK\oplus\GG&\to\RX\\
(z,v^*)&\mapsto
\begin{cases}
h(z)+\scal{Uz}{v^*},&\text{if}\;\:z\in\dom h;\\
\pinf,&\text{otherwise},
\end{cases}
\end{aligned}
\end{equation}
and the associated saddle operator \cite{Roc70b,Rock71d}
is the maximally monotone operator
\begin{equation}
\sad\colon\KK\oplus\GG\to 2^{\KK\oplus\GG}\colon(z,v^*)\mapsto
\partial\EuScript{L}(\cdot,v^*)(z)\times
\partial\big({-}\EuScript{L}(z,\cdot)\big)(v^*)
=\big(\partial h(z)+U^*v^*\big)\times\{{-}Uz\}.
\end{equation}
As shown in \cite{Roc70b}, a zero $(\overline{z},\overline{v}^*)$
of $\sad$ is a saddle point of $\EuScript{L}$, and it has the
property that $\overline{z}$ solves \eqref{e:7038p} and
$\overline{v}^*$ solves \eqref{e:7038d}. Thus, going back to the
original Fenchel--Rockafellar pair
\eqref{e:7039p}--\eqref{e:7039d}, we learn that, if
$(\overline{x},\overline{y},\overline{v}^*)$ is a zero of the
saddle operator
\begin{equation}
\sad\colon\HH\oplus\GG\oplus\GG\to 2^{\HH\oplus\GG\oplus\GG}
\colon(x,y,v^*)\mapsto
\big(\partial f(x)+L^*v^*\big)
\times\big(\partial g(y)-v^*\big)
\times\{{-}Lx+y\},
\end{equation}
then $\overline{x}$ solves \eqref{e:7039p} and
$\overline{v}^*$ solves \eqref{e:7039d}.
As shown in \cite[Section~4.5]{Mono18}, a suitable splitting of
$\sad$ leads to an implementable algorithm to solve
\eqref{e:7039p}--\eqref{e:7039d}.

A generalization of Fenchel--Rockafellar duality to
monotone inclusions was proposed in \cite{Penn00,Robi99} and 
further extended in \cite{Siop13}. Given maximally monotone
operators $A\colon\HH\to 2^{\HH}$ and $B\colon\GG\to 2^{\GG}$,
and a bounded linear operator $L\colon\HH\to\GG$, the primal
problem 
\begin{equation}
\label{e:primal}
\text{find}\;\:\overline{x}\in\HH\;\:\text{such that}\;\:
0\in A\overline{x}+L^*\big(B(L\overline{x})\big)
\end{equation}
is paired with the dual problem 
\begin{equation}
\label{e:dual}
\text{find}\;\:\overline{v}^*\in\GG\;\:\text{such that}\;\:
0\in{-}L\big(A^{-1}(-L^*\overline{v}^*)\big)+B^{-1}\overline{v}^*.
\end{equation}
Following the same pattern as that described above,
let us consider the \emph{saddle operator}
\begin{equation}
\label{e:sad1}
\sad\colon\HH\oplus\GG\oplus\GG\to 2^{\HH\oplus\GG\oplus\GG}
\colon(x,y,v^*)\mapsto(Ax+L^*v^*)
\times(By-v^*)\times\{{-}Lx+y\}.
\end{equation}
It is readily shown that, if 
$(\overline{x},\overline{y},\overline{v}^*)$ is a zero of 
$\sad$, then $\overline{x}$ solves \eqref{e:primal} and
$\overline{v}^*$ solves \eqref{e:dual}. We call the problem of
finding a zero of $\sad$ the \emph{saddle form} of 
\eqref{e:primal}--\eqref{e:dual}. We now introduce a saddle
operator for the general Problem~\ref{prob:1}.

\newpage

\begin{definition}
\label{d:S}
In the setting of Problem~\ref{prob:1}, let
$\XXX=\HHH\oplus\GGG\oplus\GGG\oplus\GGG$.
The \emph{saddle operator} associated with Problem~\ref{prob:1} is
\begin{align}
\label{e:saddle}
\sad\colon\XXX\to2^{\XXX}\colon&
(\boldsymbol{x},\boldsymbol{y},\boldsymbol{z},\boldsymbol{v}^*)
\mapsto\nonumber\\
&\hskip -12mm\Bigg(\bigtimes_{i\in I}
\bigg({-}s_i^*+A_ix_i+C_ix_i+Q_ix_i+R_i\boldsymbol{x}
+\sum_{k\in K}L^*_{ki}v^*_k\bigg),
\bigtimes_{k\in K}\big(B_k^{\MM}y_k
+B_k^{\CC}y_k+B_k^{\LL}y_k-v_k^*\big),\nonumber\\
&\hskip -12mm\;\;\bigtimes_{k\in K}\big(D_k^{\MM}z_k
+D_k^{\CC}z_k+D_k^{\LL}z_k-v_k^*\big),\bigtimes_{k\in K}\bigg\{
r_k+y_k+z_k-\sum_{i\in I}L_{ki}x_i\bigg\}~\Bigg),
\end{align}
and the \emph{saddle form} of Problem~\ref{prob:1} is to 
\begin{equation}
\label{e:sf}
\text{find}\;\:\overline{\boldsymbol{\mathsf{x}}}\in\XXX\;\:
\text{such that}\;\:\boldsymbol{\mathsf{0}}\in
\sad\overline{\boldsymbol{\mathsf{x}}}.
\end{equation}
\end{definition}

Next, we establish some properties of the saddle operator as well
as connections with Problem~\ref{prob:1}.

\begin{proposition}
\label{p:6}
Consider the setting of Problem~\ref{prob:1} 
and Definition~\ref{d:S}. Let $\mathscr{P}$ be the set of
solutions to \eqref{e:1p}, let $\mathscr{D}$ be the set of
solutions to \eqref{e:1d}, and let 
\begin{multline}
\label{e:Z}
\mathsf{Z}=\Bigg\{
(\overline{\boldsymbol{x}},\overline{\boldsymbol{v}}^*)\in
\HHH\oplus\GGG ~\bigg|~
(\forall i\in I)(\forall k\in K)\quad
s^*_i-\sum_{j\in K}L_{ji}^*\overline{v}_j^*\in
A_i\overline{x}_i+C_i\overline{x}_i+Q_i\overline{x}_i
+R_i\overline{\boldsymbol{x}}
\;\:\text{and}\\[-1mm]
\sum_{j\in I}L_{kj}\overline{x}_j-r_k\in
\big(B_k^{\MM}+B_k^{\CC}+B_k^{\LL}\big)^{-1}\overline{v}_k^*
+\big(D_k^{\MM}+D_k^{\CC}+D_k^{\LL}\big)^{-1}\overline{v}_k^*
\Bigg\}
\end{multline}
be the associated \emph{Kuhn--Tucker} set.
Then the following hold:
\begin{enumerate}
\item
\label{p:6i}
$\sad$ is maximally monotone.
\item
\label{p:6i+}
$\zer\sad$ is closed and convex.
\item
\label{p:6ii}
Suppose that $\overline{\boldsymbol{\mathsf{x}}}
=(\overline{\boldsymbol{x}},\overline{\boldsymbol{y}},
\overline{\boldsymbol{z}},\overline{\boldsymbol{v}}^*)
\in\zer\sad$. Then $(\overline{\boldsymbol{x}},
\overline{\boldsymbol{v}}^*)\in
\mathsf{Z}\subset\mathscr{P}\times\mathscr{D}$.
\item
\label{p:6iv}
$\mathscr{D}\neq\emp$ $\Leftrightarrow$
$\zer\sad\neq\emp$ $\Leftrightarrow$ $\mathsf{Z}\neq\emp$
$\Rightarrow$ $\mathscr{P}\neq\emp$.
\item
\label{p:6v}
Suppose that one of the following holds:
\begin{enumerate}[label={\rm[\alph*]}]
\item
\label{p:6va}
$I$ is a singleton.
\item
\label{p:6vb}
For every $k\in K$, $(B_k^{\MM}+B_k^{\CC}+B_k^{\LL})
\infconv(D_k^{\MM}+D_k^{\CC}+D_k^{\LL})$ is
at most single-valued.
\item
\label{p:6vc}
For every $k\in K$,
$(D_k^{\MM}+D_k^{\CC}+D_k^{\LL})^{-1}$ is strictly monotone.
\item
\label{p:6vd}
$I\subset K$, the operators
$((B_k^{\MM}+B_k^{\CC}+B_k^{\LL})
\infconv(D_k^{\MM}+D_k^{\CC}+D_k^{\LL}))_{k\in K\smallsetminus I}$
are at most single-valued, and
$(\forall i\in I)(\forall k\in I)$
$k\neq i$ $\Rightarrow$ $L_{ki}=0$.
\end{enumerate}
Then $\mathscr{P}\neq\emp$ $\Rightarrow$ $\mathsf{Z}\neq\emp$.
\end{enumerate}
\end{proposition}
\begin{proof}
Define
\begin{equation}
\label{e:8757}
\begin{cases}
\boldsymbol{A}\colon\HHH\to 2^{\HHH}
\colon\boldsymbol{x}\mapsto\boldsymbol{R}\boldsymbol{x}
+\bigtimes_{i\in I}\big(A_ix_i+C_ix_i+Q_ix_i\big)\\
\boldsymbol{B}\colon\GGG\to 2^{\GGG}\colon\boldsymbol{y}\mapsto
\bigtimes_{k\in K}\big(B_k^{\MM}y_k+B_k^{\CC}y_k+B_k^{\LL}y_k\big)
\\
\boldsymbol{D}\colon\GGG\to 2^{\GGG}\colon\boldsymbol{z}\mapsto
\bigtimes_{k\in K}
\big(D_k^{\MM}z_k+D_k^{\CC}z_k+D_k^{\LL}z_k\big)\\
\boldsymbol{L}\colon\HHH\to\GGG\colon\boldsymbol{x}\mapsto
\big(\sum_{i\in I}L_{ki}x_i\big)_{k\in K}\\
\boldsymbol{s}^*=(s_i^*)_{i\in I}\;\:\text{and}\;\:
\boldsymbol{r}=(r_k)_{k\in K}.
\end{cases}
\end{equation}
Then the adjoint of $\boldsymbol{L}$ is
\begin{equation}
\label{e:3041}
\boldsymbol{L}^*\colon\GGG\to\HHH\colon
\boldsymbol{v}^*\mapsto
\Bigg(\Sum_{k\in K}L^*_{ki}v_k^*\Bigg)_{i\in I}.
\end{equation}
Hence, in view of \eqref{e:saddle} and \eqref{e:8757},
\begin{equation}
\label{e:3762}
\sad\colon\XXX\to 2^{\XXX}\colon
(\boldsymbol{x},\boldsymbol{y},\boldsymbol{z},\boldsymbol{v}^*)
\mapsto\big({-}\boldsymbol{s}^*+\boldsymbol{A}\boldsymbol{x}
+\boldsymbol{L}^*\boldsymbol{v}^*\big)
\times\big(\boldsymbol{B}\boldsymbol{y}-\boldsymbol{v}^*\big)
\times\big(\boldsymbol{D}\boldsymbol{z}-\boldsymbol{v}^*\big)
\times\big\{\boldsymbol{r}-\boldsymbol{L}\boldsymbol{x}
+\boldsymbol{y}+\boldsymbol{z}\big\}.
\end{equation}

\ref{p:6i}:
Let us introduce the operators
\begin{equation}
\label{e:4425}
\begin{cases}
\boldsymbol{\mathsf{P}}\colon\XXX\to 2^{\XXX}\colon
(\boldsymbol{x},\boldsymbol{y},\boldsymbol{z},\boldsymbol{v}^*)
\mapsto({-}\boldsymbol{s}^*+\boldsymbol{A}\boldsymbol{x})
\times\boldsymbol{B}\boldsymbol{y}\times
\boldsymbol{D}\boldsymbol{z}\times\{\boldsymbol{r}\}\\
\boldsymbol{\mathsf{W}}\colon\XXX\to\XXX\colon
(\boldsymbol{x},\boldsymbol{y},\boldsymbol{z},\boldsymbol{v}^*)
\mapsto(\boldsymbol{L}^*\boldsymbol{v}^*,
{-}\boldsymbol{v}^*,{-}\boldsymbol{v}^*,
{-}\boldsymbol{L}\boldsymbol{x}+\boldsymbol{y}+\boldsymbol{z}).
\end{cases}
\end{equation}
Using Problem~\ref{prob:1}\ref{prob:1a}--\ref{prob:1c},
we derive from
\cite[Example~20.31, Corollaries~20.28 and~25.5(i)]{Livre1} that,
for every $i\in I$ and every $k\in K$, the operators
$A_i+C_i+Q_i$, $B_k^{\MM}+B_k^{\CC}+B_k^{\LL}$,
and $D_k^{\MM}+D_k^{\CC}+D_k^{\LL}$ are maximally monotone.
At the same time, Problem~\ref{prob:1}\ref{prob:1e}
and \cite[Corollary~20.28]{Livre1} entail that $\boldsymbol{R}$
is maximally monotone.
Therefore, it results from \eqref{e:8757}, 
\cite[Proposition~20.23 and Corollary~25.5(i)]{Livre1},
and \eqref{e:4425} that $\boldsymbol{\mathsf{P}}$ is
maximally monotone. However, since
Problem~\ref{prob:1}\ref{prob:1d} and \eqref{e:4425}
imply that $\boldsymbol{\mathsf{W}}$ is linear and bounded
with $\boldsymbol{\mathsf{W}}^*={-}\boldsymbol{\mathsf{W}}$,
\cite[Example~20.35]{Livre1} asserts that
$\boldsymbol{\mathsf{W}}$ is maximally monotone.
Hence, in view of \cite[Corollary~25.5(i)]{Livre1},
we infer from \eqref{e:3762}--\eqref{e:4425} that
$\sad=\boldsymbol{\mathsf{P}}+\boldsymbol{\mathsf{W}}$
is maximally monotone.

\ref{p:6i+}: This follows from \ref{p:6i} and 
\cite[Proposition~23.39]{Livre1}.

\ref{p:6ii}:
Using \eqref{e:8757} and \eqref{e:3041}, we deduce from \eqref{e:Z}
that
\begin{equation}
\label{e:Z1}
\mathsf{Z}=\menge{(\boldsymbol{x},\boldsymbol{v}^*)
\in\HHH\oplus\GGG}{\boldsymbol{s}^*
-\boldsymbol{L}^*\boldsymbol{v}^*\in\boldsymbol{A}\boldsymbol{x}
\;\:\text{and}\;\:\boldsymbol{L}\boldsymbol{x}-\boldsymbol{r}
\in\boldsymbol{B}^{-1}\boldsymbol{v}^*
+\boldsymbol{D}^{-1}\boldsymbol{v}^*}
\end{equation}
and from \eqref{e:1d} that
\begin{equation}
\label{e:1dd}
\mathscr{D}=\menge{\boldsymbol{v}^*\in\GGG}{
{-}\boldsymbol{r}\in{-}\boldsymbol{L}
\big(\boldsymbol{A}^{-1}(\boldsymbol{s}^*
-\boldsymbol{L}^*\boldsymbol{v}^*)\big)
+\boldsymbol{B}^{-1}\boldsymbol{v}^*
+\boldsymbol{D}^{-1}\boldsymbol{v}^*}.
\end{equation}
Suppose that $(\boldsymbol{x},\boldsymbol{v}^*)\in\mathsf{Z}$.
Then it follows from \eqref{e:Z1} that
$\boldsymbol{x}\in\boldsymbol{A}^{-1}
(\boldsymbol{s}^*-\boldsymbol{L}^*\boldsymbol{v}^*)$
and, in turn, that
${-}\boldsymbol{r}\in{-}\boldsymbol{L}\boldsymbol{x}
+\boldsymbol{B}^{-1}\boldsymbol{v}^*
+\boldsymbol{D}^{-1}\boldsymbol{v}^*\subset
{-}\boldsymbol{L}(\boldsymbol{A}^{-1}(\boldsymbol{s}^*
-\boldsymbol{L}^*\boldsymbol{v}^*))
+\boldsymbol{B}^{-1}\boldsymbol{v}^*
+\boldsymbol{D}^{-1}\boldsymbol{v}^*$.
Thus $\boldsymbol{v}^*\in\mathscr{D}$ by \eqref{e:1dd}.
In addition, \eqref{e:Z} implies that
\begin{equation}
(\forall k\in K)\quad v_k^*\in
\big((B_k^{\MM}+B_k^{\CC}+B_k^{\LL})
\infconv(D_k^{\MM}+D_k^{\CC}+D_k^{\LL})\big)\Bigg(
\Sum_{j\in I}L_{kj}x_j-r_k\Bigg)
\end{equation}
and, therefore, that
\begin{align}
(\forall i\in I)\quad s_i^*
&\in A_ix_i+C_ix_i+Q_ix_i+R_i\boldsymbol{x}
+\Sum_{k\in K}L_{ki}^*v_k^*
\nonumber\\
&\subset A_ix_i+C_ix_i+Q_ix_i+R_i\boldsymbol{x}
\nonumber\\
&\quad\;+\Sum_{k\in K}L_{ki}^*\Bigg(\Big(
\big(B_k^{\MM}+B_k^{\CC}+B_k^{\LL}\big)\infconv
\big(D_k^{\MM}+D_k^{\CC}+D_k^{\LL}\big)\Big)
\Bigg(\Sum_{j\in I}L_{kj}x_j-r_k\Bigg)\Bigg).
\end{align}
Hence, $\boldsymbol{x}\in\mathscr{P}$.
To summarize, we have shown that
$\mathsf{Z}\subset\mathscr{P}\times\mathscr{D}$.
It remains to show that
$(\overline{\boldsymbol{x}},\overline{\boldsymbol{v}}^*)
\in\mathsf{Z}$. Since
$\boldsymbol{\mathsf{0}}\in
\sad\overline{\boldsymbol{\mathsf{x}}}$,
we deduce from \eqref{e:3762} that
$\boldsymbol{s}^*-\boldsymbol{L}^*\overline{\boldsymbol{v}}^*
\in\boldsymbol{A}\overline{\boldsymbol{x}}$,
$\boldsymbol{L}\overline{\boldsymbol{x}}-\boldsymbol{r}
=\overline{\boldsymbol{y}}+\overline{\boldsymbol{z}}$,
$\boldsymbol{0}\in
\boldsymbol{B}\overline{\boldsymbol{y}}
-\overline{\boldsymbol{v}}^*$, and
$\boldsymbol{0}\in\boldsymbol{D}\overline{\boldsymbol{z}}
-\overline{\boldsymbol{v}}^*$. Therefore,
$\boldsymbol{L}\overline{\boldsymbol{x}}-\boldsymbol{r}\in
\boldsymbol{B}^{-1}\overline{\boldsymbol{v}}^*
+\boldsymbol{D}^{-1}\overline{\boldsymbol{v}}^*$
and \eqref{e:Z1} thus yields
$(\overline{\boldsymbol{x}},\overline{\boldsymbol{v}}^*)
\in\mathsf{Z}$.

\ref{p:6iv}:
The implication $\zer\sad\neq\emp$ $\Rightarrow$
$\mathscr{P}\neq\emp$ follows from \ref{p:6ii}.
Next, we derive from \eqref{e:1dd} and \eqref{e:Z1} that
\begin{align}
\mathscr{D}\neq\emp&
\Leftrightarrow(\exi\overline{\boldsymbol{v}}^*\in\GGG)\;\;
{-}\boldsymbol{r}\in
{-}\boldsymbol{L}\big(\boldsymbol{A}^{-1}
(\boldsymbol{s}^*-\boldsymbol{L}^*\overline{\boldsymbol{v}}^*)\big)
+\boldsymbol{B}^{-1}\overline{\boldsymbol{v}}^*
+\boldsymbol{D}^{-1}\overline{\boldsymbol{v}}^*
\nonumber\\
&\Leftrightarrow
\big(\exi(\overline{\boldsymbol{v}}^*,
\overline{\boldsymbol{x}})\in\GGG\oplus\HHH\big)
\;\;{-}\boldsymbol{r}\in{-}\boldsymbol{L}\overline{\boldsymbol{x}}
+\boldsymbol{B}^{-1}\overline{\boldsymbol{v}}^*
+\boldsymbol{D}^{-1}\overline{\boldsymbol{v}}^*
\;\:\text{and}\;\:\overline{\boldsymbol{x}}\in
\boldsymbol{A}^{-1}(\boldsymbol{s}^*
-\boldsymbol{L}^*\overline{\boldsymbol{v}}^*)
\nonumber\\
&\Leftrightarrow
\big(\exi(\overline{\boldsymbol{x}},\overline{\boldsymbol{v}}^*)
\in\HHH\oplus\GGG\big)\;\;
\boldsymbol{s}^*-\boldsymbol{L}^*\overline{\boldsymbol{v}}^*\in
\boldsymbol{A}\overline{\boldsymbol{x}}\;\:\text{and}\;\:
\boldsymbol{L}\overline{\boldsymbol{x}}-\boldsymbol{r}\in
\boldsymbol{B}^{-1}\overline{\boldsymbol{v}}^*
+\boldsymbol{D}^{-1}\overline{\boldsymbol{v}}^*
\nonumber\\
&\Leftrightarrow\mathsf{Z}\neq\emp.
\end{align}
However, \ref{p:6ii} asserts that
$\zer\sad\neq\emp$ $\Rightarrow$ $\mathsf{Z}\neq\emp$.
Therefore, it remains to show that $\mathsf{Z}\neq\emp$
$\Rightarrow$ $\zer\sad\neq\emp$.
Towards this end, suppose that
$(\overline{\boldsymbol{x}},\overline{\boldsymbol{v}}^*)
\in\mathsf{Z}$. Then, by \eqref{e:Z1},
$\boldsymbol{s}^*-\boldsymbol{L}^*\overline{\boldsymbol{v}}^*\in
\boldsymbol{A}\overline{\boldsymbol{x}}$ and
$\boldsymbol{L}\overline{\boldsymbol{x}}-\boldsymbol{r}\in
\boldsymbol{B}^{-1}\overline{\boldsymbol{v}}^*
+\boldsymbol{D}^{-1}\overline{\boldsymbol{v}}^*$.
Hence, $\boldsymbol{0}\in{-}\boldsymbol{s}^*
+\boldsymbol{A}\overline{\boldsymbol{x}}
+\boldsymbol{L}^*\overline{\boldsymbol{v}}^*$
and there exists
$(\overline{\boldsymbol{y}},\overline{\boldsymbol{z}})\in
\GGG\oplus\GGG$ such that
$\overline{\boldsymbol{y}}\in
\boldsymbol{B}^{-1}\overline{\boldsymbol{v}}^*$,
$\overline{\boldsymbol{z}}\in
\boldsymbol{D}^{-1}\overline{\boldsymbol{v}}^*$,
and
$\boldsymbol{L}\overline{\boldsymbol{x}}-\boldsymbol{r}
=\overline{\boldsymbol{y}}+\overline{\boldsymbol{z}}$.
We thus deduce that
$\boldsymbol{0}\in\boldsymbol{B}\overline{\boldsymbol{y}}
-\overline{\boldsymbol{v}}^*$,
$\boldsymbol{0}\in\boldsymbol{D}\overline{\boldsymbol{z}}
-\overline{\boldsymbol{v}}^*$,
and $\boldsymbol{r}-\boldsymbol{L}\overline{\boldsymbol{x}}
+\overline{\boldsymbol{y}}+\overline{\boldsymbol{z}}
=\boldsymbol{0}$. Consequently, \eqref{e:3762} implies that
$(\overline{\boldsymbol{x}},\overline{\boldsymbol{y}},
\overline{\boldsymbol{z}},\overline{\boldsymbol{v}}^*)
\in\zer\sad$.

\ref{p:6v}:
In view of \ref{p:6iv}, it suffices to establish
that $\mathscr{P}\neq\emp$ $\Rightarrow$ $\mathscr{D}\neq\emp$.
Suppose that $\overline{\boldsymbol{x}}\in\mathscr{P}$.

\ref{p:6va}:
Suppose that $I=\{1\}$. We then infer from \eqref{e:1p} that
there exists $\overline{\boldsymbol{v}}^*\in\GGG$ such that
\begin{equation}
\begin{cases}
s_1^*\in A_1\overline{x}_1
+C_1\overline{x}_1+Q_1\overline{x}_1
+R_1\overline{\boldsymbol{x}}
+\Sum_{k\in K}L_{k1}^*\overline{v}_k^*\\
(\forall k\in K)\;\;\overline{v}_k^*\in\big(
(B_k^{\MM}+B_k^{\CC}+B_k^{\LL})\infconv
(D_k^{\MM}+D_k^{\CC}+D_k^{\LL})\big)
(L_{k1}\overline{x}_1-r_k).
\end{cases}
\end{equation}
Therefore, by \eqref{e:1d},
$\overline{\boldsymbol{v}}^*\in\mathscr{D}$.

\ref{p:6vb}: Set
$(\forall k\in K)$
$\overline{v}_k^*=(
(B_k^{\MM}+B_k^{\CC}+B_k^{\LL})\infconv
(D_k^{\MM}+D_k^{\CC}+D_k^{\LL}))
(\sum_{j\in I}L_{kj}\overline{x}_j-r_k)$.
Then $\overline{\boldsymbol{v}}^*$ solves \eqref{e:1d}.

\ref{p:6vc}$\Rightarrow$\ref{p:6vb}: See 
\cite[Section~4]{Siop13}.

\ref{p:6vd}:
Let $i\in I$. It results from our assumption that
\begin{multline}
\label{e:7937}
s_i^*\in A_i\overline{x}_i+C_i\overline{x}_i+Q_i\overline{x}_i
+R_i\overline{\boldsymbol{x}}
+L_{ii}^*\Big(\big((B_i^{\MM}+B_i^{\CC}+B_i^{\LL})\infconv
(D_i^{\MM}+D_i^{\CC}+D_i^{\LL})\big)
(L_{ii}\overline{x}_i-r_i)\Big)
\\
+\Sum_{k\in K\smallsetminus I}L_{ki}^*\Bigg(\Big(
\big(B_k^{\MM}+B_k^{\CC}+B_k^{\LL}\big)\infconv
\big(D_k^{\MM}+D_k^{\CC}+D_k^{\LL}\big)\Big)
\Bigg(\Sum_{j\in I}L_{kj}\overline{x}_j-r_k\Bigg)\Bigg).
\end{multline}
Thus, there exists $\overline{v}_i^*\in\GG_i$ such that
$\overline{v}_i^*
\in((B_i^{\MM}+B_i^{\CC}+B_i^{\LL})\infconv
(D_i^{\MM}+D_i^{\CC}+D_i^{\LL}))(L_{ii}\overline{x}_i-r_i)$
and that
\begin{multline}
s_i^*\in A_i\overline{x}_i+C_i\overline{x}_i+Q_i\overline{x}_i
+R_i\overline{\boldsymbol{x}}
+L_{ii}^*\overline{v}_i^*\\
+\Sum_{k\in K\smallsetminus I}L_{ki}^*\Bigg(\Big(
\big(B_k^{\MM}+B_k^{\CC}+B_k^{\LL}\big)\infconv
\big(D_k^{\MM}+D_k^{\CC}+D_k^{\LL}\big)\Big)
\Bigg(\Sum_{j\in I}L_{kj}\overline{x}_j-r_k\Bigg)\Bigg).
\end{multline}
As a result, upon setting
\begin{equation}
(\forall k\in K\smallsetminus I)\quad
\overline{v}_k^*
=\big((B_k^{\MM}+B_k^{\CC}+B_k^{\LL})\infconv
(D_k^{\MM}+D_k^{\CC}+D_k^{\LL})\big)
\Bigg(\Sum_{j\in I}L_{kj}\overline{x}_j-r_k\Bigg),
\end{equation}
we conclude that
$\overline{\boldsymbol{v}}^*\in\mathscr{D}$.
\end{proof}

\begin{remark}
\label{r:7}
Some noteworthy observations about Proposition~\ref{p:6} are the
following.
\begin{enumerate}
\item
The Kuhn--Tucker set \eqref{e:Z} extends to
Problem~\ref{prob:1} the corresponding notion introduced
for some special cases in \cite{Siop14,Siop11,MaPr18}.
\item
In connection with Proposition~\ref{p:6}\ref{p:6v},
we note that the implication $\mathscr{P}\neq\emp$ $\Rightarrow$
$\mathsf{Z}\neq\emp$ is implicitly used in \cite[Theorems~13 and
15]{MaPr18}, where one requires $\mathsf{Z}\neq\emp$ but merely
assumes $\mathscr{P}\neq\emp$. However, this implication is not
true in general (a similar oversight is found in
\cite{Siop14,Pesq15,Bang16}). Indeed, consider as a special case
of \eqref{e:1p}, the problem of solving the system 
\begin{equation}
\begin{cases}
0\in B_1(x_1+x_2)+B_2(x_1-x_2)\\
0\in B_1(x_1+x_2)-B_2(x_1-x_2)
\end{cases}
\end{equation}
in the Euclidean plane $\RR^2$. Then, by choosing
$B_1=\{0\}^{-1}$ and $B_2=1$, we obtain
$\mathscr{P}=\menge{(x_1,-x_1)}{x_1\in\RR}$,
whereas $\mathsf{Z}=\emp$.
\item
\label{r:7iii}
As stated in Proposition~\ref{p:6}\ref{p:6ii}, any Kuhn--Tucker
point is a solution to \eqref{e:1p}--\eqref{e:1d}. In the simpler
setting considered in \cite{MaPr18}, a splitting algorithm was
devised for finding such a point. However, in the more general
context of Problem~\ref{prob:1}, there does not seem to exist a
path from the Kuhn--Tucker formalism in $\HHH\oplus\GGG$ to an
algorithm that is fully split in the sense of \ref{f:1}. This
motivates our approach, which seeks a zero of the saddle operator
$\sad$ defined on the bigger space $\XXX$ and, thereby, offers more
flexibility.
\item
\label{r:7iv}
Special cases of Problem~\ref{prob:1} can be found in
\cite{Siop14,MaPr18,John20,John19}, where they were solved by
algorithms that proceed by outer approximation of the Kuhn--Tucker
set in $\HHH\oplus\GGG$. In those special cases,
Algorithm~\ref{algo:1} below does not reduce to those of
\cite{Siop14,MaPr18,John20,John19} since it operates by outer
approximation of the set of zeros of the saddle operator $\sad$ in
the bigger space $\XXX$.
\end{enumerate}
\end{remark}

The following operators will induce a decomposition of the saddle
operator that will lead to a splitting algorithm
which complies with our requirements \ref{f:1}--\ref{f:5}.

\begin{definition}
\label{d:CM}
In the setting of Definition~\ref{d:S}, set
\begin{align}
\label{e:1179}
\boldsymbol{\mathsf{M}}\colon\XXX\to2^{\XXX}\colon&
(\boldsymbol{x},\boldsymbol{y},\boldsymbol{z},\boldsymbol{v}^*)
\mapsto\nonumber\\
&\Bigg(
\bigtimes_{i\in I}\bigg({-}s_i^*+A_ix_i+Q_ix_i
+R_i\boldsymbol{x}
+\sum_{k\in K}L^*_{ki}v^*_k\bigg),
\bigtimes_{k\in K}\big(B_k^{\MM}y_k
+B_k^{\LL}y_k-v_k^*\big),\nonumber\\
&\;\;\bigtimes_{k\in K}\big(D_k^{\MM}z_k
+D_k^{\LL}z_k-v_k^*\big),\bigtimes_{k\in K}\bigg\{
r_k+y_k+z_k-\sum_{i\in I}L_{ki}x_i\bigg\}~\Bigg)
\end{align}
and
\begin{equation}
\label{e:2796}
\boldsymbol{\mathsf{C}}\colon\XXX\to\XXX\colon
(\boldsymbol{x},\boldsymbol{y},\boldsymbol{z},\boldsymbol{v}^*)
\mapsto\Big(
\big(C_ix_i\big)_{i\in I},\big(B_k^{\CC}y_k\big)_{k\in K},
\big(D_k^{\CC}z_k\big)_{k\in K},\boldsymbol{0}\Big).
\end{equation}
\end{definition}

\begin{proposition}
\label{p:2}
In the setting of Problem~\ref{prob:1} and of
Definitions~\ref{d:S} and \ref{d:CM}, the following hold:
\begin{enumerate}
\item
\label{p:2i-}
$\sad=\boldsymbol{\mathsf{M}}+\boldsymbol{\mathsf{C}}$.
\item
\label{p:2i}
$\boldsymbol{\mathsf{M}}$ is maximally monotone.
\item
\label{p:2ii}
Set $\alpha=\min\{\alpha_i^{\CC},\beta_k^{\CC},
\delta_k^{\CC}\}_{i\in I,k\in K}$.
Then the following hold:
\begin{enumerate}
\item
\label{p:2iia}
$\boldsymbol{\mathsf{C}}$ is $\alpha$-cocoercive.
\item
\label{p:2iib}
Let $(\boldsymbol{\mathsf{p}},\boldsymbol{\mathsf{p}}^*)\in
\gra\boldsymbol{\mathsf{M}}$ and $\boldsymbol{\mathsf{q}}\in\XXX$. 
Then
$\zer\sad\subset\menge{\boldsymbol{\mathsf{x}}\in\XXX}
{\scal{\boldsymbol{\mathsf{x}}-
\boldsymbol{\mathsf{p}}}{\boldsymbol{\mathsf{p}}^*
+\boldsymbol{\mathsf{C}}\boldsymbol{\mathsf{q}}}\leq
(4\alpha)^{-1}\|\boldsymbol{\mathsf{p}}
-\boldsymbol{\mathsf{q}}\|^2}.$
\end{enumerate}
\end{enumerate}
\end{proposition}
\begin{proof}
\ref{p:2i-}:
Clear from \eqref{e:saddle}, \eqref{e:1179}, and \eqref{e:2796}.

\ref{p:2i}:
This is a special case of Proposition~\ref{p:6}\ref{p:6i},
where, for every $i\in I$ and every $k\in K$,
$C_i=0$ and $B_k^{\CC}=D_k^{\CC}=0$.

\ref{p:2iia}:
Take $\boldsymbol{\mathsf{x}}
=(\boldsymbol{x},\boldsymbol{y},\boldsymbol{z},\boldsymbol{v}^*)$
and
$\boldsymbol{\mathsf{y}}=
(\boldsymbol{a},\boldsymbol{b},\boldsymbol{c},\boldsymbol{w}^*)$
in $\XXX$. By \eqref{e:2796} and
Problem~\ref{prob:1}\ref{prob:1a}--\ref{prob:1c},
\begin{align}
&\scal{\boldsymbol{\mathsf{x}}-\boldsymbol{\mathsf{y}}}
{\boldsymbol{\mathsf{C}}\boldsymbol{\mathsf{x}}
-\boldsymbol{\mathsf{C}}\boldsymbol{\mathsf{y}}}
\nonumber\\
&\hspace{8mm}
=\Sum_{i\in I}\scal{x_i-a_i}{C_ix_i-C_ia_i}
+\Sum_{k\in K}\big(\scal{y_k-b_k}{B_k^{\CC}y_k-B_k^{\CC}b_k}
+\scal{z_k-c_k}{D_k^{\CC}z_k-D_k^{\CC}c_k}\big)
\nonumber\\
&\hspace{8mm}
\geq\Sum_{i\in I}\alpha_i^{\CC}\|C_ix_i-C_ia_i\|^2
+\Sum_{k\in K}\big(\beta_k^{\CC}\|B_k^{\CC}y_k-B_k^{\CC}b_k\|^2
+\delta_k^{\CC}\|D_k^{\CC}z_k-D_k^{\CC}c_k\|^2\big)
\nonumber\\
&\hspace{8mm}
\geq\alpha\Sum_{i\in I}\|C_ix_i-C_ia_i\|^2
+\alpha\Sum_{k\in K}\big(\|B_k^{\CC}y_k-B_k^{\CC}b_k\|^2
+\|D_k^{\CC}z_k-D_k^{\CC}c_k\|^2\big)
\nonumber\\
&\hspace{8mm}
=\alpha\|\boldsymbol{\mathsf{C}}\boldsymbol{\mathsf{x}}
-\boldsymbol{\mathsf{C}}\boldsymbol{\mathsf{y}}\|^2.
\end{align}

\ref{p:2iib}:
Suppose that $\boldsymbol{\mathsf{z}}\in\zer\sad$.
We deduce from \ref{p:2i-} that
${-}\boldsymbol{\mathsf{C}}\boldsymbol{\mathsf{z}}\in
\boldsymbol{\mathsf{M}}\boldsymbol{\mathsf{z}}$
and from our assumption that
$\boldsymbol{\mathsf{p}}^*\in
\boldsymbol{\mathsf{M}}\boldsymbol{\mathsf{p}}$.
Hence, \ref{p:2i} implies that $\scal{\boldsymbol{\mathsf{z}}
-\boldsymbol{\mathsf{p}}}{\boldsymbol{\mathsf{p}}^*
+\boldsymbol{\mathsf{C}}\boldsymbol{\mathsf{z}}}\leq 0$.
Thus, we infer from \ref{p:2iia}
and the Cauchy--Schwarz inequality that
\begin{align}
\scal{\boldsymbol{\mathsf{z}}
-\boldsymbol{\mathsf{p}}}{\boldsymbol{\mathsf{p}}^*
+\boldsymbol{\mathsf{C}}\boldsymbol{\mathsf{q}}}
&=\scal{\boldsymbol{\mathsf{z}}
-\boldsymbol{\mathsf{p}}}{\boldsymbol{\mathsf{p}}^*
+\boldsymbol{\mathsf{C}}\boldsymbol{\mathsf{z}}}
-\scal{\boldsymbol{\mathsf{z}}
-\boldsymbol{\mathsf{q}}}{
\boldsymbol{\mathsf{C}}\boldsymbol{\mathsf{z}}
-\boldsymbol{\mathsf{C}}\boldsymbol{\mathsf{q}}}
+\scal{\boldsymbol{\mathsf{p}}
-\boldsymbol{\mathsf{q}}}{
\boldsymbol{\mathsf{C}}\boldsymbol{\mathsf{z}}
-\boldsymbol{\mathsf{C}}\boldsymbol{\mathsf{q}}}
\nonumber\\
&\leq{-}\alpha\|\boldsymbol{\mathsf{C}}\boldsymbol{\mathsf{z}}
-\boldsymbol{\mathsf{C}}\boldsymbol{\mathsf{q}}\|^2
+\|\boldsymbol{\mathsf{p}}-\boldsymbol{\mathsf{q}}\|\,
\|\boldsymbol{\mathsf{C}}\boldsymbol{\mathsf{z}}
-\boldsymbol{\mathsf{C}}\boldsymbol{\mathsf{q}}\| 
\nonumber\\
&=(4\alpha)^{-1}\|\boldsymbol{\mathsf{p}}
-\boldsymbol{\mathsf{q}}\|^2
-\Big|\big(2\sqrt{\alpha}\big)^{-1}
\|\boldsymbol{\mathsf{p}}-\boldsymbol{\mathsf{q}}\|
-\sqrt{\alpha}\|\boldsymbol{\mathsf{C}}\boldsymbol{\mathsf{z}}
-\boldsymbol{\mathsf{C}}\boldsymbol{\mathsf{q}}\|\Big|^2
\nonumber\\
&\leq(4\alpha)^{-1}\|\boldsymbol{\mathsf{p}}
-\boldsymbol{\mathsf{q}}\|^2,
\end{align}
which establishes the claim.
\end{proof}

Next, we solve the saddle form \eqref{e:sf} of Problem~\ref{prob:1}
via successive projections onto the outer approximations
constructed in Proposition~\ref{p:2}\ref{p:2iib}.

\begin{proposition}
\label{p:3404}
Consider the setting of Problem~\ref{prob:1} and of
Definitions~\ref{d:S} and \ref{d:CM}, and suppose that 
$\zer\sad\neq\emp$. Set $\alpha=\min\{
\alpha_i^{\CC},\beta_k^{\CC},\delta_k^{\CC}\}_{i\in I,k\in K}$,
let $\boldsymbol{\mathsf{x}}_0\in\XXX$,
let $\varepsilon\in\zeroun$, and iterate
\begin{equation}
\label{e:1306}
\begin{array}{l}
\text{for}\;n=0,1,\ldots\\
\left\lfloor
\begin{array}{l}
(\boldsymbol{\mathsf{p}}_n,\boldsymbol{\mathsf{p}}_n^*)
\in\gra\boldsymbol{\mathsf{M}};\;
\boldsymbol{\mathsf{q}}_n\in\XXX;\\
\boldsymbol{\mathsf{t}}_n^*=\boldsymbol{\mathsf{p}}_n^*
+\boldsymbol{\mathsf{C}}\boldsymbol{\mathsf{q}}_n;\\
\Delta_n=\scal{\boldsymbol{\mathsf{x}}_n
-\boldsymbol{\mathsf{p}}_n}{\boldsymbol{\mathsf{t}}_n^*}
-(4\alpha)^{-1}\|\boldsymbol{\mathsf{p}}_n
-\boldsymbol{\mathsf{q}}_n\|^2;\\
\text{if}\;\Delta_n>0\\
\left\lfloor
\begin{array}{l}
\lambda_n\in\left[\varepsilon,2-\varepsilon\right];\\
\boldsymbol{\mathsf{x}}_{n+1}=
\boldsymbol{\mathsf{x}}_n
-(\lambda_n\Delta_n/\|\boldsymbol{\mathsf{t}}_n^*\|^2)
\,\boldsymbol{\mathsf{t}}_n^*;
\end{array}
\right.\\
\text{else}\\
\left\lfloor
\begin{array}{l}
\boldsymbol{\mathsf{x}}_{n+1}=\boldsymbol{\mathsf{x}}_n.
\end{array}
\right.\\[1mm]
\end{array}
\right.
\end{array}
\end{equation}
Then the following hold:
\begin{enumerate}
\item
\label{p:3404i}
$(\forall\boldsymbol{\mathsf{z}}\in\zer\sad)(\forall n\in\NN)$
$\|\boldsymbol{\mathsf{x}}_{n+1}-\boldsymbol{\mathsf{z}}\|
\leq\|\boldsymbol{\mathsf{x}}_n-\boldsymbol{\mathsf{z}}\|$.
\item
\label{p:3404ii}
$\sum_{n\in\NN}\|\boldsymbol{\mathsf{x}}_{n+1}
-\boldsymbol{\mathsf{x}}_n\|^2<\pinf$.
\item
\label{p:3404iii}
Suppose that $(\boldsymbol{\mathsf{t}}_n^*)_{n\in\NN}$ is bounded.
Then $\varlimsup\Delta_n\leq 0$.
\item
\label{p:3404iv}
Suppose that
$\boldsymbol{\mathsf{x}}_n-\boldsymbol{\mathsf{p}}_n\weakly
\boldsymbol{\mathsf{0}}$, $\boldsymbol{\mathsf{p}}_n
-\boldsymbol{\mathsf{q}}_n\to\boldsymbol{\mathsf{0}}$,
and $\boldsymbol{\mathsf{t}}_n^*\to\boldsymbol{\mathsf{0}}$.
Then $(\boldsymbol{\mathsf{x}}_n)_{n\in\NN}$
converges weakly to a point in $\zer\sad$.
\end{enumerate}
\end{proposition}
\begin{proof}
\ref{p:3404i}\&\ref{p:3404ii}:
Proposition~\ref{p:6}\ref{p:6i+} and our assumption ensure
that $\zer\sad$ is a nonempty closed convex subset of $\XXX$.
Now, for every $n\in\NN$, set $\eta_n=
(4\alpha)^{-1}\|\boldsymbol{\mathsf{p}}_n
-\boldsymbol{\mathsf{q}}_n\|^2
+\scal{\boldsymbol{\mathsf{p}}_n}{\boldsymbol{\mathsf{t}}_n^*}$
and $\boldsymbol{\mathsf{H}}_n=
\menge{\boldsymbol{\mathsf{x}}\in\XXX}{
\scal{\boldsymbol{\mathsf{x}}}{
\boldsymbol{\mathsf{t}}_n^*}\leq\eta_n}$.
On the one hand, according to Proposition~\ref{p:2}\ref{p:2iib},
$(\forall n\in\NN)$ $\zer\sad\subset\boldsymbol{\mathsf{H}}_n$.
On the other hand, \eqref{e:1306} gives
$(\forall n\in\NN)$ $\Delta_n=\scal{\boldsymbol{\mathsf{x}}_n}{
\boldsymbol{\mathsf{t}}_n^*}-\eta_n$.
Altogether, \eqref{e:1306} is an instantiation of \eqref{e:1972}.
The claims thus follow from
Lemma~\ref{l:yeu}\ref{l:yeui}\&\ref{l:yeuii}.

\ref{p:3404iii}:
Set $\mu=\sup_{n\in\NN}\|\boldsymbol{\mathsf{t}}_n^*\|$.
For every $n\in\NN$, if $\Delta_n>0$, then
\eqref{e:1306} yields $\Delta_n=
\lambda_n^{-1}\|\boldsymbol{\mathsf{t}}_n^*\|\,
\|\boldsymbol{\mathsf{x}}_{n+1}-\boldsymbol{\mathsf{x}}_n\|\leq
\varepsilon^{-1}\mu\|\boldsymbol{\mathsf{x}}_{n+1}
-\boldsymbol{\mathsf{x}}_n\|$;
otherwise, $\Delta_n\leq 0=\varepsilon^{-1}\mu
\|\boldsymbol{\mathsf{x}}_{n+1}-\boldsymbol{\mathsf{x}}_n\|$.
We therefore invoke \ref{p:3404ii} to get
$\varlimsup\Delta_n\leq\lim\varepsilon^{-1}\mu
\|\boldsymbol{\mathsf{x}}_{n+1}-\boldsymbol{\mathsf{x}}_n\|=0$.

\ref{p:3404iv}:
Let $\boldsymbol{\mathsf{x}}\in\XXX$, let $(k_n)_{n\in\NN}$
be a strictly increasing sequence in $\NN$, and suppose that
$\boldsymbol{\mathsf{x}}_{k_n}\weakly\boldsymbol{\mathsf{x}}$.
Then $\boldsymbol{\mathsf{p}}_{k_n}=(\boldsymbol{\mathsf{p}}_{k_n}
-\boldsymbol{\mathsf{x}}_{k_n})+\boldsymbol{\mathsf{x}}_{k_n}
\weakly\boldsymbol{\mathsf{x}}$. In addition,
\eqref{e:1306} and Proposition~\ref{p:2}\ref{p:2i-} imply that
$(\boldsymbol{\mathsf{p}}_{k_n},\boldsymbol{\mathsf{p}}_{k_n}^*
+\boldsymbol{\mathsf{C}}\boldsymbol{\mathsf{p}}_{k_n})_{n\in\NN}$
lies in $\gra(\boldsymbol{\mathsf{M}}+\boldsymbol{\mathsf{C}})
=\gra\sad$. We also note that, since $\boldsymbol{\mathsf{C}}$ is
$(1/\alpha)$-Lipschitzian by
Proposition~\ref{p:2}\ref{p:2iia}, \eqref{e:1306} yields
$\|\boldsymbol{\mathsf{p}}_n^*
+\boldsymbol{\mathsf{C}}\boldsymbol{\mathsf{p}}_n\|
=\|\boldsymbol{\mathsf{t}}_n^*
-\boldsymbol{\mathsf{C}}\boldsymbol{\mathsf{q}}_n
+\boldsymbol{\mathsf{C}}\boldsymbol{\mathsf{p}}_n\|\leq
\|\boldsymbol{\mathsf{t}}_n^*\|
+\|\boldsymbol{\mathsf{C}}\boldsymbol{\mathsf{p}}_n
-\boldsymbol{\mathsf{C}}\boldsymbol{\mathsf{q}}_n\|
\leq\|\boldsymbol{\mathsf{t}}_n^*\|
+\|\boldsymbol{\mathsf{p}}_n
-\boldsymbol{\mathsf{q}}_n\|/\alpha\to 0$.
Altogether, since $\sad$ is maximally monotone
by Proposition~\ref{p:6}\ref{p:6i}, 
\cite[Proposition~20.38(ii)]{Livre1} yields
$\boldsymbol{\mathsf{x}}\in\zer\sad$. In turn,
Lemma~\ref{l:yeu}\ref{l:yeuiii} guarantees that
$(\boldsymbol{\mathsf{x}}_n)_{n\in\NN}$
converges weakly to a point in $\zer\sad$.
\end{proof}

The next outer approximation scheme is a variant of the previous
one that guarantees strong convergence to a specific zero of the
saddle operator.

\begin{proposition}
\label{p:4738}
Consider the setting of Problem~\ref{prob:1} and of
Definitions~\ref{d:S} and \ref{d:CM}, and suppose that 
$\zer\sad\neq\emp$. Define
\begin{multline}
\label{e:6146}
\begin{aligned}
\Xi\colon\RPP\times\RPP\times\RR\times\RR&\to\RR^2\\
(\Delta,\tau,\varsigma,\chi)&\mapsto
\begin{cases}
(1,\Delta/\tau),&\text{if}\;\:\rho=0;\\
\big(0,(\Delta+\chi)/\tau\big),
&\text{if}\;\:\rho\neq 0\;\:\text{and}\;\:
\chi\Delta\geq\rho;\\
\big(1-\chi\Delta/\rho,\varsigma\Delta/\rho\big),
&\text{if}\;\:\rho\neq 0\;\:\text{and}\;\:
\chi\Delta<\rho,
\end{cases}
\end{aligned}
\\
\text{where}\;\:\rho=\tau\varsigma-\chi^2,
\end{multline}
set $\alpha=\min\{\alpha_i^{\CC},\beta_k^{\CC},
\delta_k^{\CC}\}_{i\in I,k\in K}$,
and let $\boldsymbol{\mathsf{x}}_0\in\XXX$. Iterate
\begin{equation}
\label{e:9060}
\begin{array}{l}
\text{for}\;n=0,1,\ldots\\
\left\lfloor
\begin{array}{l}
(\boldsymbol{\mathsf{p}}_n,\boldsymbol{\mathsf{p}}_n^*)\in
\gra\boldsymbol{\mathsf{M}};\;\boldsymbol{\mathsf{q}}_n\in\XXX;\\
\boldsymbol{\mathsf{t}}_n^*=\boldsymbol{\mathsf{p}}_n^*
+\boldsymbol{\mathsf{C}}\boldsymbol{\mathsf{q}}_n;\\
\Delta_n=\scal{\boldsymbol{\mathsf{x}}_n
-\boldsymbol{\mathsf{p}}_n}{
\boldsymbol{\mathsf{t}}_n^*}-(4\alpha)^{-1}
\|\boldsymbol{\mathsf{p}}_n-\boldsymbol{\mathsf{q}}_n\|^2;\\
\text{if}\;\Delta_n>0\\
\left\lfloor
\begin{array}{l}
\tau_n=\|\boldsymbol{\mathsf{t}}_n^*\|^2;\;
\varsigma_n=\|\boldsymbol{\mathsf{x}}_0
-\boldsymbol{\mathsf{x}}_n\|^2;\;
\chi_n=\scal{\boldsymbol{\mathsf{x}}_0
-\boldsymbol{\mathsf{x}}_n}{\boldsymbol{\mathsf{t}}_n^*};\\
(\kappa_n,\lambda_n)=\Xi(\Delta_n,\tau_n,\varsigma_n,\chi_n);\\
\boldsymbol{\mathsf{x}}_{n+1}
=(1-\kappa_n)\boldsymbol{\mathsf{x}}_0
+\kappa_n\boldsymbol{\mathsf{x}}_n
-\lambda_n\boldsymbol{\mathsf{t}}_n^*;
\end{array}
\right.\\
\text{else}\\
\left\lfloor
\begin{array}{l}
\boldsymbol{\mathsf{x}}_{n+1}=\boldsymbol{\mathsf{x}}_n.
\end{array}
\right.\\[1.5mm]
\end{array}
\right.
\end{array}
\end{equation}
Then the following hold:
\begin{enumerate}
\item
\label{p:4738i}
$(\forall n\in\NN)$ $\|\boldsymbol{\mathsf{x}}_n
-\boldsymbol{\mathsf{x}}_0\|
\leq\|\boldsymbol{\mathsf{x}}_{n+1}-\boldsymbol{\mathsf{x}}_0\|
\leq\|\proj_{\zer\sad}\boldsymbol{\mathsf{x}}_0
-\boldsymbol{\mathsf{x}}_0\|$.
\item
\label{p:4738ii}
$\sum_{n\in\NN}\|\boldsymbol{\mathsf{x}}_{n+1}
-\boldsymbol{\mathsf{x}}_n\|^2<\pinf$.
\item
\label{p:4738iii}
Suppose that $(\boldsymbol{\mathsf{t}}_n^*)_{n\in\NN}$ is bounded.
Then $\varlimsup\Delta_n\leq 0$.
\item
\label{p:4738iv}
Suppose that
$\boldsymbol{\mathsf{x}}_n
-\boldsymbol{\mathsf{p}}_n\weakly\boldsymbol{\mathsf{0}}$,
$\boldsymbol{\mathsf{p}}_n
-\boldsymbol{\mathsf{q}}_n\to\boldsymbol{\mathsf{0}}$,
and $\boldsymbol{\mathsf{t}}_n^*\to\boldsymbol{\mathsf{0}}$.
Then $\boldsymbol{\mathsf{x}}_n\to
\proj_{\zer\sad}\boldsymbol{\mathsf{x}}_0$.
\end{enumerate}
\end{proposition}
\begin{proof}
Set $(\forall n\in\NN)$ $\eta_n=(4\alpha)^{-1}
\|\boldsymbol{\mathsf{p}}_n-\boldsymbol{\mathsf{q}}_n\|^2
+\scal{\boldsymbol{\mathsf{p}}_n}{\boldsymbol{\mathsf{t}}_n^*}$
and $\boldsymbol{\mathsf{H}}_n
=\menge{\boldsymbol{\mathsf{x}}\in\XXX}{
\scal{\boldsymbol{\mathsf{x}}}{\boldsymbol{\mathsf{t}}_n^*}
\leq\eta_n}$. As seen in the proof of Proposition~\ref{p:3404},
$\zer\sad$ is a nonempty closed convex subset of $\XXX$
and, for every $n\in\NN$,
$\zer\sad\subset\boldsymbol{\mathsf{H}}_n$ and $\Delta_n
=\scal{\boldsymbol{\mathsf{x}}_n}{\boldsymbol{\mathsf{t}}_n^*}
-\eta_n$.
This and \eqref{e:6146} make \eqref{e:9060}
an instance of \eqref{e:9237}.

\ref{p:4738i}\&\ref{p:4738ii}:
Apply Lemma~\ref{l:manh}\ref{l:manhi}\&\ref{l:manhii}.

\ref{p:4738iii}:
Set $\mu=\sup_{n\in\NN}\|\boldsymbol{\mathsf{t}}_n^*\|$.
Take $n\in\NN$. Suppose that $\Delta_n>0$. Then,
by construction of $\boldsymbol{\mathsf{H}}_n$,
$\proj_{\boldsymbol{\mathsf{H}}_n}\boldsymbol{\mathsf{x}}_n
=\boldsymbol{\mathsf{x}}_n
-(\Delta_n/\|\boldsymbol{\mathsf{t}}_n^*\|^2)
\,\boldsymbol{\mathsf{t}}_n^*$.
This implies that $\Delta_n=\|\boldsymbol{\mathsf{t}}_n^*\|
\,\|\proj_{\boldsymbol{\mathsf{H}}_n}\boldsymbol{\mathsf{x}}_n
-\boldsymbol{\mathsf{x}}_n\|
\leq\mu\|\proj_{\boldsymbol{\mathsf{H}}_n}\boldsymbol{\mathsf{x}}_n
-\boldsymbol{\mathsf{x}}_n\|$. Next, suppose that
$\Delta_n\leq 0$. 
Then $\boldsymbol{\mathsf{x}}_n\in\boldsymbol{\mathsf{H}}_n$
and therefore $\Delta_n\leq
0=\mu\|\proj_{\boldsymbol{\mathsf{H}}_n}\boldsymbol{\mathsf{x}}_n
-\boldsymbol{\mathsf{x}}_n\|$.
Altogether, $(\forall n\in\NN)$
$\Delta_n\leq\mu\|\proj_{\boldsymbol{\mathsf{H}}_n}
\boldsymbol{\mathsf{x}}_n
-\boldsymbol{\mathsf{x}}_n\|$. Consequently,
Lemma~\ref{l:manh}\ref{l:manhii} yields
$\varlimsup\Delta_n\leq 0$.

\ref{p:4738iv}:
Follow the same procedure as in the proof of
Proposition~\ref{p:3404}\ref{p:3404iv},
invoking Lemma~\ref{l:manh}\ref{l:manhiii} instead of
Lemma~\ref{l:yeu}\ref{l:yeuiii}.	
\end{proof}

\section{Asynchronous block-iterative outer approximation methods}
\label{sec:3}

We exploit the saddle
form of Problem~\ref{prob:1} described in Definition~\ref{d:S} to
obtain splitting algorithms with features \ref{f:1}--\ref{f:5}.
Let us comment on the impact of requirements \ref{f:1}--\ref{f:4}.
\begin{itemize} 
\item[\ref{f:1}] 
For every $i\in I$ and every $k\in K$, 
each single-valued operator $C_i$, $Q_i$, $R_i$, $B_k^{\CC}$,
$B_k^{\LL}$, $D_k^{\CC}$, $D_k^{\LL}$, and $L_{ki}$ must be
activated individually via a forward step, whereas each of the
set-valued operators $A_i$, $B_k^{\MM}$, and $D_k^{\MM}$ must be
activated individually via a backward resolvent step. 
\item[\ref{f:2}] 
At iteration $n$, only
operators indexed by subgroups $I_n\subset I$ and $K_n\subset K$ of
indices need to be involved in the sense that the results of their
evaluations are incorporated. This considerably reduces the
computational load compared to standard methods, which require the
use of all the operators at every iteration. Assumption~\ref{a:1}
below regulates the frequency at which the indices should be chosen
over time.
\item[\ref{f:3}]
When an operator is involved at iteration $n$, its evaluation can
be made at a point based on data available at an earlier
iteration. This makes it possible to initiate a computation at a
given iteration and incorporate its result at a later time.
Assumption~\ref{a:3} below controls the lag allowed in 
the process of using past data.
\item[\ref{f:4}]
Assumption~\ref{a:2} below describes the range allowed for the
various scaling parameters in terms of the cocoercivity and
Lipschitz constants of the operators.
\end{itemize}

\begin{assumption}
\label{a:2}
In the setting of Problem~\ref{prob:1}, set
$\alpha=\min\{\alpha_i^{\CC},\beta_k^{\CC},
\delta_k^{\CC}\}_{i\in I,k\in K}$,
let $\sigma\in\RPP$ and $\varepsilon\in\zeroun$ be such that
\begin{equation}
\label{e:a2i'}
\sigma>1/(4\alpha)\quad\text{and}\quad
1/\varepsilon>\max\big\{\alpha_i^{\LL}+\chi+\sigma,
\beta_k^{\LL}+\sigma,\delta_k^{\LL}+\sigma
\big\}_{i\in I, k\in K},
\end{equation}
and suppose that the following are satisfied:
\begin{enumerate}[label={\rm[\alph*]}]
\item
\label{a:2ii}
For every $i\in I$ and every $n\in\NN$, $\gamma_{i,n}\in
\left[\varepsilon,1/(\alpha_i^{\LL}+\chi+\sigma)\right]$.
\item
\label{a:2iii}
For every $k\in K$ and every $n\in\NN$, 
$\mu_{k,n}\in\left[\varepsilon,1/(\beta_k^{\LL}+\sigma)\right]$,
$\nu_{k,n}\in\left[\varepsilon,1/(\delta_k^{\LL}+\sigma)\right]$, 
and $\sigma_{k,n}\in\left[\varepsilon,1/\varepsilon\right]$.
\item
\label{a:2iv}
For every $i\in I$, $x_{i,0}\in\HH_i$; for every $k\in K$,
$\{y_{k,0},z_{k,0},v_{k,0}^*\}\subset\GG_k$.
\end{enumerate}
\end{assumption}

\begin{assumption}
\label{a:1}
$I$ and $K$ are finite sets, $P\in\NN$,
$(I_n)_{n\in\NN}$ are nonempty subsets of $I$,
and $(K_n)_{n\in\NN}$ are nonempty subsets of $K$ such that
\begin{equation}
\label{e:2393}
I_0=I,\quad K_0=K,\quad\text{and}\quad(\forall n\in\NN)\;\;
\bigcup_{j=n}^{n+P}I_j=I\;\:\text{and}\;\:
\bigcup_{j=n}^{n+P}K_j=K.
\end{equation}
\end{assumption}

\begin{assumption}
\label{a:3}
$I$ and $K$ are finite sets, $T\in\NN$, and,
for every $i\in I$ and every $k\in K$, $(\pi_i(n))_{n\in\NN}$
and $(\omega_k(n))_{n\in\NN}$ are sequences in $\NN$ such that
$(\forall n\in\NN)$ $n-T\leq\pi_i(n)\leq n$ and
$n-T\leq\omega_k(n)\leq n$.
\end{assumption}

Our first algorithm is patterned after the abstract
geometric outer approximation principle described in
Proposition~\ref{p:3404}. As before, bold letters denote product
space elements, e.g., $\boldsymbol{x}_n=(x_{i,n})_{i\in I}\in\HHH$.

\begin{algorithm}
\label{algo:1}
Consider the setting of Problem~\ref{prob:1} and suppose that
Assumption~\ref{a:2}--\ref{a:3} is in force. Let
$(\lambda_n)_{n\in\NN}$ be a sequence in
$\left[\varepsilon,2-\varepsilon\right]$ and iterate
\begin{equation}
\label{e:long1}
\begin{array}{l}
\text{for}\;n=0,1,\ldots\\
\left\lfloor
\begin{array}{l}
\text{for every}\;i\in I_n\\
\left\lfloor
\begin{array}{l}
l_{i,n}^*=Q_ix_{i,\pi_i(n)}
+R_i\boldsymbol{x}_{\pi_i(n)}
+\sum_{k\in K}L_{ki}^*v_{k,\pi_i(n)}^*;\\
a_{i,n}=J_{\gamma_{i,\pi_i(n)}A_i}\big(
x_{i,\pi_i(n)}+\gamma_{i,\pi_i(n)}(s_i^*-l_{i,n}^*
-C_ix_{i,\pi_i(n)})\big);\\
a_{i,n}^*=\gamma_{i,\pi_i(n)}^{-1}(x_{i,\pi_i(n)}
-a_{i,n})-l_{i,n}^*+Q_ia_{i,n};\\
\xi_{i,n}=\|a_{i,n}-x_{i,\pi_i(n)}\|^2;
\end{array}
\right.\\
\text{for every}\;i\in I\smallsetminus I_n\\
\left\lfloor
\begin{array}{l}
a_{i,n}=a_{i,n-1};\;a_{i,n}^*=a_{i,n-1}^*;\;
\xi_{i,n}=\xi_{i,n-1};\\
\end{array}
\right.\\
\text{for every}\;k\in K_n\\
\left\lfloor
\begin{array}{l}
u_{k,n}^*=v_{k,\omega_k(n)}^*-B_k^{\LL}y_{k,\omega_k(n)};\\
w_{k,n}^*=v_{k,\omega_k(n)}^*-D_k^{\LL}z_{k,\omega_k(n)};\\
b_{k,n}=J_{\mu_{k,\omega_k(n)}B_k^{\MM}}\big(y_{k,\omega_k(n)}
+\mu_{k,\omega_k(n)}(u_{k,n}^*-B_k^{\CC}y_{k,\omega_k(n)})
\big);\\
d_{k,n}=J_{\nu_{k,\omega_k(n)}D_k^{\MM}}\big(z_{k,\omega_k(n)}
+\nu_{k,\omega_k(n)}(w_{k,n}^*-D_k^{\CC}z_{k,\omega_k(n)})\big);\\
e_{k,n}^*=\sigma_{k,\omega_k(n)}\big(
\sum_{i\in I}L_{ki}x_{i,\omega_k(n)}
-y_{k,\omega_k(n)}-z_{k,\omega_k(n)}-r_k\big)+v_{k,\omega_k(n)}^*
;\\
q_{k,n}^*=\mu_{k,\omega_k(n)}^{-1}(y_{k,\omega_k(n)}-b_{k,n})
+u_{k,n}^*+B_k^{\LL}b_{k,n}-e_{k,n}^*;\\
t_{k,n}^*=\nu_{k,\omega_k(n)}^{-1}(z_{k,\omega_k(n)}-d_{k,n})
+w_{k,n}^*+D_k^{\LL}d_{k,n}-e_{k,n}^*;\\
\eta_{k,n}=\|b_{k,n}-y_{k,\omega_k(n)}\|^2
+\|d_{k,n}-z_{k,\omega_k(n)}\|^2;\\
e_{k,n}=r_k+b_{k,n}+d_{k,n}-\sum_{i\in I}L_{ki}a_{i,n};
\end{array}
\right.\\
\text{for every}\;k\in K\smallsetminus K_n\\
\left\lfloor
\begin{array}{l}
b_{k,n}=b_{k,n-1};\;
d_{k,n}=d_{k,n-1};\;
e_{k,n}^*=e_{k,n-1}^*;\;
q_{k,n}^*=q_{k,n-1}^*;\;
t_{k,n}^*=t_{k,n-1}^*;\\
\eta_{k,n}=\eta_{k,n-1};\;
e_{k,n}=r_k+b_{k,n}+d_{k,n}-\sum_{i\in I}L_{ki}a_{i,n};
\end{array}
\right.\\
\text{for every}\;i\in I\\
\left\lfloor
\begin{array}{l}
p_{i,n}^*=a_{i,n}^*
+R_i\boldsymbol{a}_n
+\sum_{k\in K}L_{ki}^*e_{k,n}^*;
\end{array}
\right.\\
\begin{aligned}
\Delta_n&=\textstyle
{-}(4\alpha)^{-1}\big(\sum_{i\in I}\xi_{i,n}
+\sum_{k\in K}\eta_{k,n}\big)
+\sum_{i\in I}\scal{x_{i,n}-a_{i,n}}{p_{i,n}^*}\\
&\textstyle
\quad\;+\sum_{k\in K}\big(\scal{y_{k,n}-b_{k,n}}{q_{k,n}^*}
+\scal{z_{k,n}-d_{k,n}}{t_{k,n}^*}
+\scal{e_{k,n}}{v_{k,n}^*-e_{k,n}^*}\big);
\end{aligned}\\
\text{if}\;\Delta_n>0\\
\left\lfloor
\begin{array}{l}
\theta_n=\lambda_n\Delta_n/
\big(\sum_{i\in I}\|p_{i,n}^*\|^2+\sum_{k\in K}\big(
\|q_{k,n}^*\|^2+\|t_{k,n}^*\|^2+\|e_{k,n}\|^2\big)\big);\\
\text{for every}\;i\in I\\
\left\lfloor
\begin{array}{l}
x_{i,n+1}=x_{i,n}-\theta_np_{i,n}^*;
\end{array}
\right.\\
\text{for every}\;k\in K\\
\left\lfloor
\begin{array}{l}
y_{k,n+1}=y_{k,n}-\theta_nq_{k,n}^*;\;
z_{k,n+1}=z_{k,n}-\theta_nt_{k,n}^*;\;
v_{k,n+1}^*=v_{k,n}^*-\theta_ne_{k,n};
\end{array}
\right.\\[1mm]
\end{array}
\right.\\
\text{else}\\
\left\lfloor
\begin{array}{l}
\text{for every}\;i\in I\\
\left\lfloor
\begin{array}{l}
x_{i,n+1}=x_{i,n};
\end{array}
\right.\\
\text{for every}\;k\in K\\
\left\lfloor
\begin{array}{l}
y_{k,n+1}=y_{k,n};\;z_{k,n+1}=z_{k,n};\;v_{k,n+1}^*=v_{k,n}^*.
\end{array}
\right.\\[1mm]
\end{array}
\right.\\[9.5mm]
\end{array}
\right.
\end{array}
\end{equation}
\end{algorithm}

The convergence properties of Algorithm~\ref{algo:1} are 
laid out in the following theorem.

\newpage
\begin{theorem}
\label{t:1}
Consider the setting of Algorithm~\ref{algo:1} and suppose that
the dual problem \eqref{e:1d} has a solution. Then the following 
hold:
\begin{enumerate}
\item
\label{t:1i}
Let $i\in I$. Then $\sum_{n\in\NN}\|x_{i,n+1}-x_{i,n}\|^2<\pinf$.
\item
\label{t:1ii}
Let $k\in K$. Then
$\sum_{n\in\NN}\|y_{k,n+1}-y_{k,n}\|^2<\pinf$,
$\sum_{n\in\NN}\|z_{k,n+1}-z_{k,n}\|^2<\pinf$, and
$\sum_{n\in\NN}\|v_{k,n+1}^*-v_{k,n}^*\|^2<\pinf$.
\item
\label{t:1iii-}
Let $i\in I$ and $k\in K$. Then
$x_{i,n}-a_{i,n}\to 0$, $y_{k,n}-b_{k,n}\to 0$,
$z_{k,n}-d_{k,n}\to 0$, and $v_{k,n}^*-e_{k,n}^*\to 0$.
\item
\label{t:1iii}
There exist a solution $\overline{\boldsymbol{x}}$ to
\eqref{e:1p} and a solution $\overline{\boldsymbol{v}}^*$ to
\eqref{e:1d} such that, for every $i\in I$
and every $k\in K$, $x_{i,n}\weakly\overline{x}_i$,
$a_{i,n}\weakly\overline{x}_i$,
and $v_{k,n}^*\weakly\overline{v}_k^*$.
In addition,
$(\overline{\boldsymbol{x}},\overline{\boldsymbol{v}}^*)$ is a
Kuhn--Tucker point of Problem~\ref{prob:1} in the sense of
\eqref{e:Z}.
\end{enumerate}
\end{theorem}
\begin{proof}
We use the notation of Definitions~\ref{d:S} and \ref{d:CM}. 
We first observe that $\zer\sad\neq\emp$ by virtue of
Proposition~\ref{p:6}\ref{p:6iv}. Next, let us verify that
\eqref{e:long1} is a special case of \eqref{e:1306}.
For every $i\in I$, denote by $\overline{\vartheta}_i(n)$
the most recent iteration preceding an iteration $n$ at which the
results of the evaluations of the operators
$A_i$, $C_i$, $Q_i$, and $R_i$ were incorporated,
and by $\vartheta_i(n)$ the iteration at which the
corresponding calculations were initiated, i.e.,
\begin{equation}
\label{e:9214}
\overline{\vartheta}_i(n)=\max\menge{j\in\NN}{j\leq n
\;\:\text{and}\;\:i\in I_j}\quad\text{and}\quad
\vartheta_i(n)=\pi_i\big(\overline{\vartheta}_i(n)\big).
\end{equation}
Similarly, we define
\begin{equation}
\label{e:3372}
(\forall k\in K)(\forall n\in\NN)\quad
\overline{\varrho}_k(n)=\max\menge{j\in\NN}{j\leq n
\;\:\text{and}\;\:k\in K_j}\quad\text{and}\quad
\varrho_k(n)=\omega_k\big(\overline{\varrho}_k(n)\big).
\end{equation}
By virtue of \eqref{e:long1},
\begin{equation}
\label{e:4173}
(\forall i\in I)(\forall n\in\NN)\quad
a_{i,n}=a_{i,\overline{\vartheta}_i(n)},\quad
a_{i,n}^*=a_{i,\overline{\vartheta}_i(n)}^*,\quad
\xi_{i,n}=\xi_{i,\overline{\vartheta}_i(n)},
\end{equation}
and likewise
\begin{equation}
\label{e:6609}
(\forall k\in K)(\forall n\in\NN)\quad
\begin{cases}
b_{k,n}=b_{k,\overline{\varrho}_k(n)},\;\:
d_{k,n}=d_{k,\overline{\varrho}_k(n)},\;\:
\eta_{k,n}=\eta_{k,\overline{\varrho}_k(n)}\\
e_{k,n}^*=e_{k,\overline{\varrho}_k(n)}^*,\;\:
q_{k,n}^*=q_{k,\overline{\varrho}_k(n)}^*,\;\:
t_{k,n}^*=t_{k,\overline{\varrho}_k(n)}^*.
\end{cases}
\end{equation}
To proceed further, set
\begin{equation}
\label{e:8870}
(\forall n\in\NN)\quad
\begin{cases}
\boldsymbol{\mathsf{x}}_n=(\boldsymbol{x}_n,
\boldsymbol{y}_n,\boldsymbol{z}_n,\boldsymbol{v}_n^*)\\
\boldsymbol{\mathsf{p}}_n=(\boldsymbol{a}_n,
\boldsymbol{b}_n,\boldsymbol{d}_n,\boldsymbol{e}_n^*)\\
\boldsymbol{\mathsf{p}}_n^*=\big(
\boldsymbol{p}_n^*-(C_ix_{i,\vartheta_i(n)})_{i\in I},
\boldsymbol{q}_n^*-(B_k^{\CC}y_{k,\varrho_k(n)})_{k\in K},
\boldsymbol{t}_n^*-(D_k^{\CC}z_{k,\varrho_k(n)})_{k\in K},
\boldsymbol{e}_n\big)\\
\boldsymbol{\mathsf{q}}_n=
\big((x_{i,\vartheta_i(n)})_{i\in I},
(y_{k,\varrho_k(n)})_{k\in K},(z_{k,\varrho_k(n)})_{k\in K},
(e_{k,n}^*)_{k\in K}\big)\\
\boldsymbol{\mathsf{t}}_n^*=(
\boldsymbol{p}_n^*,\boldsymbol{q}_n^*,\boldsymbol{t}_n^*,
\boldsymbol{e}_n).
\end{cases}
\end{equation}
For every $i\in I$ and every $n\in\NN$,
it follows from \eqref{e:4173}, \eqref{e:9214}, \eqref{e:long1},
and \cite[Proposition~23.2(ii)]{Livre1} that
\begin{align}
\label{e:6869}
a_{i,n}^*-C_ix_{i,\vartheta_i(n)}
&=a_{i,\overline{\vartheta}_i(n)}^*
-C_ix_{i,\pi_i(\overline{\vartheta}_i(n))}
\nonumber\\
&=\gamma_{i,\pi_i(\overline{\vartheta}_i(n))}^{-1}
\big(x_{i,\pi_i(\overline{\vartheta}_i(n))}
-a_{i,\overline{\vartheta}_i(n)}\big)
-l_{i,\overline{\vartheta}_i(n)}^*
-C_ix_{i,\pi_i(\overline{\vartheta}_i(n))}
+Q_ia_{i,\overline{\vartheta}_i(n)}
\nonumber\\
&\in{-}s_i^*+A_ia_{i,\overline{\vartheta}_i(n)}
+Q_ia_{i,\overline{\vartheta}_i(n)}
\nonumber\\
&={-}s_i^*+A_ia_{i,n}+Q_ia_{i,n}
\end{align}
and, therefore, that
\begin{align}
\label{e:9073}
p_{i,n}^*-C_ix_{i,\vartheta_i(n)}
&=a_{i,n}^*-C_ix_{i,\vartheta_i(n)}
+R_i\boldsymbol{a}_n
+\sum_{k\in K}L_{ki}^*e_{k,n}^*
\nonumber\\
&\in{-}s_i^*+A_ia_{i,n}+Q_ia_{i,n}
+R_i\boldsymbol{a}_n+\sum_{k\in K}L_{ki}^*e_{k,n}^*.
\end{align}
Analogously, we invoke
\eqref{e:6609}, \eqref{e:3372}, and \eqref{e:long1} to obtain
\begin{equation}
\label{e:3549}
(\forall k\in K)(\forall n\in\NN)\quad
q_{k,n}^*-B_k^{\CC}y_{k,\varrho_k(n)}
\in B_k^{\MM}b_{k,n}+B_k^{\LL}b_{k,n}-e_{k,n}^*
\end{equation}
and
\begin{equation}
\label{e:3548}
(\forall k\in K)(\forall n\in\NN)\quad
t_{k,n}^*-D_k^{\CC}z_{k,\varrho_k(n)}
\in D_k^{\MM}d_{k,n}+D_k^{\LL}d_{k,n}-e_{k,n}^*.
\end{equation}
In addition, \eqref{e:long1} states that
\begin{equation}
\label{e:6107}
(\forall k\in K)(\forall n\in\NN)\quad
e_{k,n}=r_k+b_{k,n}+d_{k,n}-\sum_{i\in I}L_{ki}a_{i,n}.
\end{equation}
Hence, using \eqref{e:8870} and \eqref{e:1179}, we deduce that
$(\boldsymbol{\mathsf{p}}_n,
\boldsymbol{\mathsf{p}}_n^*)_{n\in\NN}$ lies in
$\gra\boldsymbol{\mathsf{M}}$.
Next, it results from \eqref{e:8870} and \eqref{e:2796}
that $(\forall n\in\NN)$
$\boldsymbol{\mathsf{t}}_n^*=\boldsymbol{\mathsf{p}}_n^*
+\boldsymbol{\mathsf{C}}\boldsymbol{\mathsf{q}}_n$. Moreover,
for every $n\in\NN$, \eqref{e:long1}--\eqref{e:8870} entail that
\begin{align}
\label{e:4376}
&\Sum_{i\in I}\xi_{i,n}
+\Sum_{k\in K}\eta_{k,n}
\nonumber\\
&\hspace{6mm}
=\Sum_{i\in I}\xi_{i,\overline{\vartheta}_i(n)}
+\Sum_{k\in K}\eta_{k,\overline{\varrho}_k(n)}
\nonumber\\
&\hspace{6mm}
=\Sum_{i\in I}\big\|a_{i,\overline{\vartheta}_i(n)}
-x_{i,\pi_i(\overline{\vartheta}_i(n))}\big\|^2
+\Sum_{k\in K}\Big(\big\|b_{k,\overline{\varrho}_k(n)}
-y_{k,\omega_k(\overline{\varrho}_k(n))}\big\|^2
+\big\|d_{k,\overline{\varrho}_k(n)}
-z_{k,\omega_k(\overline{\varrho}_k(n))}\big\|^2\Big)
\nonumber\\
&\hspace{6mm}
=\Sum_{i\in I}\big\|a_{i,n}-x_{i,\vartheta_i(n)}\big\|^2
+\Sum_{k\in K}\Big(\big\|b_{k,n}-y_{k,\varrho_k(n)}\big\|^2
+\big\|d_{k,n}-z_{k,\varrho_k(n)}\big\|^2\Big)
\nonumber\\
&\hspace{6mm}
=\|\boldsymbol{\mathsf p}_n-\boldsymbol{\mathsf q}_n\|^2
\end{align}
and, in turn, that
\begin{equation}
\label{e:8350}
\Delta_n=\scal{\boldsymbol{\mathsf x}_n
-\boldsymbol{\mathsf p}_n}{
\boldsymbol{\mathsf t}_n^*}-(4\alpha)^{-1}\|\boldsymbol{\mathsf p}_n
-\boldsymbol{\mathsf q}_n\|^2.
\end{equation}
To sum up, \eqref{e:long1} is an instantiation of \eqref{e:1306}.
Therefore, Proposition~\ref{p:3404}\ref{p:3404ii} asserts that
\begin{equation}
\label{e:4932}
\sum_{n\in\NN}\|\boldsymbol{\mathsf x}_{n+1}
-\boldsymbol{\mathsf x}_n\|^2<\pinf.
\end{equation}

\ref{t:1i}\&\ref{t:1ii}:
These follow from \eqref{e:4932} and \eqref{e:8870}.

\ref{t:1iii-}\&\ref{t:1iii}:
Proposition~\ref{p:3404}\ref{p:3404i} implies that
$(\boldsymbol{\mathsf{x}}_n)_{n\in\NN}$ is bounded.
It therefore results from \eqref{e:8870} that
\begin{equation}
\label{e:3218}
(\boldsymbol{x}_n)_{n\in\NN},\;\:
(\boldsymbol{y}_n)_{n\in\NN},\;\:
(\boldsymbol{z}_n)_{n\in\NN},\;\:\text{and}\;\:
(\boldsymbol{v}_n^*)_{n\in\NN}\;\:\text{are bounded}.
\end{equation}
Hence, \eqref{e:6609}, \eqref{e:long1}, \eqref{e:3372},
and Assumption~\ref{a:2}\ref{a:2iii} ensure that
\begin{equation}
\label{e:2095}
(\forall k\in K)\quad
(e_{k,n}^*)_{n\in\NN}=\Bigg(
\sigma_{k,\varrho_k(n)}\bigg(\sum_{i\in I}L_{ki}x_{i,\varrho_k(n)}
-y_{k,\varrho_k(n)}-z_{k,\varrho_k(n)}-r_k\bigg)
+v_{k,\varrho_k(n)}^*\Bigg)_{n\in\NN}\;\:\text{is bounded}.
\end{equation}
Next, we deduce from \eqref{e:3218} and 
Problem~\ref{prob:1}\ref{prob:1e} that
\begin{equation}
(\forall i\in I)\quad\big(
R_i\boldsymbol{x}_{\vartheta_i(n)}\big)_{n\in\NN}\;\:
\text{is bounded}.
\end{equation}
In turn, it follows from \eqref{e:long1}, \eqref{e:3218},
the fact that $(Q_i)_{i\in I}$ and
$(C_i)_{i\in I}$ are Lipschitzian,
and Assumption~\ref{a:2}\ref{a:2ii} that
\begin{equation}
(\forall i\in I)\quad
\Big(x_{i,\vartheta_i(n)}+\gamma_{i,\vartheta_i(n)}
\big(s_i^*-l_{i,\overline{\vartheta}_i(n)}^*
-C_ix_{i,\vartheta_i(n)}\big)\Big)_{n\in\NN}
\;\:\text{is bounded}.
\end{equation}
An inspection of \eqref{e:4173}, \eqref{e:long1},
\eqref{e:9214}, and Lemma~\ref{l:3945} reveals that
\begin{equation}
\label{e:7947}
(\forall i\in I)\quad(a_{i,n})_{n\in\NN}=
\Big(J_{\gamma_{i,\vartheta_i(n)}A_i}
\big(x_{i,\vartheta_i(n)}+\gamma_{i,\vartheta_i(n)}\big(
s_i^*-l_{i,\overline{\vartheta}_i(n)}^*
-C_ix_{i,\vartheta_i(n)}\big)\big)\Big)_{n\in\NN}
\;\:\text{is bounded}.
\end{equation}
Hence, we infer from \eqref{e:4173}, \eqref{e:long1},
\eqref{e:3218}, and Assumption~\ref{a:2}\ref{a:2ii} that
\begin{equation}
\label{e:1633}
(\forall i\in I)\quad(a_{i,n}^*)_{n\in\NN}\;\:
\text{is bounded}.
\end{equation}
Accordingly, by \eqref{e:long1}, \eqref{e:3218},
and Assumption~\ref{a:2}\ref{a:2iii}, 
\begin{equation}
(\forall k\in K)\quad
\Big(y_{k,\varrho_k(n)}+\mu_{k,\varrho_k(n)}\big(
u_{k,\overline{\varrho}_k(n)}^*-B_k^{\CC}y_{k,\varrho_k(n)}\big)
\Big)_{n\in\NN}\;\:\text{is bounded}.
\end{equation}
Therefore, \eqref{e:6609}, \eqref{e:long1}, \eqref{e:3372},
and Lemma~\ref{l:3945} imply that
\begin{equation}
\label{e:2917}
(\forall k\in K)\quad
(b_{k,n})_{n\in\NN}=\Big(J_{\mu_{k,\varrho_k(n)}B_k^{\MM}}\big(
y_{k,\varrho_k(n)}+\mu_{k,\varrho_k(n)}\big(
u_{k,\overline{\varrho}_k(n)}^*-B_k^{\CC}y_{k,\varrho_k(n)}\big)
\big)\Big)_{n\in\NN}\;\:\text{is bounded}.
\end{equation}
Thus, \eqref{e:6609}, \eqref{e:long1}, \eqref{e:3218},
\eqref{e:2095}, and Assumption~\ref{a:2}\ref{a:2iii} yield
\begin{equation}
\label{e:9429}
(\boldsymbol{q}_n^*)_{n\in\NN}\;\:\text{is bounded}.
\end{equation}
Likewise,
\begin{equation}
\label{e:3619}
(\boldsymbol{d}_n)_{n\in\NN}\;\:\text{and}\;\:
(\boldsymbol{t}_n^*)_{n\in\NN}\;\:\text{are bounded}.
\end{equation}
We deduce from \eqref{e:6107}, \eqref{e:2917}, \eqref{e:3619},
and \eqref{e:7947} that
\begin{equation}
\label{e:7878}
(\boldsymbol{e}_n)_{n\in\NN}\;\:\text{is bounded}.
\end{equation}
On the other hand, \eqref{e:long1}, \eqref{e:1633},
\eqref{e:7947}, Problem~\ref{prob:1}\ref{prob:1e}, 
and \eqref{e:2095} imply that
\begin{equation}
(\boldsymbol{p}_n^*)_{n\in\NN}\;\:\text{is bounded}.
\end{equation}
Hence, we infer from \eqref{e:8870} and
\eqref{e:9429}--\eqref{e:7878} that
$(\boldsymbol{\mathsf{t}}_n^*)_{n\in\NN}$ is bounded.
Consequently, \eqref{e:8350} and
Proposition~\ref{p:3404}\ref{p:3404iii} yield
\begin{equation}
\label{e:9758}
\varlimsup\big(
\scal{\boldsymbol{\mathsf{x}}_n
-\boldsymbol{\mathsf{p}}_n}{\boldsymbol{\mathsf{t}}_n^*}
-(4\alpha)^{-1}\|\boldsymbol{\mathsf{p}}_n
-\boldsymbol{\mathsf{q}}_n\|^2\big)
=\varlimsup\Delta_n\leq 0.
\end{equation}
Let $\boldsymbol{L}$ and $\boldsymbol{\mathsf{W}}$
be as in \eqref{e:8757} and \eqref{e:4425}.
For every $n\in\NN$, set
\begin{equation}
\label{e:3958}
\begin{cases}
(\forall i\in I)\;\;
E_{i,n}=\gamma_{i,\vartheta_i(n)}^{-1}\Id-Q_i\\
(\forall k\in K)\;\;
F_{k,n}=\mu_{k,\varrho_k(n)}^{-1}\Id-B_k^{\LL},\;\:
G_{k,n}=\nu_{k,\varrho_k(n)}^{-1}\Id-D_k^{\LL}\\
\boldsymbol{\mathsf{E}}_n\colon\XXX\to\XXX\colon
(\boldsymbol{x},\boldsymbol{y},\boldsymbol{z},\boldsymbol{v}^*)
\mapsto
\big((E_{i,n}x_i)_{i\in I},(F_{k,n}y_k)_{k\in K},
(G_{k,n}z_k)_{k\in K},(\sigma_{k,\varrho_k(n)}^{-1}
v_k^*)_{k\in K}\big)
\end{cases}
\end{equation}
and
\begin{equation}
\label{e:3514}
\begin{cases}
\widetilde{\boldsymbol{\mathsf x}}_n=
\big((x_{i,\vartheta_i(n)})_{i\in I},
(y_{k,\varrho_k(n)})_{k\in K},(z_{k,\varrho_k(n)})_{k\in K},
(v_{k,\varrho_k(n)}^*)_{k\in K}\big)\\
\boldsymbol{\mathsf{v}}_n^*=
\boldsymbol{\mathsf{E}}_n\boldsymbol{\mathsf{x}}_n
-\boldsymbol{\mathsf{E}}_n\boldsymbol{\mathsf{p}}_n,\;\:
\boldsymbol{\mathsf{w}}_n^*=
\boldsymbol{\mathsf{W}}\boldsymbol{\mathsf{p}}_n
-\boldsymbol{\mathsf{W}}\boldsymbol{\mathsf{x}}_n\\
\boldsymbol{\mathsf{r}}_n^*
=\big((R_i\boldsymbol{a}_n
-R_i\boldsymbol{x}_n)_{i\in I},
\boldsymbol{0},\boldsymbol{0},\boldsymbol{0}\big),\;\:
\widetilde{\boldsymbol{\mathsf{r}}}_n^*
=\big((R_i\boldsymbol{a}_n
-R_i\boldsymbol{x}_{\vartheta_i(n)})_{i\in I},
\boldsymbol{0},\boldsymbol{0},\boldsymbol{0}\big)\\
\boldsymbol{\mathsf l}_n^*=
\Big(
\big({-}\sum_{k\in K}L_{ki}^*v_{k,\vartheta_i(n)}^*\big)_{i\in I},
\big(v_{k,\varrho_k(n)}^*\big)_{k\in K},
\big(v_{k,\varrho_k(n)}^*\big)_{k\in K},\\
\hspace{60mm}
\big(\sum_{i\in I}L_{ki}x_{i,\varrho_k(n)}
-y_{k,\varrho_k(n)}-z_{k,\varrho_k(n)}\big)_{k\in K}\Big).
\end{cases}
\end{equation}
In view of Problem~\ref{prob:1}\ref{prob:1a}--\ref{prob:1c}
and Assumption~\ref{a:2}\ref{a:2ii}\&\ref{a:2iii},
we deduce from Lemma~\ref{l:2156} that
\begin{equation}
\label{e:6983}
(\forall n\in\NN)\quad
\begin{cases}
\text{the operators}\;\:
(E_{i,n})_{i\in I}\;\:\text{are $(\chi+\sigma)$-strongly monotone}
\\
\text{the operators}\;\:
(F_{k,n})_{k\in K}\:\:\text{and}\;\:(G_{k,n})_{k\in K}\;\:
\text{are $\sigma$-strongly monotone},
\end{cases}
\end{equation}
and from \eqref{e:3958} that there exists $\kappa\in\RPP$
such that
\begin{equation}
\label{e:5297}
\text{the operators}\;\:
(\boldsymbol{\mathsf E}_n)_{n\in\NN}\;\:
\text{are $\kappa$-Lipschitzian}.
\end{equation}
It results from \eqref{e:4173},
\eqref{e:long1}, \eqref{e:9214}, and \eqref{e:3958} that
\begin{align}
\label{e:5004}
(\forall i\in I)(\forall n\in\NN)\quad a_{i,n}^*
&=a_{i,\overline{\vartheta}_i(n)}^*
\nonumber\\
&=\Big(\gamma_{i,\pi_i(\overline{\vartheta}_i(n))}^{-1}
x_{i,\pi_i(\overline{\vartheta}_i(n))}
-Q_ix_{i,\pi_i(\overline{\vartheta}_i(n))}\Big)
-\Big(\gamma_{i,\pi_i(\overline{\vartheta}_i(n))}^{-1}
a_{i,\overline{\vartheta}_i(n)}
-Q_ia_{i,\overline{\vartheta}_i(n)}\Big)
\nonumber\\
&\quad\;
-R_i\boldsymbol{x}_{\pi_i(\overline{\vartheta}_i(n))}
-\sum_{k\in K}L_{ki}^*v_{k,\pi_i(\overline{\vartheta}_i(n))}^*
\nonumber\\
&=E_{i,n}x_{i,\vartheta_i(n)}-E_{i,n}a_{i,n}
-R_i\boldsymbol{x}_{\vartheta_i(n)}
-\Sum_{k\in K}L_{ki}^*v_{k,\vartheta_i(n)}^*
\end{align}
and, therefore, that
\begin{align}
\label{e:8714}
(\forall i\in I)(\forall n\in\NN)\quad
p_{i,n}^*
&=a_{i,n}^*+R_i\boldsymbol{a}_n
+\sum_{k\in K}L_{ki}^*e_{k,n}^*
\nonumber\\
&=E_{i,n}x_{i,\vartheta_i(n)}-E_{i,n}a_{i,n}
+R_i\boldsymbol{a}_n
-R_i\boldsymbol{x}_{\vartheta_i(n)}
-\Sum_{k\in K}L_{ki}^*v_{k,\vartheta_i(n)}^*
+\sum_{k\in K}L_{ki}^*e_{k,n}^*.
\end{align}
At the same time, \eqref{e:6609}, \eqref{e:long1},
\eqref{e:3372}, and \eqref{e:3958} entail that
\begin{align}
(\forall k\in K)(\forall n\in\NN)\quad q_{k,n}^*
&=q_{k,\overline{\varrho}_k(n)}^*
\nonumber\\
&=\Big(\mu_{k,\omega_k(\overline{\varrho}_k(n))}^{-1}
y_{k,\omega_k(\overline{\varrho}_k(n))}
-B_k^{\LL}y_{k,\omega_k(\overline{\varrho}_k(n))}\Big)
\nonumber\\
&\quad\;-\Big(\mu_{k,\omega_k(\overline{\varrho}_k(n))}^{-1}
b_{k,\overline{\varrho}_k(n)}
-B_k^{\LL}b_{k,\overline{\varrho}_k(n)}\Big)
+v_{k,\omega_k(\overline{\varrho}_k(n))}^*
-e_{k,\overline{\varrho}_k(n)}^*
\nonumber\\
&=F_{k,n}y_{k,\varrho_k(n)}-F_{k,n}b_{k,n}+v_{k,\varrho_k(n)}^*
-e_{k,n}^*
\end{align}
and that
\begin{equation}
\label{e:5246}
(\forall k\in K)(\forall n\in\NN)\quad
t_{k,n}^*=G_{k,n}z_{k,\varrho_k(n)}-G_{k,n}d_{k,n}
+v_{k,\varrho_k(n)}^*-e_{k,n}^*.
\end{equation}
Further, we derive from \eqref{e:6609}, \eqref{e:long1},
and \eqref{e:3372} that
\begin{equation}
(\forall k\in K)(\forall n\in\NN)\quad
r_k=\sigma_{k,\varrho_k(n)}^{-1}v_{k,\varrho_k(n)}^*
-\sigma_{k,\varrho_k(n)}^{-1}e_{k,n}^*
-y_{k,\varrho_k(n)}-z_{k,\varrho_k(n)}
+\sum_{i\in I}L_{ki}x_{i,\varrho_k(n)}
\end{equation}
and, in turn, from \eqref{e:6107} that
\begin{multline}
\label{e:2740}
(\forall k\in K)(\forall n\in\NN)\quad
e_{k,n}
=\sigma_{k,\varrho_k(n)}^{-1}v_{k,\varrho_k(n)}^*
-\sigma_{k,\varrho_k(n)}^{-1}e_{k,n}^*
-y_{k,\varrho_k(n)}-z_{k,\varrho_k(n)}
\\
+\sum_{i\in I}L_{ki}x_{i,\varrho_k(n)}
+b_{k,n}+d_{k,n}-\sum_{i\in I}L_{ki}a_{i,n}.
\end{multline}
Altogether, it follows from \eqref{e:8870},
\eqref{e:8714}--\eqref{e:5246}, \eqref{e:2740},
\eqref{e:3958}, \eqref{e:3514},
\eqref{e:4425}, and \eqref{e:3041} that
\begin{equation}
\label{e:2842}
(\forall n\in\NN)\quad\boldsymbol{\mathsf{t}}_n^*=
\boldsymbol{\mathsf{E}}_n\widetilde{\boldsymbol{\mathsf{x}}}_n
-\boldsymbol{\mathsf{E}}_n\boldsymbol{\mathsf{p}}_n
+\widetilde{\boldsymbol{\mathsf{r}}}_n^*
+\boldsymbol{\mathsf{l}}_n^*
+\boldsymbol{\mathsf{W}}\boldsymbol{\mathsf{p}}_n.
\end{equation}
Next, in view of \eqref{e:4932}, \eqref{e:9214}, \eqref{e:3372},
and Assumption~\ref{a:1}--\ref{a:3}, we learn from
Lemma~\ref{l:5368} that
\begin{equation}
\label{e:8706}
(\forall i\in I)(\forall k\in K)\quad
\begin{cases}
\boldsymbol{x}_{\vartheta_i(n)}-\boldsymbol{x}_n\to
\boldsymbol{0},\;\:
\boldsymbol{x}_{\varrho_k(n)}-\boldsymbol{x}_n\to
\boldsymbol{0},\;\:\text{and}\;\:
\boldsymbol{v}_{\vartheta_i(n)}^*-\boldsymbol{v}_n^*\to
\boldsymbol{0}\\
\boldsymbol{y}_{\varrho_k(n)}-\boldsymbol{y}_n\to
\boldsymbol{0},\;\:
\boldsymbol{z}_{\varrho_k(n)}-\boldsymbol{z}_n\to
\boldsymbol{0},\;\:\text{and}\;\:
\boldsymbol{v}_{\varrho_k(n)}^*-\boldsymbol{v}_n^*\to
\boldsymbol{0}.
\end{cases}
\end{equation}
Thus, \eqref{e:3514}, \eqref{e:4425}, \eqref{e:3041},
and \eqref{e:8757} yield
\begin{equation}
\label{e:6047}
\boldsymbol{\mathsf{l}}_n^*
+\boldsymbol{\mathsf{W}}\boldsymbol{\mathsf{x}}_n
\to\boldsymbol{\mathsf{0}},
\end{equation}
while Problem~\ref{prob:1}\ref{prob:1e} gives
\begin{equation}
(\forall i\in I)\quad
\|R_i\boldsymbol{x}_{\vartheta_i(n)}
-R_i\boldsymbol{x}_n\|
\leq\chi\|\boldsymbol{x}_{\vartheta_i(n)}-\boldsymbol{x}_n\|
\to 0.
\end{equation}
On the other hand, we infer from \eqref{e:5297}, \eqref{e:3514},
and \eqref{e:8706} that
\begin{equation}
\label{e:8506}
\|\boldsymbol{\mathsf{E}}_n\widetilde{\boldsymbol{\mathsf{x}}}_n
-\boldsymbol{\mathsf{E}}_n\boldsymbol{\mathsf{x}}_n\|
\leq\kappa\|\widetilde{\boldsymbol{\mathsf{x}}}_n
-\boldsymbol{\mathsf{x}}_n\|\to 0.
\end{equation}
Combining \eqref{e:2842}, \eqref{e:3514}, and
\eqref{e:6047}--\eqref{e:8506}, we obtain
\begin{equation}
\label{e:9860}
\boldsymbol{\mathsf{t}}_n^*-\big(\boldsymbol{\mathsf{v}}_n^*
+\boldsymbol{\mathsf{r}}_n^*
+\boldsymbol{\mathsf{w}}_n^*\big)
=\boldsymbol{\mathsf{l}}_n^*
+\boldsymbol{\mathsf{W}}\boldsymbol{\mathsf{x}}_n
+\boldsymbol{\mathsf{E}}_n\widetilde{\boldsymbol{\mathsf{x}}}_n
-\boldsymbol{\mathsf{E}}_n\boldsymbol{\mathsf{x}}_n
+\widetilde{\boldsymbol{\mathsf{r}}}_n^*
-\boldsymbol{\mathsf{r}}_n^*
\to\boldsymbol{\mathsf{0}}.
\end{equation}
Now set
\begin{equation}
\label{e:qnt}
(\forall n\in\NN)\quad
\widetilde{\boldsymbol{\mathsf{q}}}_n=
(\boldsymbol{x}_n,\boldsymbol{y}_n,\boldsymbol{z}_n,
\boldsymbol{e}_n^*).
\end{equation}
Then $(\widetilde{\boldsymbol{\mathsf{q}}}_n)_{n\in\NN}$
is bounded by virtue of \eqref{e:3218} and
\eqref{e:2095}.
On the one hand, \eqref{e:8870}, \eqref{e:2095},
\eqref{e:7947}, \eqref{e:2917}, and \eqref{e:3619} imply that
$(\boldsymbol{\mathsf{p}}_n)_{n\in\NN}$ is bounded.
On the other hand, \eqref{e:8870} and \eqref{e:8706} give
\begin{equation}
\label{e:5101}
\widetilde{\boldsymbol{\mathsf{q}}}_n
-\boldsymbol{\mathsf{q}}_n\to\boldsymbol{\mathsf{0}}.
\end{equation}
Therefore, appealing to the Cauchy--Schwarz inequality, we obtain
\begin{equation}
\label{e:5202}
\big|\scal{\boldsymbol{\mathsf{p}}_n
-\widetilde{\boldsymbol{\mathsf{q}}}_n
}{\widetilde{\boldsymbol{\mathsf{q}}}_n
-\boldsymbol{\mathsf{q}}_n}\big|
\leq\bigg(\sup_{m\in\NN}\|\boldsymbol{\mathsf{p}}_m\|
+\sup_{m\in\NN}\|\widetilde{\boldsymbol{\mathsf{q}}}_m\|\bigg)
\|\widetilde{\boldsymbol{\mathsf{q}}}_n
-\boldsymbol{\mathsf{q}}_n\|\to 0
\end{equation}
and, by \eqref{e:9860},
\begin{equation}
\label{e:5203}
\big|\scal{\boldsymbol{\mathsf{x}}_n-\boldsymbol{\mathsf{p}}_n}{
\boldsymbol{\mathsf{t}}_n^*
-(\boldsymbol{\mathsf{v}}_n^*+\boldsymbol{\mathsf{r}}_n^*
+\boldsymbol{\mathsf{w}}_n^*)
}\big|\leq\bigg(\sup_{m\in\NN}\|\boldsymbol{\mathsf{x}}_m\|
+\sup_{m\in\NN}\|\boldsymbol{\mathsf{p}}_m\|
\bigg)\|\boldsymbol{\mathsf{t}}_n^*
-(\boldsymbol{\mathsf{v}}_n^*+\boldsymbol{\mathsf{r}}_n^*
+\boldsymbol{\mathsf{w}}_n^*)\|\to 0.
\end{equation}
However, since
$\boldsymbol{\mathsf{W}}^*={-}\boldsymbol{\mathsf{W}}$
by \eqref{e:4425}, it results from \eqref{e:3514} that
$(\forall n\in\NN)$ 
$\scal{\boldsymbol{\mathsf{x}}_n-\boldsymbol{\mathsf{p}}_n}{
\boldsymbol{\mathsf{w}}_n^*}=0$.
Thus, by \eqref{e:9758} and \eqref{e:5101}--\eqref{e:5203},
\begin{align}
\label{e:5224}
0
&\geq\varlimsup\big(
\scal{\boldsymbol{\mathsf{x}}_n
-\boldsymbol{\mathsf{p}}_n}{\boldsymbol{\mathsf{t}}_n^*}
-(4\alpha)^{-1}\|\boldsymbol{\mathsf{p}}_n
-\boldsymbol{\mathsf{q}}_n\|^2\big)
\nonumber\\
&=\varlimsup\big(
\scal{\boldsymbol{\mathsf{x}}_n-\boldsymbol{\mathsf{p}}_n}{
\boldsymbol{\mathsf{v}}_n^*+\boldsymbol{\mathsf{r}}_n^*
+\boldsymbol{\mathsf{w}}_n^*}
+\scal{\boldsymbol{\mathsf{x}}_n-\boldsymbol{\mathsf{p}}_n}{
\boldsymbol{\mathsf{t}}_n^*
-(\boldsymbol{\mathsf{v}}_n^*+\boldsymbol{\mathsf{r}}_n^*
+\boldsymbol{\mathsf{w}}_n^*)
}-(4\alpha)^{-1}\|\boldsymbol{\mathsf{p}}_n
-\boldsymbol{\mathsf{q}}_n\|^2\big)
\nonumber\\
&=\varlimsup\big(
\scal{\boldsymbol{\mathsf{x}}_n-\boldsymbol{\mathsf{p}}_n}{
\boldsymbol{\mathsf{v}}_n^*+\boldsymbol{\mathsf{r}}_n^*}
-(4\alpha)^{-1}\big(
\|\boldsymbol{\mathsf{p}}_n
-\widetilde{\boldsymbol{\mathsf{q}}}_n\|^2
+2\scal{\boldsymbol{\mathsf{p}}_n
-\widetilde{\boldsymbol{\mathsf{q}}}_n
}{\widetilde{\boldsymbol{\mathsf{q}}}_n
-\boldsymbol{\mathsf{q}}_n}
+\|\widetilde{\boldsymbol{\mathsf{q}}}_n
-\boldsymbol{\mathsf{q}}_n\|^2
\big)\big)
\nonumber\\
&=\varlimsup\big(
\scal{\boldsymbol{\mathsf{x}}_n-\boldsymbol{\mathsf{p}}_n}{
\boldsymbol{\mathsf{v}}_n^*+\boldsymbol{\mathsf{r}}_n^*}
-(4\alpha)^{-1}\|\boldsymbol{\mathsf{p}}_n
-\widetilde{\boldsymbol{\mathsf{q}}}_n\|^2\big).
\end{align}
On the other hand, we deduce from \eqref{e:3514}, \eqref{e:8870},
\eqref{e:3958}, \eqref{e:6983}, Assumption~\ref{a:2}\ref{a:2iii},
the Cauchy--Schwarz inequality, Problem~\ref{prob:1}\ref{prob:1e},
and \eqref{e:qnt} that, for every $n\in\NN$,
\begin{align}
\label{e:1234}
&\scal{\boldsymbol{\mathsf{x}}_n-\boldsymbol{\mathsf{p}}_n}{
\boldsymbol{\mathsf{v}}_n^*+\boldsymbol{\mathsf{r}}_n^*}
-(4\alpha)^{-1}\|\boldsymbol{\mathsf{p}}_n
-\widetilde{\boldsymbol{\mathsf{q}}}_n\|^2
\nonumber\\
&\hspace{5mm}
=\scal{\boldsymbol{\mathsf{x}}_n-\boldsymbol{\mathsf{p}}_n}{
\boldsymbol{\mathsf{E}}_n\boldsymbol{\mathsf{x}}_n
-\boldsymbol{\mathsf{E}}_n\boldsymbol{\mathsf{p}}_n
}
+\scal{\boldsymbol{\mathsf{x}}_n-\boldsymbol{\mathsf{p}}_n}{
\boldsymbol{\mathsf{r}}_n^*}
-(4\alpha)^{-1}\|\boldsymbol{\mathsf{p}}_n
-\widetilde{\boldsymbol{\mathsf{q}}}_n\|^2
\nonumber\\
&\hspace{5mm}
=\sum_{i\in I}\scal{x_{i,n}-a_{i,n}}{E_{i,n}x_{i,n}
-E_{i,n}a_{i,n}}
+\sum_{k\in K}\scal{y_{k,n}-b_{k,n}}{F_{k,n}y_{k,n}
-F_{k,n}b_{k,n}}
\nonumber\\
&\hspace{5mm}
\quad\;+\sum_{k\in K}\scal{z_{k,n}-d_{k,n}}{G_{k,n}z_{k,n}
-G_{k,n}d_{k,n}}+\sum_{k\in K}
\sigma_{k,\varrho_k(n)}^{-1}\|v_{k,n}^*-e_{k,n}^*\|^2
\nonumber\\
&\hspace{5mm}
\quad\;+\scal{\boldsymbol{x}_n-\boldsymbol{a}_n}
{\boldsymbol{R}\boldsymbol{a}_n-\boldsymbol{R}\boldsymbol{x}_n}
-(4\alpha)^{-1}\|\boldsymbol{\mathsf{p}}_n
-\widetilde{\boldsymbol{\mathsf{q}}}_n\|^2
\nonumber\\
&\hspace{5mm}
\geq(\chi+\sigma)\|\boldsymbol{x}_n-\boldsymbol{a}_n\|^2
+\sigma\|\boldsymbol{y}_n-\boldsymbol{b}_n\|^2
+\sigma\|\boldsymbol{z}_n-\boldsymbol{d}_n\|^2
\nonumber\\
&\hspace{5mm}
\quad\;
+\varepsilon\|\boldsymbol{v}_n^*-\boldsymbol{e}_n^*\|^2
-\|\boldsymbol{x}_n-\boldsymbol{a}_n\|\,
\|\boldsymbol{R}\boldsymbol{a}_n-\boldsymbol{R}\boldsymbol{x}_n\|
-(4\alpha)^{-1}\|\boldsymbol{\mathsf{p}}_n
-\widetilde{\boldsymbol{\mathsf{q}}}_n\|^2
\nonumber\\
&\hspace{5mm}
\geq(\chi+\sigma)\|\boldsymbol{x}_n-\boldsymbol{a}_n\|^2
+\sigma\|\boldsymbol{y}_n-\boldsymbol{b}_n\|^2
+\sigma\|\boldsymbol{z}_n-\boldsymbol{d}_n\|^2
\nonumber\\
&\hspace{5mm}
\quad\;
+\varepsilon\|\boldsymbol{v}_n^*-\boldsymbol{e}_n^*\|^2
-\chi\|\boldsymbol{x}_n-\boldsymbol{a}_n\|^2
-(4\alpha)^{-1}\|\boldsymbol{\mathsf{p}}_n
-\widetilde{\boldsymbol{\mathsf{q}}}_n\|^2
\nonumber\\
&\hspace{5mm}
=\big(\sigma-(4\alpha)^{-1}\big)
\big(\|\boldsymbol{x}_n-\boldsymbol{a}_n\|^2
+\|\boldsymbol{y}_n-\boldsymbol{b}_n\|^2
+\|\boldsymbol{z}_n-\boldsymbol{d}_n\|^2\big)
+\varepsilon\|\boldsymbol{v}_n^*-\boldsymbol{e}_n^*\|^2.
\end{align}
Hence, since $\sigma>1/(4\alpha)$ by \eqref{e:a2i'},
taking the limit superior in \eqref{e:1234}
and invoking \eqref{e:5224} yield
\begin{equation}
\label{e:3513}
\boldsymbol{x}_n-\boldsymbol{a}_n\to\boldsymbol{0},\;\:
\boldsymbol{y}_n-\boldsymbol{b}_n\to\boldsymbol{0},\;\:
\boldsymbol{z}_n-\boldsymbol{d}_n\to\boldsymbol{0},\;\:
\text{and}\;\:
\boldsymbol{v}_n^*-\boldsymbol{e}_n^*\to\boldsymbol{0},
\end{equation}
which establishes \ref{t:1iii-}.
In turn, \eqref{e:8870} and \eqref{e:5297} force
\begin{equation}
\label{e:9421}
\boldsymbol{\mathsf{x}}_n
-\boldsymbol{\mathsf{p}}_n\to\boldsymbol{\mathsf{0}}
\quad\text{and}\quad
\|\boldsymbol{\mathsf{E}}_n\boldsymbol{\mathsf{x}}_n
-\boldsymbol{\mathsf{E}}_n\boldsymbol{\mathsf{p}}_n\|
\leq\kappa\|\boldsymbol{\mathsf{x}}_n
-\boldsymbol{\mathsf p}_n\|\to 0
\end{equation}
and \eqref{e:8706} thus yields
$\boldsymbol{\mathsf{p}}_n-\boldsymbol{\mathsf{q}}_n
\to\boldsymbol{\mathsf{0}}$.
Further, we infer from \eqref{e:3514}, \eqref{e:3513},
and Problem~\ref{prob:1}\ref{prob:1e} that
\begin{equation}
\label{e:725}
\|\boldsymbol{\mathsf{r}}_n^*\|^2
=\|\boldsymbol{R}\boldsymbol{a}_n
-\boldsymbol{R}\boldsymbol{x}_n\|^2
\leq\chi^2\|\boldsymbol{a}_n-\boldsymbol{x}_n\|^2\to 0.
\end{equation}
Altogether, it follows from \eqref{e:3514}, \eqref{e:9860},
\eqref{e:9421}, and \eqref{e:725} that
\begin{equation}
\boldsymbol{\mathsf{t}}_n^*=
\big(\boldsymbol{\mathsf{t}}_n^*-\big(\boldsymbol{\mathsf{v}}_n^*
+\boldsymbol{\mathsf{r}}_n^*
+\boldsymbol{\mathsf{w}}_n^*\big)
\big)
+\big(\boldsymbol{\mathsf{E}}_n\boldsymbol{\mathsf{x}}_n
-\boldsymbol{\mathsf{E}}_n\boldsymbol{\mathsf{p}}_n\big)
+\boldsymbol{\mathsf{W}}(\boldsymbol{\mathsf{p}}_n
-\boldsymbol{\mathsf{x}}_n)
+\boldsymbol{\mathsf{r}}_n^*
\to\boldsymbol{\mathsf{0}}.
\end{equation}
Hence, Proposition~\ref{p:3404}\ref{p:3404iv}
guarantees that there exists
$\overline{\boldsymbol{\mathsf{x}}}
=(\overline{\boldsymbol{x}},\overline{\boldsymbol{y}},
\overline{\boldsymbol{z}},\overline{\boldsymbol{v}}^*)\in\zer\sad$
such that $\boldsymbol{\mathsf{x}}_n\weakly
\overline{\boldsymbol{\mathsf{x}}}$.
This and \eqref{e:3513} imply that, for every $i\in I$
and every $k\in K$,
$x_{i,n}\weakly\overline{x}_i$, $a_{i,n}\weakly\overline{x}_i$,
and $v_{k,n}^*\weakly\overline{v}_k^*$.
Finally, Proposition~\ref{p:6}\ref{p:6ii} asserts that
$(\overline{\boldsymbol{x}},\overline{\boldsymbol{v}}^*)$
lies in the set of Kuhn--Tucker
points \eqref{e:Z}, that $\overline{\boldsymbol{x}}$
solves \eqref{e:1p}, and that
$\overline{\boldsymbol{v}}^*$ solves \eqref{e:1d}.
\end{proof}

Some infinite-dimensional applications require strong 
convergence of the iterates; see, e.g., \cite{Sico10,Atto16}. 
This will be guaranteed by the
following variant of Algorithm~\ref{algo:1}, which hinges on the
principle outlined in Proposition~\ref{p:4738}.

\begin{algorithm}
\label{algo:3}
Consider the setting of Problem~\ref{prob:1},
define $\Xi$ as in \eqref{e:6146}, and suppose that
Assumption~\ref{a:2}--\ref{a:3} is in force. Iterate
\pagebreak[1]
\begin{equation}
\begin{array}{l}
\text{for}\;n=0,1,\ldots\\
\left\lfloor
\begin{array}{l}
\text{for every}\;i\in I_n\\
\left\lfloor
\begin{array}{l}
l_{i,n}^*=Q_ix_{i,\pi_i(n)}
+R_i\boldsymbol{x}_{\pi_i(n)}
+\sum_{k\in K}L_{ki}^*v_{k,\pi_i(n)}^*;\\
a_{i,n}=J_{\gamma_{i,\pi_i(n)}A_i}\big(
x_{i,\pi_i(n)}+\gamma_{i,\pi_i(n)}(s_i^*-l_{i,n}^*
-C_ix_{i,\pi_i(n)})\big);\\
a_{i,n}^*=\gamma_{i,\pi_i(n)}^{-1}(x_{i,\pi_i(n)}
-a_{i,n})-l_{i,n}^*+Q_ia_{i,n};\\
\xi_{i,n}=\|a_{i,n}-x_{i,\pi_i(n)}\|^2;
\end{array}
\right.\\
\text{for every}\;i\in I\smallsetminus I_n\\
\left\lfloor
\begin{array}{l}
a_{i,n}=a_{i,n-1};\;a_{i,n}^*=a_{i,n-1}^*;\;
\xi_{i,n}=\xi_{i,n-1};\\
\end{array}
\right.\\
\text{for every}\;k\in K_n\\
\left\lfloor
\begin{array}{l}
u_{k,n}^*=v_{k,\omega_k(n)}^*-B_k^{\LL}y_{k,\omega_k(n)};\\
w_{k,n}^*=v_{k,\omega_k(n)}^*-D_k^{\LL}z_{k,\omega_k(n)};\\
b_{k,n}=J_{\mu_{k,\omega_k(n)}B_k^{\MM}}\big(y_{k,\omega_k(n)}
+\mu_{k,\omega_k(n)}(u_{k,n}^*-B_k^{\CC}y_{k,\omega_k(n)})
\big);\\
d_{k,n}=J_{\nu_{k,\omega_k(n)}D_k^{\MM}}\big(z_{k,\omega_k(n)}
+\nu_{k,\omega_k(n)}(w_{k,n}^*-D_k^{\CC}z_{k,\omega_k(n)})\big);\\
e_{k,n}^*=\sigma_{k,\omega_k(n)}\big(
\sum_{i\in I}L_{ki}x_{i,\omega_k(n)}
-y_{k,\omega_k(n)}-z_{k,\omega_k(n)}-r_k\big)+v_{k,\omega_k(n)}^*
;\\
q_{k,n}^*=\mu_{k,\omega_k(n)}^{-1}(y_{k,\omega_k(n)}-b_{k,n})
+u_{k,n}^*+B_k^{\LL}b_{k,n}-e_{k,n}^*;\\
t_{k,n}^*=\nu_{k,\omega_k(n)}^{-1}(z_{k,\omega_k(n)}-d_{k,n})
+w_{k,n}^*+D_k^{\LL}d_{k,n}-e_{k,n}^*;\\
\eta_{k,n}=\|b_{k,n}-y_{k,\omega_k(n)}\|^2
+\|d_{k,n}-z_{k,\omega_k(n)}\|^2;\\
e_{k,n}=r_k+b_{k,n}+d_{k,n}-\sum_{i\in I}L_{ki}a_{i,n};
\end{array}
\right.\\
\text{for every}\;k\in K\smallsetminus K_n\\
\left\lfloor
\begin{array}{l}
b_{k,n}=b_{k,n-1};\;
d_{k,n}=d_{k,n-1};\;
e_{k,n}^*=e_{k,n-1}^*;\;
q_{k,n}^*=q_{k,n-1}^*;\;
t_{k,n}^*=t_{k,n-1}^*;\\
\eta_{k,n}=\eta_{k,n-1};\;
e_{k,n}=r_k+b_{k,n}+d_{k,n}-\sum_{i\in I}L_{ki}a_{i,n};
\end{array}
\right.\\
\text{for every}\;i\in I\\
\left\lfloor
\begin{array}{l}
p_{i,n}^*=a_{i,n}^*
+R_i\boldsymbol{a}_n
+\sum_{k\in K}L_{ki}^*e_{k,n}^*;
\end{array}
\right.\\
\begin{aligned}
\Delta_n&=\textstyle
{-}(4\alpha)^{-1}\big(\sum_{i\in I}\xi_{i,n}
+\sum_{k\in K}\eta_{k,n}\big)
+\sum_{i\in I}\scal{x_{i,n}-a_{i,n}}{p_{i,n}^*}\\
&\textstyle
\quad\;+\sum_{k\in K}\big(\scal{y_{k,n}-b_{k,n}}{q_{k,n}^*}
+\scal{z_{k,n}-d_{k,n}}{t_{k,n}^*}
+\scal{e_{k,n}}{v_{k,n}^*-e_{k,n}^*}\big);
\end{aligned}\\
\text{if}\;\Delta_n>0\\
\left\lfloor
\begin{array}{l}
\tau_n=\sum_{i\in I}\|p_{i,n}^*\|^2+\sum_{k\in K}\big(
\|q_{k,n}^*\|^2+\|t_{k,n}^*\|^2+\|e_{k,n}\|^2\big);\\
\begin{aligned}
\varsigma_n&=\textstyle
\sum_{i\in I}\|x_{i,0}-x_{i,n}\|^2\\
&\textstyle\quad\;
+\sum_{k\in K}\big(\|y_{k,0}-y_{k,n}\|^2
+\|z_{k,0}-z_{k,n}\|^2+\|v_{k,0}^*-v_{k,n}^*\|^2\big);
\end{aligned}\\
\begin{aligned}
\chi_n&=\textstyle
\sum_{i\in I}\scal{x_{i,0}-x_{i,n}}{p_{i,n}^*}\\
&\textstyle\quad\;
+\sum_{k\in K}\big(\scal{y_{k,0}-y_{k,n}}{q_{k,n}^*}
+\scal{z_{k,0}-z_{k,n}}{t_{k,n}^*}
+\scal{e_{k,n}}{v_{k,0}^*-v_{k,n}^*}\big);
\end{aligned}\\
(\kappa_n,\lambda_n)=\Xi(\Delta_n,\tau_n,\varsigma_n,\chi_n);\\
\text{for every}\;i\in I\\
\left\lfloor
\begin{array}{l}
x_{i,n+1}=(1-\kappa_n)x_{i,0}+\kappa_nx_{i,n}-\lambda_np_{i,n}^*;
\end{array}
\right.\\
\text{for every}\;k\in K\\
\left\lfloor
\begin{array}{l}
y_{k,n+1}=(1-\kappa_n)y_{k,0}+\kappa_ny_{k,n}-\lambda_nq_{k,n}^*;\\
z_{k,n+1}=(1-\kappa_n)z_{k,0}+\kappa_nz_{k,n}-\lambda_nt_{k,n}^*;\\
v_{k,n+1}^*=(1-\kappa_n)v_{k,0}^*
+\kappa_nv_{k,n}^*-\lambda_ne_{k,n};
\end{array}
\right.\\[6mm]
\end{array}
\right.\\
\text{else}\\
\left\lfloor
\begin{array}{l}
\text{for every}\;i\in I\\
\left\lfloor
\begin{array}{l}
x_{i,n+1}=x_{i,n};
\end{array}
\right.\\
\text{for every}\;k\in K\\
\left\lfloor
\begin{array}{l}
y_{k,n+1}=y_{k,n};\;z_{k,n+1}=z_{k,n};\;v_{k,n+1}^*=v_{k,n}^*.
\end{array}
\right.\\[1mm]
\end{array}
\right.\\[9.5mm]
\end{array}
\right.
\end{array}
\end{equation}
\end{algorithm}

\newpage
\begin{theorem}
\label{t:2}
Consider the setting of Algorithm~\ref{algo:3} and suppose that
the dual problem \eqref{e:1d} has a solution. Then the following
hold:
\begin{enumerate}
\item
\label{t:2i}
Let $i\in I$. Then $\sum_{n\in\NN}\|x_{i,n+1}-x_{i,n}\|^2<\pinf$.
\item
\label{t:2ii}
Let $k\in K$. Then $\sum_{n\in\NN}\|y_{k,n+1}-y_{k,n}\|^2<\pinf$,
$\sum_{n\in\NN}\|z_{k,n+1}-z_{k,n}\|^2<\pinf$, and
$\sum_{n\in\NN}\|v_{k,n+1}^*-v_{k,n}^*\|^2<\pinf$.
\item
\label{t:2iii-}
Let $i\in I$ and $k\in K$. Then
$x_{i,n}-a_{i,n}\to 0$, $y_{k,n}-b_{k,n}\to 0$,
$z_{k,n}-d_{k,n}\to 0$, and $v_{k,n}^*-e_{k,n}^*\to 0$.
\item
\label{t:2iii}
There exist a solution $\overline{\boldsymbol{x}}$ to \eqref{e:1p} 
and a solution $\overline{\boldsymbol{v}}^*$ to \eqref{e:1d} 
such that, for every $i\in I$ and every $k\in K$,
$x_{i,n}\to\overline{x}_i$, $a_{i,n}\to\overline{x}_i$, and
$v_{k,n}^*\to\overline{v}_k^*$.
In addition, $(\overline{\boldsymbol{x}},
\overline{\boldsymbol{v}}^*)$ is a Kuhn--Tucker point
of Problem~\ref{prob:1} in the sense of \eqref{e:Z}.
\end{enumerate}
\end{theorem}
\begin{proof}
Proceed as in the proof of Theorem~\ref{t:1}
and use Proposition~\ref{p:4738} instead of
Proposition~\ref{p:3404}.
\end{proof}

\section{Applications}
\label{sec:4}

In nonlinear analysis and optimization, problems with multiple
variables occur in areas such as 
game theory \cite{Atto08,Bric13,YiPa19},
evolution inclusions \cite{Sico10},
traffic equilibrium \cite{Sico10,Fuku96},
domain decomposition \cite{Atto16},
machine learning \cite{Bach12,Bric19},
image recovery \cite{Bric09,Jmiv11},
infimal-convolution regularization \cite{Siop13}, 
statistics \cite{Ejst20,Bien20},
neural networks \cite{Neur20},
and variational inequalities \cite{Fuku96}.
The numerical methods used in the above papers are limited to
special cases of Problem~\ref{prob:1} and they do not perform block
iterations and they operate in synchronous mode. The methods
presented in Theorems~\ref{t:1}~and~\ref{t:2} provide a unified
treatment of these problems as well as extensions, within a
considerably more flexible algorithmic framework. In this section,
we illustrate this in the context of variational inequalities and
multivariate minimization. Below we present only the applications
of Theorem~\ref{t:1} as similar applications of Theorem~\ref{t:2}
follow using similar arguments. 

\newpage
\subsection{Application to variational inequalities}
\label{sec:41}

The standard variational inequality problem associated with a
closed convex subset $D$ of a real Hilbert space $\GG$ and a
maximally monotone operator $B\colon\GG\to\GG$ is to 
\begin{equation}
\label{e:2z3}
\text{find}\;\:\overline{y}\in D
\;\:\text{such that}\;\:
(\forall y\in D)\;\;\scal{\overline{y}-y}{B\overline{y}}\leq 0.
\end{equation}
Classical methods require the ability to project onto $D$ and
specific assumptions on $B$ such as cocoercivity, Lipschitz
continuity, or the ability to compute the resolvent
\cite{Livre1,Facc03,Tsen00}. 
Let us consider a refined version of \eqref{e:2z3} in which $B$
and $D$ are decomposed into basic components, and for which these
classical methods are not applicable. 

\begin{problem}
\label{prob:8}
Let $I$ be a nonempty finite set and let
$(\HH_i)_{i\in I}$ and $\GG$ be real Hilbert spaces.
For every $i\in I$, let $E_i$ and $F_i$
be closed convex subsets of $\HH_i$ such that
$E_i\cap F_i\neq\emp$ and let $L_i\colon\HH_i\to\GG$ be linear and
bounded. In addition, let 
$B^\MM\colon\GG\to 2^\GG$ be at most single-valued and
maximally monotone, let $B^\CC\colon\GG\to\GG$
be cocoercive with constant $\beta^\CC\in\RPP$, and let
$B^\LL\colon\GG\to\GG$ be Lipschitzian with constant
$\beta^\LL\in\RP$. The objective is to
\begin{equation}
\label{e:2327}
\text{find}\;\:\overline{y}\in\sum_{i\in I}L_i(E_i\cap F_i)
\;\:\text{such that}\;\:
\bigg(\forall y\in\sum_{i\in I}L_i(E_i\cap F_i)\bigg)\;\;
\sscal{\overline{y}-y}{B^\MM\overline{y}+B^\CC\overline{y}
+B^\LL\overline{y}}\leq 0.
\end{equation}
\end{problem}

To motivate our analysis, let us consider an illustration of
\eqref{e:2327}.

\begin{example}
\label{ex:24}
Let $I$ be a nonempty finite set and let $(\ZZ_i)_{i\in I}$ and
$\KK$ be real Hilbert spaces. 
For every $i\in I$, let $S_i\subset\ZZ_i$ be 
closed and convex, and let $M_i\colon\ZZ_i\to\KK$ be
linear and bounded. 
In addition, let $f\in\Gamma_0(\KK)$ be G\^ateaux differentiable
on $\dom\partial f$,
let $\varphi\colon\KK\to\RR$ be convex and differentiable with a
Lipschitzian gradient, let $\WW$ be a real Hilbert space, let
$g\in\Gamma_0(\WW)$ be such that $g^*$ is G\^ateaux
differentiable on $\dom\partial g^*$,
let $D$ be a closed convex subset of $\WW$ such that
\begin{equation}
\label{e:gd}
0\in\sri(D-\dom g^*),
\end{equation}
let $h\in\Gamma_0(\WW)$ be
strongly convex, and let $L\colon\KK\to\WW$ be linear and bounded.
Note that, by \cite[Theorem~18.15]{Livre1},
$h^*$ is differentiable on $\WW$ and $\nabla h^*$ is cocoercive.
The objective is to solve the Kuhn--Tucker problem
\begin{multline}
\label{e:d4}
\text{find}\;\:(\overline{x},\overline{v}^*)\in\KK\oplus\WW
\;\:\text{such that}\\
\begin{bmatrix}
0\\
0
\end{bmatrix}
\in
\underbrace{\begin{bmatrix}
\nabla f&0\\
0&\nabla g^*\\
\end{bmatrix}}_{\text{monotone}}
\begin{bmatrix}
\overline{x}\\
\overline{v}^*
\end{bmatrix}
+
\underbrace{\begin{bmatrix}
\nabla\varphi&0\\
0&\nabla h^*\\
\end{bmatrix}}_{\text{cocoercive}}
\begin{bmatrix}
\overline{x}\\
\overline{v}^*
\end{bmatrix}
+
\underbrace{\begin{bmatrix}
0&L^*\\
-L&0\\
\end{bmatrix}}_{\text{Lipschitzian}}
\begin{bmatrix}
\overline{x}\\
\overline{v}^*
\end{bmatrix}
+
\underbrace{\begin{bmatrix}
N_C&0\\
0&N_D\\
\end{bmatrix}}_{\text{normal cone}}
\begin{bmatrix}
\overline{x}\\
\overline{v}^*
\end{bmatrix},
\end{multline}
where it is assumed that
\begin{equation}
\label{e:d5}
C=\sum_{i\in I}M_i(S_i)\;\:\text{is closed and}\;\:
0\in\sri(C-\dom f).
\end{equation}
Since $\dom h^*=\WW$, we deduce from \eqref{e:gd} and
\cite[Proposition~15.7(i)]{Livre1}
that $g\infconv h\infconv\sigma_D\in\Gamma_0(\WW)$. 
It follows from standard convex calculus \cite{Livre1} that a
solution $(\overline{x},\overline{v}^*)$ to \eqref{e:d4} provides a
solution $\overline{x}$ to
\begin{equation}
\label{e:2329}
\minimize{x\in C}
{f(x)+\big(g\infconv h\infconv\sigma_D\big)(Lx)+\varphi(x)},
\end{equation}
as well as a solution $\overline{v}^*$ to the associated
Fenchel--Rockafellar dual 
\begin{equation}
\label{e:2375}
\minimize{v^*\in D}
{\big((f+\varphi)^*\infconv\sigma_C\big)(-L^*v^*)
+g^*(v^*)+h^*(v^*)}.
\end{equation}
To see that \eqref{e:d4}--\eqref{e:d5} is a special case of
Problem~\ref{prob:8}, set $\GG=\KK\oplus\WW$ and 
\begin{equation}
\label{e:24j}
(\forall i\in I)\quad
L_i\colon\HH_i=\ZZ_i\oplus\WW\to\GG\colon
(z_i,v^*)\mapsto(M_iz_i,v^*/\card I),\;\;
E_i=S_i\times D,\;\;\text{and}\;\; F_i=\ZZ_i\times\WW.
\end{equation}
Note that
\begin{equation}
\label{e:CD}
C\times D=\sum_{i\in I}L_i(E_i\cap F_i).
\end{equation}
Further, in view of \cite[Proposition~17.31(i)]{Livre1},
let us define
\begin{equation}
\label{e:2r}
\begin{cases}
B^\MM\colon\GG\to 2^{\GG}\colon(x,v^*)\mapsto
\partial(f\oplus g^*)(x,v^*)=
\begin{cases}
\big(\nabla f(x),\nabla g^*(v^*)\big),&\text{if}\;\:
(x,v^*)\in\dom\partial f\times\dom\partial g^*;\\
\emp,&\text{otherwise}
\end{cases}
\\
B^\CC\colon\GG\to\GG\colon(x,v^*)\mapsto
\big(\nabla\varphi(x),\nabla h^*(v^*)\big)\\
B^\LL\colon\GG\to\GG\colon (x,v^*)\mapsto (L^*v^*,{-}Lx).
\end{cases}
\end{equation}
Then $B^\MM$ is maximally monotone \cite[Theorem~20.25]{Livre1}, 
$B^\CC$ is cocoercive \cite[Corollary~18.17]{Livre1}, and 
$B^\LL$ is a skew bounded linear operator, hence monotone and
Lipschitzian \cite[Example~20.35]{Livre1}. 
In turn, combining \eqref{e:24j} and \eqref{e:2r},
we conclude that \eqref{e:d4} can be written as
\begin{equation}
\text{find}\;\:(\overline{x},\overline{v}^*)\in\KK\oplus\WW
\;\:\text{such that}\;\:
(0,0)\in B^\MM(\overline{x},\overline{v}^*)
+B^\CC(\overline{x},\overline{v}^*)
+B^\LL(\overline{x},\overline{v}^*)
+N_{C\times D}(\overline{x},\overline{v}^*)
\end{equation}
which, in the light of \eqref{e:CD}, fits the format of
\eqref{e:2327}. Special cases of \eqref{e:2329} involving
minimization over Minkowski sum of sets are found in areas such as
signal and image processing \cite{Aujo05,Smms05,Onos14},
location and network problems \cite{Namm14}, as well as
robotics and computational mechanics \cite{Wang20}.
\end{example}

We are going to reformulate Problem~\ref{prob:8} as
a realization of Problem~\ref{prob:1} and solve it via a
block-iterative method derived from Algorithm~\ref{algo:1}. 
In addition, our approach employs the individual
projection operators onto the sets $(E_i)_{i\in I}$ and 
$(F_i)_{i\in I}$, and the resolvents of the operator $B^\MM$.
We are not aware of any method which features such flexibility. 
For instance, consider the special case discussed in
\cite[Section~4]{Fuku96}, where $\GG=\RR^N$, $B^\CC=B^\LL=0$,
$T\colon\RR^N\to\RR^M$ is a linear operator, and,
for every $i\in I$, $\HH_i=\RR^N$, $L_i=\Id$,
$E_i=T^{-1}(\{d_i\})$ for some $d_i\in\RR^M$,
and $F_i=\RP^N$. There, the evaluations of all the projectors 
$(\proj_{E_i\cap F_i})_{i\in I}$ are required at every iteration.
Note that there are no closed-form expressions for 
$(\proj_{E_i\cap F_i})_{i\in I}$ in general.

\begin{corollary}
\label{c:5}
Consider the setting of Problem~\ref{prob:8}.
Let $\sigma\in\left]1/(4\beta^{\CC}),\pinf\right[$,
$\varepsilon\in\left]0,\min\{1,1/(\beta^{\LL}
+\sigma)\}\right[$,
and $K=I\cup\{\overline{k}\}$, where
$\overline{k}\notin I$. Suppose that Assumption~\ref{a:1} is in
force, together with the following:
\begin{enumerate}[label={\rm[\alph*]}]
\item
For every $i\in I$ and every $n\in\NN$, 
$\{\gamma_{i,n},\mu_{i,n},\nu_{i,n}\}\subset
\left[\varepsilon,1/\sigma\right]$ and 
$\sigma_{i,n}\in\left[\varepsilon,1/\varepsilon\right]$.
\item
For every $n\in\NN$, 
$\lambda_n\in\left[\varepsilon,2-\varepsilon\right]$,
$\mu_{\overline{k},n}\in
\left[\varepsilon,1/(\beta^{\LL}+\sigma)\right]$,
$\nu_{\overline{k},n}\in\left[\varepsilon,1/\sigma\right]$, 
and $\sigma_{\overline{k},n}\in
\left[\varepsilon,1/\varepsilon\right]$.
\item
For every $i\in I$, 
$\{x_{i,0},y_{i,0},z_{i,0},v_{i,0}^*\}\subset\HH_i$;
$\{y_{\overline{k},0},z_{\overline{k},0},
v_{\overline{k},0}^*\}\subset\GG$.
\end{enumerate}
Iterate
\begin{equation}
\label{e:7133}
\begin{array}{l}
\text{for}\;n=0,1,\ldots\\
\left\lfloor
\begin{array}{l}
\text{for every}\;i\in I_n\\
\left\lfloor
\begin{array}{l}
l_{i,n}^*=v_{i,n}^*+L_i^*v_{\overline{k},n}^*;\\
a_{i,n}=\proj_{E_i}\big(x_{i,n}-\gamma_{i,n}l_{i,n}^*\big);\\
a_{i,n}^*=\gamma_{i,n}^{-1}(x_{i,n}-a_{i,n})-l_{i,n}^*;\\
\xi_{i,n}=\|a_{i,n}-x_{i,n}\|^2;
\end{array}
\right.\\
\text{for every}\;i\in I\smallsetminus I_n\\
\left\lfloor
\begin{array}{l}
a_{i,n}=a_{i,n-1};\;a_{i,n}^*=a_{i,n-1}^*;\;
\xi_{i,n}=\xi_{i,n-1};\\
\end{array}
\right.\\
\text{for every}\;k\in K_n\\
\left\lfloor
\begin{array}{l}
\text{if}\;k\in I\\
\left\lfloor
\begin{array}{l}
b_{k,n}=\proj_{F_k}\big(y_{k,n}+\mu_{k,n}v_{k,n}^*\big);\\
e_{k,n}^*=\sigma_{k,n}(x_{k,n}-y_{k,n}-z_{k,n})+v_{k,n}^*;\\
q_{k,n}^*=\mu_{k,n}^{-1}(y_{k,n}-b_{k,n})+v_{k,n}^*-e_{k,n}^*;\\
e_{k,n}=b_{k,n}-a_{k,n};
\end{array}
\right.\\
\text{if}\;k=\overline{k}\\
\left\lfloor
\begin{array}{l}
u_{k,n}^*=v_{k,n}^*-B^{\LL}y_{k,n};\\
b_{k,n}=J_{\mu_{k,n}B^{\MM}}\big(y_{k,n}
+\mu_{k,n}(u_{k,n}^*-B^{\CC}y_{k,n})\big);\\
e_{k,n}^*=\sigma_{k,n}\big(\sum_{i\in I}L_ix_{i,n}
-y_{k,n}-z_{k,n}\big)+v_{k,n}^*;\\
q_{k,n}^*=\mu_{k,n}^{-1}(y_{k,n}-b_{k,n})
+u_{k,n}^*+B^{\LL}b_{k,n}-e_{k,n}^*;\\
e_{k,n}=b_{k,n}-\sum_{i\in I}L_ia_{i,n};
\end{array}
\right.\\
t_{k,n}^*=\nu_{k,n}^{-1}z_{k,n}+v_{k,n}^*-e_{k,n}^*;\\
\eta_{k,n}=\|b_{k,n}-y_{k,n}\|^2+\|z_{k,n}\|^2;
\end{array}
\right.\\
\text{for every}\;k\in K\smallsetminus K_n\\
\left\lfloor
\begin{array}{l}
b_{k,n}=b_{k,n-1};\;
e_{k,n}^*=e_{k,n-1}^*;\;
q_{k,n}^*=q_{k,n-1}^*;\;
t_{k,n}^*=t_{k,n-1}^*;\;
\eta_{k,n}=\eta_{k,n-1};\\
\text{if}\;k\in I\\
\left\lfloor
\begin{array}{l}
e_{k,n}=b_{k,n}-a_{k,n};
\end{array}
\right.\\
\text{if}\;k=\overline{k}\\
\left\lfloor
\begin{array}{l}
e_{k,n}=b_{k,n}-\sum_{i\in I}L_ia_{i,n};
\end{array}
\right.\\[2mm]
\end{array}
\right.\\
\text{for every}\;i\in I\\
\left\lfloor
\begin{array}{l}
p_{i,n}^*=a_{i,n}^*+e_{i,n}^*+L_i^*e_{\overline{k},n}^*;
\end{array}
\right.\\
\begin{aligned}
\Delta_n&=\textstyle
{-}(4\beta^\CC)^{-1}\big(\sum_{i\in I}\xi_{i,n}
+\sum_{k\in K}\eta_{k,n}\big)
+\sum_{i\in I}\scal{x_{i,n}-a_{i,n}}{p_{i,n}^*}\\
&\textstyle
\quad\;+\sum_{k\in K}\big(\scal{y_{k,n}-b_{k,n}}{q_{k,n}^*}
+\scal{z_{k,n}}{t_{k,n}^*}
+\scal{e_{k,n}}{v_{k,n}^*-e_{k,n}^*}\big);
\end{aligned}\\
\text{if}\;\Delta_n>0\\
\left\lfloor
\begin{array}{l}
\theta_n=\lambda_n\Delta_n/
\big(\sum_{i\in I}\|p_{i,n}^*\|^2+\sum_{k\in K}\big(
\|q_{k,n}^*\|^2+\|t_{k,n}^*\|^2+\|e_{k,n}\|^2\big)\big);\\
\text{for every}\;i\in I\\
\left\lfloor
\begin{array}{l}
x_{i,n+1}=x_{i,n}-\theta_np_{i,n}^*;
\end{array}
\right.\\
\text{for every}\;k\in K\\
\left\lfloor
\begin{array}{l}
y_{k,n+1}=y_{k,n}-\theta_nq_{k,n}^*;\;
z_{k,n+1}=z_{k,n}-\theta_nt_{k,n}^*;\;
v_{k,n+1}^*=v_{k,n}^*-\theta_ne_{k,n};
\end{array}
\right.\\[1mm]
\end{array}
\right.\\
\text{else}\\
\left\lfloor
\begin{array}{l}
\text{for every}\;i\in I\\
\left\lfloor
\begin{array}{l}
x_{i,n+1}=x_{i,n};
\end{array}
\right.\\
\text{for every}\;k\in K\\
\left\lfloor
\begin{array}{l}
y_{k,n+1}=y_{k,n};\;z_{k,n+1}=z_{k,n};\;v_{k,n+1}^*=v_{k,n}^*.
\end{array}
\right.\\[1mm]
\end{array}
\right.\\[9.5mm]
\end{array}
\right.
\end{array}
\end{equation}
Furthermore, suppose that \eqref{e:2327} has a solution and that
\begin{equation}
\label{e:3437}
(\forall i\in I)\quad N_{E_i\cap F_i}=N_{E_i}+N_{F_i}.
\end{equation}
Then there exists
$(\overline{x}_i)_{i\in I}\in\bigoplus_{i\in I}\HH_i$
such that $\sum_{i\in I}L_i\overline{x}_i$ solves \eqref{e:2327}
and, for every $i\in I$, $x_{i,n}\weakly\overline{x}_i$ and
$a_{i,n}\weakly\overline{x}_i$.
\end{corollary}
\begin{proof}
Set $\HHH=\bigoplus_{i\in I}\HH_i$. Let us consider the problem
\begin{equation}
\label{e:3877p}
\text{find}\;\:\overline{\boldsymbol{x}}\in\HHH
\;\:\text{such that}\;\:
(\forall i\in I)\;\;
0\in N_{E_i}\overline{x}_i+N_{F_i}\overline{x}_i+
L_i^*(B^\MM+B^\CC+B^\LL)
\Bigg(\sum_{j\in I}L_j\overline{x}_j\Bigg)
\end{equation}
together with the associated dual problem
\begin{equation}
\label{e:3877d}
\text{find}\;\:
\big(\overline{\boldsymbol{x}}^*,\overline{v}^*\big)
\in\HHH\oplus\GG\;\:\text{such that}\;\:
\big(\exi\boldsymbol{x}\in\HHH\big)\;\;
\begin{cases}
(\forall i\in I)\;\;
{-}\overline{x}_i^*-L_i^*\overline{v}^*\in N_{E_i}x_i
\;\:\text{and}\;\:\overline{x}_i^*\in N_{F_i}x_i\\
\overline{v}^*
=(B^\MM+B^\CC+B^\LL)\big(\sum_{j\in I}L_jx_j\big).
\end{cases}
\end{equation}
Denote by $\mathscr{P}$ and $\mathscr{D}$ the sets of
solutions to \eqref{e:3877p} and \eqref{e:3877d}, respectively.
We observe that the primal-dual problem
\eqref{e:3877p}--\eqref{e:3877d} is a special case
of Problem~\ref{prob:1} with
\begin{equation}
(\forall i\in I)\quad A_i=N_{E_i},\quad
C_i=Q_i=0,\quad R_i=0,
\quad\text{and}\quad s_i^*=0,
\end{equation}
and
\begin{equation}
\label{e:9834}
(\forall k\in K)\quad
\begin{cases}
\GG_k=\HH_k,\;\:B_k^{\MM}=N_{F_k},\;\:
B_k^{\CC}=B_k^{\LL}=0\;\:\text{if}\;\:k\in I;\\
\GG_{\overline{k}}=\GG,\;\:
B_{\overline{k}}^{\MM}=B^{\MM},\;\:
B_{\overline{k}}^{\CC}=B^{\CC},\;\:
B_{\overline{k}}^{\LL}=B^{\LL}\\
D_k^{\MM}=\{0\}^{-1},\;\:D_k^{\CC}=D_k^{\LL}=0,\;\:r_k=0\\
(\forall j\in I)\;\; L_{kj}=
\begin{cases}
\Id,&\text{if}\;\:k=j;\\
0,&\text{if}\;\:k\in I\;\:\text{and}\;\:k\neq j;\\
L_j,&\text{if}\;\:k=\overline{k}.
\end{cases}
\end{cases}
\end{equation}
Further, we have
\begin{equation}
\begin{cases}
(\forall i\in I)(\forall n\in\NN)\;\:
J_{\gamma_{i,n}A_i}=\proj_{E_i}\\
(\forall k\in K)(\forall n\in\NN)\;\:
J_{\nu_{k,n}D_k^{\MM}}=0
\;\:\text{and}\;\:J_{\mu_{k,n}B_k^{\MM}}=
\begin{cases}
\proj_{F_k},&\text{if}\;\:k\in I;\\
J_{\mu_{k,n}B^{\MM}},&\text{if}\;\:k=\overline{k}.
\end{cases}
\end{cases}
\end{equation}
Therefore, \eqref{e:7133} is a realization of
Algorithm~\ref{algo:1}
in the context of \eqref{e:3877p}--\eqref{e:3877d}.
Now define
$\boldsymbol{D}=\bigtimes_{i\in I}(E_i\cap F_i)$ and
$\boldsymbol{L}\colon\HHH\to\GG\colon
\boldsymbol{x}\mapsto\sum_{i\in I}L_ix_i$.
Then $\boldsymbol{L}^*\colon\GG\to\HHH\colon
y^*\mapsto(L_i^*y^*)_{i\in I}$.
Hence, by \eqref{e:2327}, \cite[Proposition~16.9]{Livre1},
and \eqref{e:3437},
\begin{align}
\label{e:6368}
&(\forall\overline{y}\in\GG)\quad
\overline{y}\;\:\text{solves \eqref{e:2327}}
\nonumber\\
&\hspace{22mm}
\Leftrightarrow
(\exi\overline{\boldsymbol{x}}\in\boldsymbol{D})\;\:
\begin{cases}
\overline{y}=\boldsymbol{L}\overline{\boldsymbol{x}}\\
(\forall\boldsymbol{x}\in\boldsymbol{D})\;\;
\sscal{\boldsymbol{L}\overline{\boldsymbol{x}}
-\boldsymbol{L}\boldsymbol{x}}{
(B^{\MM}+B^{\CC}+B^{\LL})(
\boldsymbol{L}\overline{\boldsymbol{x}})}\leq 0
\end{cases}
\nonumber\\
&\hspace{22mm}
\Leftrightarrow
(\exi\overline{\boldsymbol{x}}\in\boldsymbol{D})\;\:
\begin{cases}
\overline{y}=\boldsymbol{L}\overline{\boldsymbol{x}}\\
(\forall\boldsymbol{x}\in\boldsymbol{D})\;\;
\sscal{\overline{\boldsymbol{x}}-\boldsymbol{x}}{
\boldsymbol{L}^*\big((B^{\MM}+B^{\CC}+B^{\LL})
(\boldsymbol{L}\overline{\boldsymbol{x}})\big)
}\leq 0
\end{cases}
\nonumber\\
&\hspace{22mm}
\Leftrightarrow
(\exi\overline{\boldsymbol{x}}\in\HHH)\;\:
\begin{cases}
\overline{y}=\boldsymbol{L}\overline{\boldsymbol{x}}\\
\boldsymbol{0}\in
N_{\boldsymbol{D}}\overline{\boldsymbol{x}}
+\boldsymbol{L}^*\big((B^{\MM}+B^{\CC}+B^{\LL})
(\boldsymbol{L}\overline{\boldsymbol{x}})\big)
\end{cases}
\nonumber\\
&\hspace{22mm}
\Leftrightarrow
(\exi\overline{\boldsymbol{x}}\in\HHH)\;\:
\begin{cases}
\overline{y}=\boldsymbol{L}\overline{\boldsymbol{x}}\\
(\forall i\in I)\;\;
0\in N_{E_i\cap F_i}\overline{x}_i
+L_i^*(B^\MM+B^\CC+B^\LL)
\big(\sum_{j\in I}L_j\overline{x}_j\big)
\end{cases}
\nonumber\\
&\hspace{22mm}
\Leftrightarrow
(\exi\overline{\boldsymbol{x}}\in\HHH)\;\:
\begin{cases}
\overline{y}=\boldsymbol{L}\overline{\boldsymbol{x}}\\
(\forall i\in I)\;\;
0\in N_{E_i}\overline{x}_i+N_{F_i}\overline{x}_i
+L_i^*(B^\MM+B^\CC+B^\LL)
\big(\sum_{j\in I}L_j\overline{x}_j\big)
\end{cases}
\nonumber\\
&\hspace{22mm}
\Leftrightarrow
(\exi\overline{\boldsymbol{x}}\in\mathscr{P})\;\;
\overline{y}=\boldsymbol{L}\overline{\boldsymbol{x}}.
\end{align}
In turn, $\mathscr{P}\neq\emp$ since \eqref{e:2327} has a 
solution. Therefore, in view of \eqref{e:9834}, 
Proposition~\ref{p:6}\ref{p:6v}\ref{p:6vd}
yields $\mathscr{D}\neq\emp$. As a result,
Theorem~\ref{t:1}\ref{t:1iii} asserts that there exists
$(\overline{x}_i)_{i\in I}\in\mathscr{P}$ such that,
for every $i\in I$, $x_{i,n}\weakly\overline{x}_i$ and
$a_{i,n}\weakly\overline{x}_i$.
Finally, using \eqref{e:6368},
we conclude that $\sum_{i\in I}L_i\overline{x}_i$
solves \eqref{e:2327}.
\end{proof}

\begin{remark}
\label{r:j1}
Theorem~\ref{t:1} allows us to tackle other types of variational
inequalities. For instance, let $(\HH_i)_{i\in I}$ be a finite
family of real Hilbert spaces and set 
$\HHH=\bigoplus_{i\in I}\HH_i$. 
For every $i\in I$, let $\varphi_i\in\Gamma_0(\HH_i)$
and let $R_i\colon\HHH\to\HH_i$
be such that Problem~\ref{prob:1}\ref{prob:1e} holds.
The objective is to
\begin{equation}
\label{e:j1}
\text{find}\;\:\overline{\boldsymbol{x}}\in\HHH
\;\:\text{such that}\;\:(\forall i\in I)\;\;
0\in\partial\varphi_i(\overline{x}_i)
+R_i\overline{\boldsymbol{x}}.
\end{equation}
This simple instantiation of Problem~\ref{prob:1} shows up in
neural networks \cite{Neur20} and in game theory
\cite{Atto08,Bric13}. Thanks to Theorem~\ref{t:1}, it can be
solved using an asynchronous block-iterative strategy, which is not
possible with current splitting techniques such as those of
\cite{MaPr18,John20}. 
\end{remark}

\subsection{Application to multivariate minimization}
\label{sec:45}

We consider a composite multivariate minimization problem involving
various types of convex functions and combinations between them.

\begin{problem}
\label{prob:2}
Let $(\HH_i)_{i\in I}$ and $(\GG_k)_{k\in K}$ be finite families of
real Hilbert spaces, and set $\HHH=\bigoplus_{i\in I}\HH_i$ and
$\GGG=\bigoplus_{k\in K}\GG_k$. For every $i\in I$ and every 
$k\in K$, let $f_i\in\Gamma_0(\HH_i)$, let $\alpha_i\in\RPP$, let
$\varphi_i\colon\HH_i\to\RR$ be convex and differentiable
with a $(1/\alpha_i)$-Lipschitzian gradient,
let $g_k\in\Gamma_0(\GG_k)$, let $h_k\in\Gamma_0(\GG_k)$,
let $\beta_k\in\RPP$, let $\psi_k\colon\GG_k\to\RR$
be convex and differentiable with a $(1/\beta_k)$-Lipschitzian
gradient, and suppose that $L_{ki}\colon\HH_i\to\GG_k$
is linear and bounded. In addition, 
let $\chi\in\RP$ and let $\Theta\colon\HHH\to\RR$
be convex and differentiable with a $\chi$-Lipschitzian gradient.
The objective is to
\begin{equation}
\label{e:min-p}
\minimize{\boldsymbol{x}\in\HHH}{
\Theta(\boldsymbol{x})
+\sum_{i\in I}\big(f_i(x_i)+\varphi_i(x_i)\big)
+\sum_{k\in K}\big((g_k+\psi_k)\infconv h_k\big)\Bigg(
\sum_{j\in I}L_{kj}x_j\Bigg)}.
\end{equation}
\end{problem}

Special cases of Problem~\ref{prob:2} are found in various
contexts, e.g., \cite{Bric09,Jmiv11,Siop13,MaPr18,Hint06,John20}.
Formulation \eqref{e:min-p} brings together these disparate
problems and the following algorithm makes it possible to solve
them in an asynchronous block-iterative fashion in full generality.

\newpage

\begin{algorithm}
\label{algo:2}
Consider the setting of Problem~\ref{prob:2} and suppose that
Assumption~\ref{a:1}--\ref{a:3} is in force. Set
$\alpha=\min\{\alpha_i,\beta_k\}_{i\in I,k\in K}$, let
$\sigma\in\left]1/(4\alpha),\pinf\right[$, and let
$\varepsilon\in\left]0,\min\{1,1/(\chi+\sigma)\}\right[$. For
every $i\in I$, every $k\in K$, and every $n\in\NN$, let 
$\gamma_{i,n}\in\left[\varepsilon,1/(\chi+\sigma)\right]$, let 
$\{\mu_{k,n},\nu_{k,n}\}\subset
\left[\varepsilon,1/\sigma\right]$, let
$\sigma_{k,n}\in\left[\varepsilon,1/\varepsilon\right]$,
and let $\lambda_n\in\left[\varepsilon,2-\varepsilon\right]$. 
In addition, let $\boldsymbol{x}_{0}\in\HHH$ and 
$\{\boldsymbol{y}_{0},\boldsymbol{z}_{0},
\boldsymbol{v}_{0}^*\}\subset\GGG$.
Iterate
\begin{equation}
\label{e:9951}
\begin{array}{l}
\text{for}\;n=0,1,\ldots\\
\left\lfloor
\begin{array}{l}
\text{for every}\;i\in I_n\\
\left\lfloor
\begin{array}{l}
l_{i,n}^*=\nabla_{\!i}\,\Theta(\boldsymbol{x}_{\pi_i(n)})
+\sum_{k\in K}L_{ki}^*v_{k,\pi_i(n)}^*;\\
a_{i,n}=\prox_{\gamma_{i,\pi_i(n)}f_i}\big(
x_{i,\pi_i(n)}-\gamma_{i,\pi_i(n)}\big(l_{i,n}^*
+\nabla\varphi_i(x_{i,\pi_i(n)})\big)\big);\\
a_{i,n}^*=\gamma_{i,\pi_i(n)}^{-1}(x_{i,\pi_i(n)}
-a_{i,n})-l_{i,n}^*;\\
\xi_{i,n}=\|a_{i,n}-x_{i,\pi_i(n)}\|^2;
\end{array}
\right.\\
\text{for every}\;i\in I\smallsetminus I_n\\
\left\lfloor
\begin{array}{l}
a_{i,n}=a_{i,n-1};\;a_{i,n}^*=a_{i,n-1}^*;\;
\xi_{i,n}=\xi_{i,n-1};\\
\end{array}
\right.\\
\text{for every}\;k\in K_n\\
\left\lfloor
\begin{array}{l}
b_{k,n}=\prox_{\mu_{k,\omega_k(n)}g_k}\big(y_{k,\omega_k(n)}
+\mu_{k,\omega_k(n)}\big(v_{k,\omega_k(n)}^*
-\nabla\psi_k(y_{k,\omega_k(n)})\big)\big);\\
d_{k,n}=\prox_{\nu_{k,\omega_k(n)}h_k}\big(z_{k,\omega_k(n)}
+\nu_{k,\omega_k(n)}v_{k,\omega_k(n)}^*\big);\\
e_{k,n}^*=\sigma_{k,\omega_k(n)}\big(
\sum_{i\in I}L_{ki}x_{i,\omega_k(n)}
-y_{k,\omega_k(n)}-z_{k,\omega_k(n)}\big)+v_{k,\omega_k(n)}^*
;\\
q_{k,n}^*=\mu_{k,\omega_k(n)}^{-1}(y_{k,\omega_k(n)}-b_{k,n})
+v_{k,\omega_k(n)}^*-e_{k,n}^*;\\
t_{k,n}^*=\nu_{k,\omega_k(n)}^{-1}(z_{k,\omega_k(n)}-d_{k,n})
+v_{k,\omega_k(n)}^*-e_{k,n}^*;\\
\eta_{k,n}=\|b_{k,n}-y_{k,\omega_k(n)}\|^2
+\|d_{k,n}-z_{k,\omega_k(n)}\|^2;\\
e_{k,n}=b_{k,n}+d_{k,n}-\sum_{i\in I}L_{ki}a_{i,n};
\end{array}
\right.\\
\text{for every}\;k\in K\smallsetminus K_n\\
\left\lfloor
\begin{array}{l}
b_{k,n}=b_{k,n-1};\;
d_{k,n}=d_{k,n-1};\;
e_{k,n}^*=e_{k,n-1}^*;\;
q_{k,n}^*=q_{k,n-1}^*;\;
t_{k,n}^*=t_{k,n-1}^*;\\
\eta_{k,n}=\eta_{k,n-1};\;
e_{k,n}=b_{k,n}+d_{k,n}-\sum_{i\in I}L_{ki}a_{i,n};
\end{array}
\right.\\
\text{for every}\;i\in I\\
\left\lfloor
\begin{array}{l}
p_{i,n}^*=a_{i,n}^*
+\nabla_{\!i}\,\Theta(\boldsymbol{a}_n)
+\sum_{k\in K}L_{ki}^*e_{k,n}^*;
\end{array}
\right.\\
\begin{aligned}
\Delta_n&=\textstyle
{-}(4\alpha)^{-1}\big(\sum_{i\in I}\xi_{i,n}
+\sum_{k\in K}\eta_{k,n}\big)
+\sum_{i\in I}\scal{x_{i,n}-a_{i,n}}{p_{i,n}^*}\\
&\textstyle
\quad\;+\sum_{k\in K}\big(\scal{y_{k,n}-b_{k,n}}{q_{k,n}^*}
+\scal{z_{k,n}-d_{k,n}}{t_{k,n}^*}
+\scal{e_{k,n}}{v_{k,n}^*-e_{k,n}^*}\big);
\end{aligned}\\
\text{if}\;\Delta_n>0\\
\left\lfloor
\begin{array}{l}
\theta_n=\lambda_n\Delta_n/
\big(\sum_{i\in I}\|p_{i,n}^*\|^2+\sum_{k\in K}\big(
\|q_{k,n}^*\|^2+\|t_{k,n}^*\|^2+\|e_{k,n}\|^2\big)\big);\\
\text{for every}\;i\in I\\
\left\lfloor
\begin{array}{l}
x_{i,n+1}=x_{i,n}-\theta_np_{i,n}^*;
\end{array}
\right.\\
\text{for every}\;k\in K\\
\left\lfloor
\begin{array}{l}
y_{k,n+1}=y_{k,n}-\theta_nq_{k,n}^*;\;
z_{k,n+1}=z_{k,n}-\theta_nt_{k,n}^*;\;
v_{k,n+1}^*=v_{k,n}^*-\theta_ne_{k,n};
\end{array}
\right.\\[1mm]
\end{array}
\right.\\
\text{else}\\
\left\lfloor
\begin{array}{l}
\text{for every}\;i\in I\\
\left\lfloor
\begin{array}{l}
x_{i,n+1}=x_{i,n};
\end{array}
\right.\\
\text{for every}\;k\in K\\
\left\lfloor
\begin{array}{l}
y_{k,n+1}=y_{k,n};\;z_{k,n+1}=z_{k,n};\;v_{k,n+1}^*=v_{k,n}^*.
\end{array}
\right.\\[1mm]
\end{array}
\right.\\[10mm]
\end{array}
\right.
\end{array}
\end{equation}
\end{algorithm}

\begin{corollary}
\label{c:12}
Consider the setting of Algorithm~\ref{algo:2}. Suppose that
\begin{equation}
\label{e:9069}
(\forall k\in K)\quad\epi(g_k+\psi_k)+\epi h_k\;\:
\text{is closed}
\end{equation}
and that Problem~\ref{prob:2} admits a Kuhn--Tucker point, that
is, there exist $\widetilde{\boldsymbol{x}}\in\HHH$
and $\widetilde{\boldsymbol{v}}^*\in\GGG$ such that
\begin{equation}
\label{e:1484}
(\forall i\in I)(\forall k\in K)\quad
\begin{cases}
{-}\sum_{j\in K}L_{ji}^*\widetilde{v}_j^*\in
\partial f_i(\widetilde{x}_i)+\nabla\varphi_i(\widetilde{x}_i)
+\nabla_{\!i}\,\Theta(\widetilde{\boldsymbol{x}})\\
\sum_{j\in I}L_{kj}\widetilde{x}_j\in
\partial\big(g_k^*\infconv\psi_k^*\big)(\widetilde{v}_k^*)
+\partial h_k^*(\widetilde{v}_k^*).
\end{cases}
\end{equation}
Then there exists a solution $\overline{\boldsymbol{x}}$ to 
\eqref{e:min-p} such that, for every $i\in I$,
$x_{i,n}\weakly\overline{x}_i$ and
$a_{i,n}\weakly\overline{x}_i$.
\end{corollary}
\begin{proof}
Set
\begin{equation}
\label{e:3906a}
\begin{cases}
(\forall i\in I)\;\; A_i=\partial f_i,\;\:C_i=\nabla\varphi_i,
\;\:\text{and}\;\:R_i=\nabla_{\!i}\,\Theta\\
(\forall k\in K)\;\; B_k^{\MM}=\partial g_k,\;\:
B_k^{\CC}=\nabla\psi_k,\;\:\text{and}\;\:
D_k^{\MM}=\partial h_k.
\end{cases}
\end{equation}
First, \cite[Theorem~20.25]{Livre1} asserts that
the operators $(A_i)_{i\in I}$, $(B_k^{\MM})_{k\in K}$, and
$(D_k^{\MM})_{k\in K}$ are maximally monotone.
Second, it follows from \cite[Corollary~18.17]{Livre1} 
that, for every $i\in I$, $C_i$ is $\alpha_i$-cocoercive
and, for every $k\in K$, $B_k^{\CC}$ is $\beta_k$-cocoercive.
Third, in view of \eqref{e:3906a} and
\cite[Proposition~17.7]{Livre1}, $\boldsymbol{R}=\nabla\Theta$ is
monotone and $\chi$-Lipschitzian. Now consider the problem
\begin{multline}
\label{e:2020p}
\text{find}\;\:\overline{\boldsymbol{x}}\in\HHH\;\:
\text{such that}\\
(\forall i\in I)\;\; 0\in
A_i\overline{x}_i+C_i\overline{x}_i
+R_i\overline{\boldsymbol{x}}
+\Sum_{k\in K}L_{ki}^*
\Bigg(\Big(\big(B_k^{\MM}+B_k^{\CC}\big)\infconv
D_k^{\MM}\Big)
\Bigg(\Sum_{j\in I}L_{kj}\overline{x}_j\Bigg)\Bigg)
\end{multline}
together with its dual
\begin{multline}
\label{e:2020d}
\text{find}\;\:\overline{\boldsymbol{v}}^*\in\GGG
\;\:\text{such that}\\
\big(\exi\boldsymbol{x}\in\HHH\big)
(\forall i\in I)(\forall k\in K)\;\;
\begin{cases}
{-}\Sum_{j\in K}L_{ji}^*\overline{v}_j^*\in
A_ix_i+C_ix_i+R_i\boldsymbol{x}\\
\overline{v}_k^*\in
\Big(\big(B_k^{\MM}+B_k^{\CC}\big)\infconv D_k^{\MM}\Big)
\Bigg(\Sum_{j\in I}L_{kj}x_j\Bigg).
\end{cases}
\end{multline}
Denote by $\mathscr{P}$ and $\mathscr{D}$ the sets of
solutions to \eqref{e:2020p} and \eqref{e:2020d}, respectively.
We observe that, by \eqref{e:3906a}
and \cite[Example~23.3]{Livre1},
Algorithm~\ref{algo:2} is an application of
Algorithm~\ref{algo:1} to the primal-dual problem
\eqref{e:2020p}--\eqref{e:2020d}.
Furthermore, it results from \eqref{e:1484} and 
Proposition~\ref{p:6}\ref{p:6iv}
that $\mathscr{D}\neq\emp$.
According to Theorem~\ref{t:1}\ref{t:1iii},
there exist $\overline{\boldsymbol{x}}\in\mathscr{P}$
and $\overline{\boldsymbol{v}}^*\in\mathscr{D}$
such that, for every $i\in I$ and every $k\in K$,
\begin{equation}
\label{e:3401}
x_{i,n}\weakly\overline{x}_i,\quad
a_{i,n}\weakly\overline{x}_i,
\quad\text{and}\quad
\begin{cases}
{-}\Sum_{j\in K}L_{ji}^*\overline{v}_j^*\in
A_i\overline{x}_i+C_i\overline{x}_i
+R_i\overline{\boldsymbol{x}}\\
\overline{v}_k^*\in
\Big(\big(B_k^{\MM}+B_k^{\CC}\big)\infconv D_k^{\MM}\Big)
\Bigg(\Sum_{j\in I}L_{kj}\overline{x}_j\Bigg).
\end{cases}
\end{equation}
It remains to show that $\overline{\boldsymbol{x}}$ solves
\eqref{e:min-p}. Define
\begin{equation}
\begin{cases}
\label{e:3906}
\boldsymbol{f}=\bigoplus_{i\in I}f_i,
\;\:\boldsymbol{\varphi}=\bigoplus_{i\in I}\varphi_i,
\;\:\boldsymbol{g}=\bigoplus_{k\in K}g_k,
\;\:\boldsymbol{h}=\bigoplus_{k\in K}h_k,\;\:
\text{and}\;\:\boldsymbol{\psi}=\bigoplus_{k\in K}\psi_k\\
\boldsymbol{L}\colon\HHH\to\GGG\colon\boldsymbol{x}\mapsto\big(
\sum_{i\in I}L_{ki}x_i\big)_{k\in K}.
\end{cases}
\end{equation}
We deduce from \cite[Theorem~15.3]{Livre1} that
$(\forall k\in K)$ $(g_k+\psi_k)^*=g_k^*\einfconv\psi_k^*$.
In turn, \eqref{e:1484} implies that
\begin{equation}
\label{e:6588}
(\forall k\in K)\quad
\emp\neq\dom\big(g_k^*\einfconv\psi_k^*\big)\cap\dom h_k^*
=\dom(g_k+\psi_k)^*\cap\dom h_k^*.
\end{equation}
On the other hand, since the sets
$(\epi(g_k+\psi_k)+\epi h_k)_{k \in K}$ are convex,
it follows from \eqref{e:9069} and
\cite[Theorem~3.34]{Livre1} that they are weakly closed.
Therefore, \cite[Theorem~1]{Bura05}
and the Fenchel--Moreau theorem \cite[Theorem~13.37]{Livre1}
imply that
\begin{equation}
\label{e:6589}
(\forall k\in K)\quad
\big((g_k+\psi_k)^*+h_k^*\big)^*
=(g_k+\psi_k)^{**}\einfconv h_k^{**}
=(g_k+\psi_k)\einfconv h_k.
\end{equation}
Hence, we derive from \eqref{e:3906a},
\cite[Corollaries~16.48(iii) and 16.30]{Livre1},
\eqref{e:6589}, and \cite[Proposition~16.42]{Livre1} that
\begin{align}
\label{e:1807}
(\forall k\in K)\quad
\big(B_k^{\MM}+B_k^{\CC}\big)\infconv D_k^{\MM}
&=\big(\partial g_k+\nabla\psi_k\big)\infconv(\partial h_k)
\nonumber\\
&=\Big(\big(\partial(g_k+\psi_k)\big)^{-1}+
(\partial h_k)^{-1}\Big)^{-1}
\nonumber\\
&=\big(\partial(g_k+\psi_k)^*+\partial h_k^*\big)^{-1}
\nonumber\\
&=\Big(\partial\big((g_k+\psi_k)^*+h_k^*\big)\Big)^{-1}
\nonumber\\
&=\partial\big((g_k+\psi_k)^*+h_k^*\big)^*
\nonumber\\
&=\partial\big((g_k+\psi_k)\einfconv h_k\big).
\end{align}
Since it results from \eqref{e:3906} and \eqref{e:6589} that
\begin{equation}
\label{e:5728}
(\boldsymbol{g}+\boldsymbol{\psi})\infconv\boldsymbol{h}
=(\boldsymbol{g}+\boldsymbol{\psi})\einfconv\boldsymbol{h}
=\bigoplus_{k\in K}\big((g_k+\psi_k)\einfconv h_k\big),
\end{equation}
we deduce from \cite[Proposition~16.9]{Livre1} and
\eqref{e:1807} that
\begin{equation}
\partial\big((\boldsymbol{g}
+\boldsymbol{\psi})\einfconv\boldsymbol{h}\big)
=\bigtimes_{k\in K}\partial\big((g_k+\psi_k)\einfconv h_k\big)
=\bigtimes_{k\in K}\big((B_k^{\MM}+B_k^{\CC})
\infconv D_k^{\MM}\big).
\end{equation}
It thus follows from \eqref{e:3401} and \eqref{e:3906} that
$\overline{\boldsymbol{v}}^*\in
\partial((\boldsymbol{g}+\boldsymbol{\psi})\einfconv\boldsymbol{h})
(\boldsymbol{L}\overline{\boldsymbol{x}})$.
On the other hand, since
$\boldsymbol{L}^*\colon\GGG\to\HHH\colon
\boldsymbol{v}^*\mapsto(\sum_{k\in K}L^*_{ki}v_k^*)_{i\in I}$,
we infer from \eqref{e:3401}, \eqref{e:3906a}, \eqref{e:3906},
and \cite[Proposition~16.9]{Livre1} that
${-}\boldsymbol{L}^*\overline{\boldsymbol{v}}^*
\in(C_i\overline{x}_i)_{i\in I}
+\boldsymbol{R}\overline{\boldsymbol{x}}
+\bigtimes_{i\in I}A_i\overline{x}_i
=\nabla\boldsymbol{\varphi}(\overline{\boldsymbol{x}})
+\nabla\Theta(\overline{\boldsymbol{x}})
+\partial\boldsymbol{f}(\overline{\boldsymbol{x}})$.
Hence, we invoke \cite[Proposition~16.6(ii)]{Livre1} to obtain
\begin{align}
\label{e:fu4}
\boldsymbol{0}
&\in\partial\boldsymbol{f}(\overline{\boldsymbol{x}})
+\nabla\boldsymbol{\varphi}(\overline{\boldsymbol{x}})
+\nabla\Theta(\overline{\boldsymbol{x}})
+\boldsymbol{L}^*\overline{\boldsymbol{v}}^*
\nonumber\\
&\subset
\partial\boldsymbol{f}(\overline{\boldsymbol{x}})
+\nabla\boldsymbol{\varphi}(\overline{\boldsymbol{x}})
+\nabla\Theta(\overline{\boldsymbol{x}})
+\boldsymbol{L}^*\Big(
\partial\big((\boldsymbol{g}
+\boldsymbol{\psi})\einfconv\boldsymbol{h}\big)
(\boldsymbol{L}\overline{\boldsymbol{x}})\Big)
\nonumber\\
&\subset
\partial\Big(\boldsymbol{f}+\boldsymbol{\varphi}+\Theta
+\big((\boldsymbol{g}
+\boldsymbol{\psi})\einfconv\boldsymbol{h}\big)\circ
\boldsymbol{L}\Big)(\overline{\boldsymbol{x}}).
\end{align}
However, thanks to \eqref{e:3906} and \eqref{e:5728},
\eqref{e:min-p} is equivalent to
\begin{equation}
\minimize{\boldsymbol{x}\in\HHH}{\boldsymbol{f}(\boldsymbol{x})
+\boldsymbol{\varphi}(\boldsymbol{x})+\Theta(\boldsymbol{x})
+\big((\boldsymbol{g}+\boldsymbol{\psi})
\einfconv\boldsymbol{h}\big)
(\boldsymbol{L}\boldsymbol{x})}.
\end{equation}
Consequently, in view of Fermat's rule \cite[Theorem~16.3]{Livre1},
\eqref{e:fu4} implies that $\overline{\boldsymbol{x}}$ solves
\eqref{e:min-p}.
\end{proof}

\begin{remark}
\label{r:fz79}
In \cite{Jmiv11}, multicomponent image recovery problems were
approached by applying the forward-backward
and the Douglas--Rachford algorithms in a product space. Using
Corollary~\ref{c:12}, we can now solve these problems with
asynchronous block-iterative algorithms and more sophisticated
formulations. For instance, the standard total variation
loss used in \cite{Jmiv11} can be replaced by the $p$th
order Huber total variation penalty of \cite{Hint06}, which turns
out to involve an infimal convolution. 
\end{remark}

To conclude, we provide some scenarios in which condition
\eqref{e:9069} is satisfied.

\begin{proposition}
\label{p:95}
Consider the setting of Problem~\ref{prob:2}.
Suppose that there exist
$\widetilde{\boldsymbol{x}}\in\HHH$
and $\widetilde{\boldsymbol{v}}^*\in\GGG$ such that
\begin{equation}
\label{e:1489}
(\forall i\in I)(\forall k\in K)\quad
\begin{cases}
{-}\sum_{j\in K}L_{ji}^*\widetilde{v}_j^*\in
\partial f_i(\widetilde{x}_i)+\nabla\varphi_i(\widetilde{x}_i)
+\nabla_{\!i}\,\Theta(\widetilde{\boldsymbol{x}})\\
\sum_{j\in I}L_{kj}\widetilde{x}_j\in
\partial\big(g_k^*\infconv\psi_k^*\big)(\widetilde{v}_k^*)
+\partial h_k^*(\widetilde{v}_k^*)
\end{cases}
\end{equation}
and that, for every $k\in K$, one of the following is satisfied:
\begin{enumerate}[label={\rm[\alph*]}]
\item
\label{p:95a}
$0\in\sri(\dom g_k^*+\dom\psi_k^*-\dom h_k^*)$.
\item
\label{p:95b}
$\GG_k$ is finite-dimensional, $h_k$ is polyhedral,
and $\dom h_k^*\cap\reli\dom(g_k+\psi_k)^*\neq\emp$.
\item
\label{p:95c}
$\GG_k$ is finite-dimensional, $g_k$ and $h_k$ are polyhedral,
and $\psi_k=0$.
\end{enumerate}
Then, for every $k\in K$, $\epi(g_k+\psi_k)+\epi h_k$
is closed.
\end{proposition}
\begin{proof}
Let $k\in K$. Since $\dom\psi_k=\GG_k$,
\cite[Theorem~15.3]{Livre1} yields
\begin{equation}
\label{e:8960}
(g_k+\psi_k)^*=g_k^*\einfconv\psi_k^*.
\end{equation}
Therefore, \eqref{e:1489} implies that
\begin{equation}
\label{e:6905}
\emp\neq\dom\big(g_k^*\einfconv\psi_k^*\big)\cap\dom h_k^*
=\dom(g_k+\psi_k)^*\cap\dom h_k^*.
\end{equation}
In view of \eqref{e:6905},
\cite[Theorem~1]{Bura05}, and \cite[Theorem~3.34]{Livre1},
it suffices to show that
$((g_k+\psi_k)^*+h_k^*)^*=(g_k+\psi_k)^{**}\einfconv h_k^{**}$.

\ref{p:95a}:
We deduce from \cite[Proposition~12.6(ii)]{Livre1}
and \eqref{e:8960} that
$0\in\sri(\dom(g_k^*\einfconv\psi_k^*)-\dom h_k^*)
=\sri(\dom(g_k+\psi_k)^*-\dom h_k^*)$.
In turn, \cite[Theorem~15.3]{Livre1} gives
$((g_k+\psi_k)^*+h_k^*)^*=(g_k+\psi_k)^{**}\einfconv h_k^{**}$.

\ref{p:95b}:
Since \cite[Theorem~19.2]{Rock70}
asserts that $h_k^*$ is polyhedral,
we infer from \cite[Theorem~20.1]{Rock70} that
$((g_k+\psi_k)^*+h_k^*)^*=(g_k+\psi_k)^{**}\einfconv h_k^{**}$.

\ref{p:95c}:
Since $g_k^*$ and $h_k^*$ are polyhedral by
\cite[Theorem~19.2]{Rock70},
it follows from \eqref{e:6905} and \cite[Theorem~20.1]{Rock70}
that $(g_k^*+h_k^*)^*=g_k^{**}\einfconv h_k^{**}$.
\end{proof}

\appendix
\section{Appendix}
\label{sec:A}

In this section, $\KK$ is a real Hilbert space.

\begin{lemma}
\label{l:3945}
Let $A\colon\KK\to 2^{\KK}$ be maximally monotone,
let $(x_n)_{n\in\NN}$ be a bounded sequence in $\KK$,
and let $(\gamma_n)_{n\in\NN}$ be a bounded sequence in $\RPP$.
Then $(J_{\gamma_nA}x_n)_{n\in\NN}$ is bounded.
\end{lemma}
\begin{proof}
Fix $x\in\KK$. Using the triangle inequality,
the nonexpansiveness of $(J_{\gamma_nA})_{n\in\NN}$,
and \cite[Proposition~23.31(iii)]{Livre1}, we obtain
$(\forall n\in\NN)$
$\|J_{\gamma_nA}x_n-J_Ax\|
\leq\|J_{\gamma_nA}x_n-J_{\gamma_nA}x\|+\|J_{\gamma_nA}x-J_Ax\|
\leq\|x_n-x\|+|1-\gamma_n|\,\|J_Ax-x\|
\leq\|x\|+\sup_{m\in\NN}\|x_m\|+(1+\sup_{m\in\NN}\gamma_m)
\|J_Ax-x\|$.
\end{proof}

\begin{lemma}
\label{l:2156}
Let $\alpha\in\RP$,
let $A\colon\KK\to\KK$ be $\alpha$-Lipschitzian,
let $\sigma\in\RPP$,
and let $\gamma\in\left]0,1/(\alpha+\sigma)\right]$.
Then $\gamma^{-1}\Id-A$ is $\sigma$-strongly monotone.
\end{lemma}
\begin{proof}
By Cauchy--Schwarz, 
\begin{align}
&(\forall x\in\KK)(\forall y\in\KK)\quad
\sscal{x-y}{\big(\gamma^{-1}\Id-A\big)x
-\big(\gamma^{-1}\Id-A\big)y}
\nonumber\\
&\hspace{40mm}
=\gamma^{-1}\|x-y\|^2-\scal{x-y}{Ax-Ay}
\nonumber\\
&\hspace{40mm}
\geq(\alpha+\sigma)\|x-y\|^2-\|x-y\|\,\|Ax-Ay\|
\nonumber\\
&\hspace{40mm}
\geq(\alpha+\sigma)\|x-y\|^2-\alpha\|x-y\|^2
\nonumber\\
&\hspace{40mm}
=\sigma\|x-y\|^2,
\end{align}
which proves the assertion.
\end{proof}

\begin{lemma}
\label{l:5368}
Let $I$ be a nonempty finite set,
let $(I_n)_{n\in\NN}$ be nonempty subsets of $I$,
let $P\in\NN$,
and let $(x_n)_{n\in\NN}$ be a sequence in $\KK$.
Suppose that
$\sum_{n\in\NN}\|x_{n+1}-x_n\|^2<\pinf$,
$I_0=I$, and $(\forall n\in\NN)$
$\bigcup_{j=n}^{n+P}I_j=I$.
Furthermore, let $T\in\NN$, let $i\in I$, and let
$(\pi_i(n))_{n\in\NN}$ be a sequence in $\NN$ such that 
$(\forall n\in\NN)$ $n-T\leq\pi_i(n)\leq n$.
For every $n\in\NN$, set
$\overline{\vartheta}_i(n)=
\max\menge{j\in\NN}{j\leq n\;\text{and}\;i\in I_j}$
and $\vartheta_i(n)=\pi_i(\overline{\vartheta}_i(n))$.
Then $x_{\vartheta_i(n)}-x_n\to 0$.
\end{lemma}
\begin{proof}
For every integer $n\geq P$,
since $i\in\bigcup_{j=n-P}^nI_j$, we have 
$n\leq\overline{\vartheta}_i(n)+P
\leq\pi_i(\overline{\vartheta}_i(n))+P+T=\vartheta_i(n)+P+T$.
Hence $\vartheta_i(n)\to\pinf$ and therefore
$\sum_{j=\vartheta_i(n)}^{\vartheta_i(n)+P+T}
\|x_{j+1}-x_j\|^2\to 0$. 
However, it results from our assumption that $(\forall n\in\NN)$
$\vartheta_i(n)=\pi_i(\overline{\vartheta}_i(n))
\leq\overline{\vartheta}_i(n)\leq n$.
We thus deduce from the triangle
and Cauchy--Schwarz inequalities that
\begin{equation}
\|x_n-x_{\vartheta_i(n)}\|^2
\leq\Bigg|\sum_{j=\vartheta_i(n)}^{\vartheta_i(n)+P+T}
\|x_{j+1}-x_j\|\Bigg|^2
\leq(P+T+1)\sum_{j=\vartheta_i(n)}^{\vartheta_i(n)+P+T}
\|x_{j+1}-x_j\|^2
\to 0.
\end{equation}
Consequently, $x_{\vartheta_i(n)}-x_n\to 0$.
\end{proof}

\begin{lemma}[{\cite{Eoop01}}]
\label{l:yeu}
Let $Z$ be a nonempty closed convex subset of $\KK$,
$x_0\in\KK$, and $\varepsilon\in\zeroun$. Suppose that
\begin{equation}
\label{e:1972}
\begin{array}{l}
\text{for}\;n=0,1,\ldots\\
\left\lfloor
\begin{array}{l}
t_n^*\in\KK\;\text{and}\;\eta_n\in\RR\;\text{satisfy}\;
Z\subset H_n=\menge{x\in\KK}{\scal{x}{t_n^*}\leq\eta_n};\\
\Delta_n=\scal{x_n}{t_n^*}-\eta_n;\\
\text{if}\;\Delta_n>0\\
\left\lfloor
\begin{array}{l}
\lambda_n\in\left[\varepsilon,2-\varepsilon\right];\\
x_{n+1}=x_n-(\lambda_n\Delta_n/\|t_n^*\|^2)\,t_n^*;
\end{array}
\right.\\
\text{else}\\
\left\lfloor
\begin{array}{l}
x_{n+1}=x_n.
\end{array}
\right.\\[1mm]
\end{array}
\right.
\end{array}
\end{equation}
Then the following hold:
\begin{enumerate}
\item
\label{l:yeui}
$(\forall z\in Z)(\forall n\in\NN)$ $\|x_{n+1}-z\|
\leq\|x_n-z\|$.
\item
\label{l:yeuii}
$\sum_{n\in\NN}\|x_{n+1}-x_n\|^2<\pinf$.
\item
\label{l:yeuiii}
Suppose that, for every $x\in\KK$ and every strictly
increasing sequence $(k_n)_{n\in\NN}$ in $\NN$,
$x_{k_n}\weakly x$ $\Rightarrow$ $x\in Z$.
Then $(x_n)_{n\in\NN}$ converges weakly to a point in $Z$.
\end{enumerate}
\end{lemma}

We now revisit ideas found in \cite{Moor01,Sico00}
in a format that is be more suited for our purposes.

\begin{lemma}
\label{l:manh}
Let $Z$ be a nonempty closed convex subset of $\KK$ and let
$x_0\in\KK$. Suppose that
\begin{equation}
\label{e:9237}
\begin{array}{l}
\text{for}\;n=0,1,\ldots\\
\left\lfloor
\begin{array}{l}
t_n^*\in\KK\;\text{and}\;\eta_n\in\RR\;\text{satisfy}\;
Z\subset H_n=\menge{x\in\KK}{\scal{x}{t_n^*}\leq\eta_n};\\
\Delta_n=\scal{x_n}{t_n^*}-\eta_n;\\
\text{if}\;\Delta_n>0\\
\left\lfloor
\begin{array}{l}
\tau_n=\|t_n^*\|^2;\;
\varsigma_n=\|x_0-x_n\|^2;\;
\chi_n=\scal{x_0-x_n}{t_n^*};\;
\rho_n=\tau_n\varsigma_n-\chi_n^2;\\
\text{if}\;\rho_n=0\\
\left\lfloor
\begin{array}{l}
\kappa_n=1;\;\lambda_n=\Delta_n/\tau_n;
\end{array}
\right.\\
\text{else}\\
\left\lfloor
\begin{array}{l}
\text{if}\;\chi_n\Delta_n\geq\rho_n\\
\left\lfloor
\begin{array}{l}
\kappa_n=0;\;
\lambda_n=\big(\Delta_n+\chi_n\big)/\tau_n;
\end{array}
\right.\\
\text{else}\\
\left\lfloor
\begin{array}{l}
\kappa_n=1-\chi_n\Delta_n/\rho_n;\;
\lambda_n=\varsigma_n\Delta_n/\rho_n;
\end{array}
\right.\\[1.5mm]
\end{array}
\right.\\
x_{n+1}=(1-\kappa_n)x_0+\kappa_nx_n-\lambda_nt_n^*;
\end{array}
\right.\\
\text{else}\\
\left\lfloor
\begin{array}{l}
x_{n+1}=x_n.
\end{array}
\right.\\[1mm]
\end{array}
\right.
\end{array}
\end{equation}
Then the following hold:
\begin{enumerate}
\item
\label{l:manhi}
$(\forall n\in\NN)$ $\|x_n-x_0\|\leq\|x_{n+1}-x_0\|
\leq\|\proj_Zx_0-x_0\|$.
\item
\label{l:manhii}
$\sum_{n\in\NN}\|x_{n+1}-x_n\|^2<\pinf$ and
$\sum_{n\in\NN}\|\proj_{H_n}x_n-x_n\|^2<\pinf$.
\item
\label{l:manhiii}
Suppose that, for every $x\in\KK$ and every
strictly increasing sequence $(k_n)_{n\in\NN}$ in $\NN$,
$x_{k_n}\weakly x$ $\Rightarrow$ $x\in Z$.
Then $x_n\to\proj_Zx_0$.
\end{enumerate}
\end{lemma}
\begin{proof}
Define
$(\forall n\in\NN)$
$G_n=\menge{x\in\KK}{\scal{x-x_n}{x_0-x_n}\leq 0}$.
Then, by virtue of \eqref{e:9237},
\begin{equation}
\label{e:5852}
(\forall n\in\NN)\quad
x_n=\proj_{G_n}x_0\quad\text{and}\quad
\big[\;\Delta_n>0\;\;\Rightarrow\;\;
\proj_{H_n}x_n=x_n-\big(\Delta_n/\|t_n^*\|^2\big)\,t_n^*\;\big].
\end{equation}
Let us establish that
\begin{equation}
\label{e:9401}
(\forall n\in\NN)\quad Z\subset H_n\cap G_n\quad\text{and}\quad
x_{n+1}=\proj_{H_n\cap G_n}x_0.
\end{equation}
Since $G_0=\KK$, \eqref{e:9237} yields $Z\subset H_0=H_0\cap G_0$.
Hence, we derive from
\eqref{e:5852} and \eqref{e:9237} that $\Delta_0>0$ $\Rightarrow$
$[\,\proj_{H_0}x_0=x_0-(\Delta_0/\tau_0)\,t_0^*$ and $\rho_0=0\,]$
$\Rightarrow$
$[\,\proj_{H_0}x_0=x_0-(\Delta_0/\tau_0)\,t_0^*$,
$\kappa_0=1$, and $\lambda_0=\Delta_0/\tau_0\,]$
$\Rightarrow$ $x_1=x_0-(\Delta_0/\tau_0)\,t_0^*
=\proj_{H_0}x_0=\proj_{H_0\cap G_0}x_0$.
On the other hand, $\Delta_0\leq 0$ $\Rightarrow$
$x_1=x_0\in H_0=H_0\cap G_0$ $\Rightarrow$
$x_1=\proj_{H_0\cap G_0}x_0$.
Now assume that, for some integer $n\geq 1$,
$Z\subset H_{n-1}\cap G_{n-1}$ and
$x_n=\proj_{H_{n-1}\cap G_{n-1}}x_0$.
Then, according to \cite[Theorem~3.16]{Livre1},
$Z\subset H_{n-1}\cap G_{n-1}\subset\menge{x\in\KK}{
\scal{x-x_n}{x_0-x_n}\leq 0}=G_n$.
In turn, \eqref{e:9237} entails that $Z\subset H_n\cap G_n$.
Next, it follows from \eqref{e:9237}, \eqref{e:5852}, and
\cite[Proposition~29.5]{Livre1} that
$\Delta_n\leq 0$ $\Rightarrow$
$[\,x_{n+1}=x_n$ and $\proj_{G_n}x_0=x_n\in H_n\,]$
$\Rightarrow$ $x_{n+1}=\proj_{G_n}x_0=\proj_{H_n\cap G_n}x_0$.
To complete the induction argument,
it remains to verify that $\Delta_n>0$ $\Rightarrow$
$x_{n+1}=\proj_{H_n\cap G_n}x_0$.
Assume that $\Delta_n>0$ and set
\begin{equation}
\label{e:2488}
y_n=\proj_{H_n}x_n,\quad
\widetilde{\chi}_n=\scal{x_0-x_n}{x_n-y_n},\quad
\widetilde{\nu}_n=\|x_n-y_n\|^2,\quad\text{and}\quad
\widetilde{\rho}_n=
\varsigma_n\widetilde{\nu}_n-\widetilde{\chi}_n^2.
\end{equation}
Since $\Delta_n>0$, we have
$H_n=\menge{x\in\KK}{\scal{x-y_n}{x_n-y_n}\leq 0}$
and $y_n=x_n-\theta_nt_n^*$, where $\theta_n=\Delta_n/\tau_n>0$.
In turn, we infer from \eqref{e:2488} and \eqref{e:9237} that
\begin{equation}
\label{e:3445}
\widetilde{\chi}_n=\theta_n\chi_n,\quad
\widetilde{\nu}_n=\theta_n^2\tau_n=\theta_n\Delta_n,
\quad\text{and}\quad
\widetilde{\rho}_n=\theta_n^2\rho_n.
\end{equation}
Furthermore, \eqref{e:9237} and the Cauchy--Schwarz inequality
ensure that $\rho_n\geq 0$, which leads to two cases.
\begin{itemize}
\item
$\rho_n=0$:
On the one hand, \eqref{e:9237} asserts that
$x_{n+1}=x_n-(\Delta_n/\tau_n)\,t_n^*=y_n$.
On the other hand, \eqref{e:3445} yields 
$\widetilde{\rho}_n=0$ and, therefore, since $H_n\cap G_n\neq\emp$,
\cite[Corollary~29.25(ii)]{Livre1} yields
$\proj_{H_n\cap G_n}x_0=y_n$. Altogether, $x_{n+1}
=\proj_{H_n\cap G_n}x_0$.

\item
$\rho_n>0$:
By \eqref{e:3445}, $\widetilde{\rho}_n>0$.
First, suppose that
$\chi_n\Delta_n\geq\rho_n$.
It follows from \eqref{e:9237} that
$x_{n+1}=x_0-((\Delta_n+\chi_n)/\tau_n)\,t_n^*$
and from \eqref{e:3445} that
$\widetilde{\chi}_n\widetilde{\nu}_n
=\theta_n^2\chi_n\Delta_n\geq\theta_n^2\rho_n
=\widetilde{\rho}_n$. Thus
\cite[Corollary~29.25(ii)]{Livre1} and \eqref{e:3445} imply that
\begin{align}
\proj_{H_n\cap G_n}x_0
&=x_0+\bigg(1+\frac{\widetilde{\chi}_n}{\widetilde{\nu}_n}\bigg)
(y_n-x_n)
\nonumber\\
&=x_0-\bigg(1+\frac{\chi_n}{\theta_n\tau_n}\bigg)\theta_nt_n^*
\nonumber\\
&=x_0-\frac{\theta_n\tau_n+\chi_n}{\tau_n}\,t_n^*
\nonumber\\
&=x_0-\frac{\Delta_n+\chi_n}{\tau_n}\,t_n^*
\nonumber\\
&=x_{n+1}.
\end{align}
Now suppose that $\chi_n\Delta_n<\rho_n$.
Then $\widetilde{\chi}_n\widetilde{\nu}_n<
\widetilde{\rho}_n$ and hence it results from
\cite[Corollary~29.25(ii)]{Livre1}, \eqref{e:3445},
and \eqref{e:9237} that
\begin{align}
\proj_{H_n\cap G_n}x_0
&=x_n+\frac{\widetilde{\nu}_n}{\widetilde{\rho}_n}\Big(
\widetilde{\chi}_n(x_0-x_n)+\varsigma_n(y_n-x_n)\Big)
\nonumber\\
&=\frac{\widetilde{\chi}_n\widetilde{\nu}_n}{\widetilde{\rho}_n}
\,x_0+\bigg(1-\frac{\widetilde{\chi}_n\widetilde{\nu}_n}
{\widetilde{\rho}_n}\bigg)x_n
+\frac{\widetilde{\nu}_n\varsigma_n}{\widetilde{\rho}_n}(y_n-x_n)
\nonumber\\
&=\frac{\chi_n\Delta_n}{\rho_n}\,x_0+
\bigg(1-\frac{\chi_n\Delta_n}{\rho_n}\bigg)x_n
-\frac{\tau_n\varsigma_n}{\rho_n}\frac{\Delta_n}{\tau_n}\,t_n^*
\nonumber\\
&=x_{n+1}.
\end{align}
\end{itemize}

\ref{l:manhi}:
Let $n\in\NN$. We derive from \eqref{e:9401} that
$\|x_{n+1}-x_0\|=\|\proj_{H_n\cap G_n}x_0-x_0\|
\leq\|\proj_Zx_0-x_0\|$. On the other hand,
since $x_{n+1}\in G_n$ by virtue of \eqref{e:9401}, we have
\begin{align}
\label{e:1124}
\|x_n-x_0\|^2+\|x_{n+1}-x_n\|^2
&\leq\|x_n-x_0\|^2+\|x_{n+1}-x_n\|^2
+2\scal{x_{n+1}-x_n}{x_n-x_0}
\nonumber\\
&=\|x_{n+1}-x_0\|^2.
\end{align}

\ref{l:manhii}:
Let $N\in\NN$.
In view of \eqref{e:1124} and \ref{l:manhi},
$\sum_{n=0}^N\|x_{n+1}-x_n\|^2
\leq\sum_{n=0}^N(\|x_{n+1}-x_0\|^2-\|x_n-x_0\|^2)
=\|x_{N+1}-x_0\|^2\leq\|\proj_Zx_0-x_0\|^2$.
Therefore, $\sum_{n\in\NN}\|x_{n+1}-x_n\|^2<\pinf$.
However, for every $n\in\NN$, since \eqref{e:9401} asserts that
$x_{n+1}\in H_n$, we have
$\|\proj_{H_n}x_n-x_n\|\leq\|x_{n+1}-x_n\|$.
Thus $\sum_{n\in\NN}\|\proj_{H_n}x_n-x_n\|^2<\pinf$.

\ref{l:manhiii}:
It results from \ref{l:manhi} that $(x_n)_{n\in\NN}$ is bounded.
Now let $x\in\KK$, let $(k_n)_{n\in\NN}$ be a strictly
increasing sequence in $\NN$, and suppose that $x_{k_n}\weakly x$.
Using \cite[Lemma~2.42]{Livre1} and \ref{l:manhi},
we deduce that $\|x-x_0\|\leq\varliminf\|x_{k_n}-x_0\|
\leq\|\proj_Zx_0-x_0\|$. Thus, since it results from 
our assumption that $x\in Z$,
we have $x=\proj_Zx_0$, which implies that
$x_n\weakly\proj_Zx_0$ \cite[Lemma~2.46]{Livre1}.
In turn, since
$\varlimsup\|x_n-x_0\|\leq\|\proj_Zx_0-x_0\|$ by \ref{l:manhi},
\cite[Lemma~2.51(i)]{Livre1} forces $x_n\to\proj_Zx_0$.
\end{proof}

\end{document}